\documentclass[a4paper]{amsart}
\usepackage[latin1]{inputenc}
\usepackage[T1]{fontenc}
\usepackage[english]{babel}
\usepackage{amsmath,mathrsfs, geometry}
\usepackage{amscd}
\usepackage{graphicx}
\usepackage[all]{xy}
\usepackage{amsthm}
\usepackage{amssymb,url}
\usepackage[charter]{mathdesign}
\usepackage{bbm}
\usepackage{ae}
\usepackage{color}

\DeclareFontFamily{OT1}{pzc}{}
\DeclareFontShape{OT1}{pzc}{m}{it}{<-> s * [1.1] pzcmi7t}{}
\DeclareMathAlphabet{\mathpzc}{OT1}{pzc}{m}{it}

\newtheorem{thm}{Theorem}[subsection]
\newtheorem{thm*}{Theorem}[]
\newtheorem{thmapp}{Theorem}[section]
\newtheorem{cor}[thm]{Corollary}
\newtheorem{prop}[thm]{Proposition}
\newtheorem{propapp}[thmapp]{Proposition}
\newtheorem{prop*}[thm*]{Proposition}
\newtheorem{lem}[thm]{Lemma}
\newtheorem{lemapp}[thmapp]{Lemma}
\newtheorem{lem*}[thm*]{Lemma}

\theoremstyle{remark}
\newtheorem{remark}[thm]{Remark}
\newtheorem{remark*}[thm*]{Remark}
\newtheorem{remarkapp}[thmapp]{Remark}

\theoremstyle{definition}
\newtheorem{deef}[thm]{Definition}
\newtheorem{deef*}[thm*]{Definition}
\newtheorem{deefapp}[thmapp]{Definition}

\newcommand{\Ko}{\mathbb{K}}
\newcommand{\Bo}{\mathbb{B}}
\newcommand{\He}{\mathpzc{H}}
\newcommand{\minotimes}{\bar{\otimes}}
\newcommand{\hotimes}{\tilde{\otimes}}
\newcommand{\R}{\mathbbm{R}}
\newcommand{\C}{\mathbbm{C}}

\newcommand{\Z}{\mathbbm{Z}}
\newcommand{\Q}{\mathbbm{Q}}
\newcommand{\N}{\mathbbm{N}}

\newcommand{\id}{\, \mathrm{id}}
\newcommand{\rd}{\mathrm{d}}

\newcommand{\Fg}{\mathcal{F}}

\newcommand{\ind}{\mathrm{i} \mathrm{n} \mathrm{d} \,}

\newcommand{\Hom}{\mathrm{Hom}\;}

\newcommand{\Ext}{\mathrm{Ext}}

\renewcommand{\epsilon}{\varepsilon}
\renewcommand{\phi}{\varphi}

\newcommand{\cstar}{A}

\newcommand{\im}{\mathrm{i} \mathrm{m} \,}

\newcommand{\supp}{\mathrm{s} \mathrm{u} \mathrm{p} \mathrm{p}\,}
\newcommand{\e}{\mathrm{e}}
\newcommand{\tra}{\textnormal{Tr}}

\newcommand{\End}{\textnormal{End}}
\newcommand{\Dom}{\textnormal{Dom}\,}
\newcommand{\K}{\mathbb{K}}
\newcommand{\B}{\mathbb{B}}
\newcommand{\Lip}{\textnormal{Lip}}
\newcommand{\halfepsilon}{\delta}
\renewcommand{\theequation}{\arabic{section}.\arabic{equation}}
\newcommand{\cemptyset}{\circ}

\title[Finite summability on Cuntz-Krieger algebras]{Spectral triples and finite summability on Cuntz-Krieger algebras}
\author{Magnus Goffeng and Bram Mesland}
\thanks{email: 
\texttt{goffeng@math.uni-hannover.de}, \texttt{b.mesland@warwick.ac.uk}
}
\address{M. Goffeng\\ Insitut f\"{u}r Analysis, Leibniz Universit\"{a}t Hannover\\Welfengarten 1\\
30167 Hannover\\ Germany}

\address{B. Mesland\\ Department of Mathematics, University of Warwick\\Zeeman Building\\Coventry CV4 7AL\\ UK}

\begin{document}

\begin{abstract}
We produce a variety of odd bounded Fredholm modules and odd spectral triples on Cuntz-Krieger algebras by means of realizing these algebras as ``the algebra of functions on a non-commutative space" coming from a sub shift of finite type. We show that any odd $K$-homology class can be represented by such an odd bounded Fredholm module or odd spectral triple. The odd bounded Fredholm modules that are constructed are finitely summable. The spectral triples are $\theta$-summable, although their phases will already on the level of analytic $K$-cycles be finitely summable bounded Fredholm modules. Using the unbounded Kasparov product, we exhibit a family of generalized spectral triples, related to work of Bellissard-Pearson, possessing mildly unbounded commutators, whilst still giving well defined $K$-homology classes.
\end{abstract}

\maketitle

\vspace{-8mm}
\small
\tableofcontents

\Large
\section*{Introduction}
\normalsize

This paper is a study of how odd $K$-homology classes on Cuntz-Krieger algebras can be realized by explicit cycles; both by means of bounded Fredholm modules (also known as analytic $K$-cycles) and as unbounded Fredholm modules, e.g. spectral triples. We will use the Poincar\'e duality for Cuntz-Krieger algebras constructed by Kaminker-Putnam \cite{kaminkerputnam} to find explicit finitely summable Fredholm modules representing any odd $K$-homology class. This allows us to realize odd $K$-homology classes by means of abstract Toeplitz operators and the finite summability of the cycles is proved using the work of the first named author \cite{extensionsgoffeng}. 

The construction of unbounded representatives of these $K$-homology classes is more elaborate. We discuss the possibility of using Kasparov products of unbounded Fredholm modules for the fixed point algebra of the gauge action with a well studied unbounded bivariant cycle. Related constructions can be found in \cite{gabgrens}. In many cases it is difficult to understand cohomological properties of unbounded Fredholm modules on the fixed point algebra, they nevertheless exist in abundance due to \cite{chrivan}. However, in interesting cases such as the Cuntz algebra $O_N$ this will produce $K$-homologically trivial unbounded Fredholm modules. 

A more fruitful viewpoint comes from describing the Cuntz-Krieger algebra as the noncommutative quotient of the underlying subshift of finite type, via its groupoid model. This viewpoint is common to noncommutative geometry. The maximal abelian subalgebra corresponding to the unit space in the groupoid plays the r\^{o}le of the base space in a fibration. The unbounded Fredholm modules are then obtained by restricting an unbounded bivariant cycle to a ``fiber" over the unit space. The bivariant cycle is inspired both by the dynamics of the underlying subshift of finite type and the structures appearing in Kaminker-Putnam's Poincar\'e duality class. This uses the idea of multiplication by real valued functions defined on the groupoid to obtain regular operators as in \cite{Mes1}. Localizations of this bivariant cycle to the commutative base exhausts the odd $K$-homology of the Cuntz-Krieger algebra.

An explicit construction of the unbounded Kasparov product of this cycle with canonically defined spectral triples on the commutative base from \cite{BP} yields a generalization of the notion of unbounded Fredholm module, allowing for unbounded commutators. This generalization is compatible with $K$-homology. The Kasparov products are constructed using the operator space approach to connections initiated by the second named author in \cite{Mes} and developed further in \cite{BMS, KaLe}.
\newline

The problem that this work originates from can be formulated as follows. Whenever $B$ is a $C^*$-algebra and $x\in K^*(B)$ is a $K$-homology class, is it possible to find an explicit analytic $K$-cycle or unbounded Fredholm module representing $x$ with favourable analytic properties? Here we are mainly concerned with finite- and $\theta$-summability. We return to discuss this problem setting more precisely below. The Cuntz-Krieger algebras are interesting in this aspect because results of Connes \cite{Connestrace}, combined with the fact that (under weak assumptions) Cuntz-Krieger algebras admit no traces, imply that it is (under these weak assumptions) not possible to have a finitely summable unbounded Fredholm module on a Cuntz-Krieger algebra\footnote{Not even for the generalized notion of unbounded Fredholm modules alluded to above.}. In this paper we show that any odd $K$-homology class is represented by a finitely summable $K$-cycle. It should be mentioned that this interesting structure has been shown to appear also on the crossed product of boundary actions of a hyperbolic group \cite{EN2}. We believe that our constructions illuminate the differences between finite summability in the bounded and the unbounded models for $K$-homology.

In contrast to the obstructions to finite summability of unbounded Fredholm modules from \cite{Connestrace}, there is to our knowledge no analog for bounded Fredholm modules. Nor were we able to find an example in the literature of a $K$-homology class that can not be represented by an analytic $K$-cycle which is finitely summable on some dense sub-algebra. We provide such an example on a commutative $C^*$-algebra.

\subsection*{Preliminaries}
\label{prelsubsection}

Before entering into finite summability issues and the precise formulation of the results in this paper, we recall some concepts of noncommutative geometry. This paper discusses the noncommutative geometry of Cuntz-Krieger algebras from the point of view of Kasparov's $KK$-theory \cite{Kas1, Kas2}, and the unbounded formulation thereof due to Connes \cite{Connesbook} and Baaj-Julg \cite{BJ}. The central objects in Kasparov's approach to $KK$-theory are Fredholm modules. Fredholm modules come in two flavors, bounded and unbounded. The bounded Fredholm modules are sometimes referred to as analytic $K$-cycles.  

\begin{deef*} 
\label{bddfred}
Let $A$ be a $C^{*}$-algebra. A bounded \emph{even Fredholm module over $A$} is a triple $(\pi, \mathpzc{H}, F)$ consisting of 
\begin{enumerate}
\item a $\Z/2$-graded Hilbert space $\mathpzc{H}$ carrying an even $*$-representation $\pi:A\rightarrow \mathbb{B}(\mathpzc{H})$;
\item an odd operator $F\in\mathbb{B}(\mathpzc{H})$ with the property that, for all $a\in A$, $\pi(a)(F^{2}-1), \pi(a)(F-F^{*})$ and $[F,\pi(a)]$ are all compact operators.
\end{enumerate}
A triple $(\pi,\mathpzc{H},F)$ with the above properties save for the fact that the Hilbert space $\mathpzc{H}$ is graded, defines a bounded \emph{odd Fredholm module}. If in the above $\mathpzc{H}$ is replaced with a Hilbert $C^{*}$-module $\mathpzc{E}$ over a second $C^{*}$-algebra $B$,\footnote{ In which case $\mathbb{B}(\mathpzc{H})$ is replaced with $\End^{*}_{B}(\mathpzc{E})$ -- the $C^*$-algebra of adjointable $B$-linear operators on $\mathpzc{E}$, and the $C^*$-algebra of compact operators by the $C^*$-algebra of $B$-compact operators $\Ko_B(\mathpzc{E})$.} then $(\pi,\mathpzc{E},F)$ defines an $(A,B)$-\emph{Kasparov module}. 

\end{deef*}

By defining a suitable notion of homotopy, the set of homotopy classes of even Fredholm modules forms an abelian group $K^{0}(A)$, and the odd Fredholm modules are used to build an abelian group $K^{1}(A)$. The groups $K^0(A)$ and $K^1(A)$ are called the $K$-homology groups of $A$ combining into the $\Z/2\Z$-graded abelian group $K^*(A)=K^0(A)\oplus K^1(A)$. The $K$-homology groups are homotopy invariants of $A$ and encode index theoretic information. See \cite{AKH} for an excellent exposition of this theory.  Historically, the Fredholm picture of $K$-homology was conceived by Atiyah \cite{atiyahkhom} who introduced it to make the Atiyah-Singer index theorem into a functorial statement. It reached full maturity in the work of Kasparov \cite{Kas1}, where the groups $KK_{*}(A,B)=KK_0(A,B)\oplus KK_1(A,B)$ are defined similarly, as an abelian group of homotopy classes of $(A,B)$-Kasparov modules. This culminated in his proof of the Novikov conjecture for a large class of groups \cite{Kas2}.  For computational purposes, it is sometimes convenient to work with unbounded Fredholm modules.

\begin{deef*}
\label{unbddfred}
An unbounded even Fredholm module over a $C^*$-algebra $A$ consists of a triple $(\pi,\He,D)$ containing the data:
\begin{enumerate}
\item a $\Z/2$-graded Hilbert space $\He$ carrying an even $*$-representation $\pi:A\rightarrow \mathbb{B}(\He)$.
\item a selfadjoint odd operator $D$ with locally compact resolvents $\pi(a)(D\pm i)^{-1}\in \Ko(\He)$ for all $a\in A$, such that the $*$-algebra 
\[\mathrm{Lip}(\pi,\He,D):=\left\{a\in \,A\,:\; 
\begin{split}
&\pi(a)\Dom(D)\subseteq \Dom(D)\quad\textnormal{and}\\ 
&[D,\pi(a)]\quad\textnormal{extends to a bounded operator}
\end{split} \right\},
\]
is norm dense in $A$.
\end{enumerate}
A triple $(\pi,\mathpzc{H},D)$ with the above properties save for the fact that the Hilbert space $\mathpzc{H}$ is graded, defines an unbounded \emph{odd Fredholm module}. If $\pi$ is faithful and $\mathcal{A}\subseteq \pi(\mathrm{Lip}(\pi,\He,D))$ is dense in $\pi(A)$ the triple $(\mathcal{A},\He,D)$ is called an even (odd) spectral triple.  

If in the above $\mathpzc{H}$ is replaced with a Hilbert $C^*$-module $\mathpzc{E}$ over a second $C^{*}$-algebra $B$,  $\mathbb{B}(\mathpzc{H})$ with $\End^{*}_{B}(\mathpzc{E})$ -- the $C^*$-algebra of adjointable operators on $\mathpzc{E}$, and on the operator $D$ the further assumption that $D$ is \emph{regular} is added (for details see \cite{BJ}), then $(\pi,\mathpzc{E},D)$ defines an \emph{unbounded $KK$-cycle} for $(A,B)$.
\end{deef*}

An unbounded $KK$-cycle defines a Kasparov module by setting $F:=D(1+D^{2})^{-\frac{1}{2}}$, the \emph{bounded transform} of $D$. It should be noted that with any choice of bounded continuous function $\chi\in C_{b}(\R,\R)$, such that $\chi^2-1\in C_0(\R)$, we can associate a Kasparov module by setting $F_\chi:=\chi(D)$, producing a Kasparov module differing from that defined by $F$ only by a compact perturbation. 

For the special case of unbounded Fredholm modules, another way of doing this, which features prominently in the present work, is through the \emph{phase} of $D$; the phase is defined as $D|D|^{-1}$. Here $|D|^{-1}$ is defined to be $0$ on the kernel of $D$. The construction of the phase hinges on the fact that for unbounded Fredholm modules, the spectrum of $D$ is discrete, and there is a $\chi\in C^\infty(\R,\R)$ as above with $\chi'\in C^\infty_c(\R,\R)$ such that $F_\chi=D|D|^{-1}$. In the special case of unbounded Fredholm modules, a modification of $\chi$ on a compact subset of $\R$ affects the bounded Fredholm module by a mere finite rank perturbation. In particular, the associated $KK$-class does not depend on $\chi$. We will in this paper see several examples of how finer analytic properties depend on the choice $\chi$. \newline

The foundation of noncommutative geometry is built on the idea that the geometry of a ``noncommutative space" is encoded by a spectral triple on the ``algebra of functions", i.e. a $C^*$-algebra. Conformal geometry is encoded by a choice of a bounded Fredholm module. Homological algebra corresponds to $K$-theory and $K$-homology. These ideas were pioneered by Connes and many examples are to be found in \cite{Connesbook}. In the classical case of manifolds, this circle of ideas is supported by facts such as 
\begin{enumerate}
\item The geodesic distance on a manifold can be reconstructed from any spectral triple defined from a Dirac type operator, see \cite[Chapter VI]{Connesbook}.
\item The conformal class of a metric is uniquely determined by the bounded transform of a spectral triple defined from a Dirac type operator modulo compact perturbations, see \cite{baer}.
\item A Riemannian spin$^c$-manifold can be reconstructed from the spectral triple\footnote{Once it is decorated with some further manifold-like structures.} associated with the spin$^c$-Dirac operator, see \cite{connesreconstr}.
\end{enumerate}
We have made a choice of a distinguishing in terminology between {\bf \emph{spectral triples}} and {\bf \emph{unbounded Fredholm modules}} as the former corresponds to prescribing a ``non-commutative geometry" while the latter is a cocycle for a cohomology theory for $C^*$-algebras. Despite this, we abuse the notation by sometimes identifying an unbounded Fredholm module $(\pi,\He,D)$ with the spectral triple $(\pi(\textnormal{Lip}(\pi,\He,D)),\He,D)$ for $A/\ker \pi$.

\subsection*{Obstructions to finite summability}
\label{finsumsubsec}

Summability of Fredholm modules is based on the idea of refining the compactness properties in its definition by requiring that the compact operators appearing in Definition \ref{bddfred} and \ref{unbddfred} belong to a finer symmetrically normed operator ideals. For details on symmetrically normed operator ideals, the reader is referred to \cite{Connesbook, losuza, simon}. We will mainly use finite summability and $\theta$-summability; they are respectively defined using Schatten ideals and the $\mathrm{Li}$-ideals. Throughout the paper, we let $\He$ denote a separable Hilbert space. For a compact operator $T$ on $\He$ we let $(\mu_k(T))_{k\in \N}\subseteq \R_+$ denote a decreasing enumeration of the singular values of $T$. Recall that the \emph{Schatten ideals} are defined as 
\[\mathcal{L}^{p}(\He):=\left\{T\in\mathbb{K}(\He): (\mu_k(T))_{k\in \N}\in \ell^p(\N) \right\},\]
for $p>0$.  These spaces are not closed in operator norm and form ideals of compact operators in $\mathbb{B}(\He)$. The homogeneous function 
$$\|T\|_{\mathcal{L}^p}:=\sqrt[p]{\textnormal{Tr}((T^{*}T)^{\frac{p}{2}})}=\|(\mu_k(T))_{k\in \N}\|_{\ell^p(\N)},$$ 
makes $\mathcal{L}^p(\He)$ into a  symmetrically normed operator ideal (in particular a Banach $*$-algebra) for $p\in [1,\infty)$ and for $p\in (0,1)$ into a quasi-normed space. For $p\in [1,\infty)$, the spaces
\begin{align*}
\mathcal{L}^{p,\infty}(\He)&:=\left\{T\in\mathbb{K}(\He): \mu_k(T)=O(k^{-1/p}) \right\}\quad\mbox{and}\\
\textnormal{Li}^{\frac{1}{p}}(\He)&:=\left\{T\in\Ko(\He): \mu_k(T)=O(\log(k)^{-1/p})\right\},
\end{align*}
form symmetrically normed operator ideals as well. We use the notation $\textnormal{Li}(\He):=\textnormal{Li}^1(\He)$. 

\begin{deef*}

Let $(\pi, \He,F)$ be an analytic $K$-cycle for a $C^{*}$-algebra $\cstar$. Then $(\pi, \He, F)$ is said to be $p$-\emph{summable} if the $*$-algebra
\begin{align*}
\mbox{H\"ol}^p&(\pi,\He,F)\\
&:= \{a\in \cstar: [F,\pi(a)]\in\mathcal{L}^{p}(\He), \;\pi(a)(F^*-F),\,\pi(a)(F^2-1)\in \mathcal{L}^{p/2}(\He)\},
\end{align*}
is norm dense in $\cstar$. If $\mathcal{L}^{p}(\He)$ and $\mathcal{L}^{p/2}(\He)$ is replaced with $\textnormal{Li}^{\frac{1}{2}}(\He)$ respectively $\textnormal{Li}(\He)$, $(\pi,\He, F)$ is $\theta$-\emph{summable}. An unbounded Fredholm module $(\pi,\He,D)$ is $p$-summable if $\pi(a)(D\pm i)^{-1}\in\mathcal{L}^{p}(\He)$, for $a$ in a subalgebra of $\mathrm{Lip}(\pi,\He,D)$ dense in $\cstar$, and $\theta$-summable if $\mathcal{L}^{p}(\He)$ is replaced with $\textnormal{Li}^{\frac{1}{2}}(\He)$.
\end{deef*}

More generally, one can speak of summability relative to any ideal of operators. Whenever $\mathcal{I},\mathcal{J}\subseteq \Ko(\He)$ are $*$-ideals such that $\mathcal{J}=\{a:\,a^*a\in \mathcal{I}\}$, we will say that $(\pi, \He, F)$ is $\mathcal{J}$-summable if the $*$-algebra
\[\mbox{H\"ol}^\mathcal{J}(\pi,\He,F):= \left\{a\in \cstar: [F,\pi(a)]\in\mathcal{J}, \;\pi(a)(F^*-F),\,\pi(a)(F^2-1)\in \mathcal{I}\right\},\]
is norm dense in $\cstar$. The $*$-algebra $\mbox{H\"ol}^\mathcal{J}(\pi,\He,F)$ forms a Banach $*$-algebra closed under holomorphic functional calculus once the $*$-ideals $\mathcal{I}$ and $\mathcal{J}$ are Banach $*$-ideals in $\Bo(\He)$ such that $\|a\|_\mathcal{J}^2=\|a^*a\|_{\mathcal{I}}$, see more in  \cite[Proposition $3.12$]{cuntzblackadar}. Yet another instance is if $\mathcal{I}(\He)=\mathcal{L}^{p/2,\infty}(\He)$ and $\mathcal{J}(\He)=\mathcal{L}^{p,\infty}(\He)$, in this case we refer to summability as $p^+$-summability. \newline

The motivation for the terminology of the H\"older subalgebra comes from the prototypical example of a bounded Fredholm module on a manifold. Let $M$ be a smooth closed $n$-dimensional manifold and $F$ a self-adjoint pseudo-differential operator of order $0$ acting on a hermitean vector bundle $E\to M$ such that $F^2=1$. Letting $\pi$ denote the representation of $C(M)$ on $L^2(M,E)$ given by pointwise multiplication, we obtain an odd bounded Fredholm module $(\pi,L^2(M,E),F)$. If $E$ is graded and $F$ odd in this grading this Fredholm module can be viewed as an even Fredholm module. It follows from combining the results reviewed in \cite[Section $3.6$]{nlin} with \cite[Proposition $1$]{SW}, the Weyl law for elliptic operators and standard results of real interpolation theory that there is a continuous inclusion of the H\"older continuous functions into the H\"older algebra:
\begin{equation}
\label{holdincl}
C^\alpha(M)\subseteq \mbox{H\"ol}^{\frac{n}{\alpha}^+}(\pi,L^2(M,E),F).
\end{equation}

If $(\pi, \He, F)$ is $p$-summable, $p$ is referred to as the degree of summability of $(\pi,\He,F)$. In geometric situations, we saw in Equation \eqref{holdincl} that the degree of summability often is related to the dimension of the underlying space via some type of Weyl law. The notion of $\theta$-summability is robust in the sense that $\theta$-summable $K$-cycles can be lifted to unbounded $\theta$-summable Fredholm modules (cf. \cite[Chapter IV.$8.\alpha$, Theorem 4]{Connesbook}). Particular instances of this phenomenon are known for finite summability as well; notably, in the paper \cite{SW} a lifting result for the group algebra of a group of polynomial growth was established. The general situation is quite different in the case of finite summability.

The paper \cite{Connestrace} shows that the existence of a finitely summable unbounded Fredholm module over a $C^{*}$-algebra $\cstar$ implies the existence of a tracial state on $\cstar$. In particular, purely infinite $C^{*}$-algebras do not admit finitely summable unbounded Fredholm modules. Recent results by Emerson-Nica \cite{EN2} show that certain purely infinite $C^*$-algebras arising as boundary crossed product $C^*$-algebras associated to hyperbolic groups are ``uniformly summable"; that is, they admit finitely summable bounded Fredholm modules in a strong sense made precise below in Definition \ref{defofunifsum}. Thus, a general lifting construction for Fredholm modules, preserving finite summability, is impossible. 

We show, among other results in this paper, that a result similar to that of \cite {EN2} holds for Cuntz-Krieger algebras, which also are purely infinite in many cases. We furthermore provide a class of examples of $\theta$-summable unbounded Fredholm modules on Cuntz-Krieger algebras such that their phases are finitely summable. This difference in finite summability for bounded and unbounded Fredholm modules indicates not only that lifting is a delicate matter but that the same holds for finer analytic properties of bounded transforms and the choice of $\chi$.

\begin{deef*} 
Let $\cstar$ be a $C^{*}$-algebra. We say that the class $x\in K^*(\cstar)$ is \emph{summable of degree p} if there exists a $p$-summable Fredholm module representing $x$. We define the degree of summability of $x$ to be the infimum of the set of numbers $p>0$ for which $x$ is $p$-summable. The odd (even) degree of summability of $\cstar$ is the supremum of the degree of summability of all odd (even) $K$-homology classes.
\end{deef*}

When taking the infimum of a set, we always apply the convention that the infimum of the empty set if infinite. We say that the $C^*$-algebra $\cstar$ has \emph{finitely summable odd respectively even $K$-homology} if it has a finite odd respectively even degree of summability. We say that a $K$-homology class is finitely summable if it has a finite degree of summability. The summability degree of a $C^*$-algebra is clearly an isomorphism invariant. An interesting question is if it is a homotopy invariant. A related open problem hinted above is to find obstructions for finite summability of $K$-homology classes similar to the tracial obstructions for finitely summable unbounded Fredholm modules from \cite{Connestrace}. We also note the terminology \emph{uniformly summable} from \cite{EN2}.

\begin{deef*}[\cite{EN2}]
\label{defofunifsum}
A $C^*$-algebra $\cstar$ is said to be uniformly summable if there is a $p>0$ and a dense $*$-subalgebra $\mathcal{\cstar}\subseteq \cstar$ such that any $x\in K^*(\cstar)$ admits a representative that is $p$-summable on $\mathcal{\cstar}$. 
\end{deef*}

\subsubsection*{Example}

We have one example of a $K$-homology class that is not finitely summable. This result is not to be confused with the interesting results of \cite{puschnigg, ravethesis} where a subalgebra of $\cstar$, on which $K$-homology classes are required to be finitely summable, is fixed. 

\begin{lem*}
Let $\cstar:=\bigoplus_{j=1}^\infty C( S^{2j-1})$ -- the $C_0$-direct sum of odd-dimensional spheres. There is a $K$-homology class $x\in K^1(\cstar)$ with infinite degree of summability.
\end{lem*}

\begin{proof}
Consider the sum of fundamental classes $x=\sum_{j=1}^\infty [S^{2j-1}]$, that is, $x$ is represented by $(\pi,\bigoplus_{j=1}^\infty L^2(S^{2j-1}),F)$ where $\pi$ is action by pointwise multiplication and $F=\oplus F_j$ where $F_j=2P_j-1$ and $P_j$ is the Szeg\"o projection on $S^{2j-1}$. This is a well defined Fredholm module since $a\mapsto [F,\pi(a)]$ is a norm-continuous mapping $A\to \Bo(\bigoplus_{j=1}^\infty L^2(S^{2j-1}))$ of norm at most $2$ and for $a$ in the dense subalgebra $C_c(\coprod_{j=1}^\infty S^{2j-1})\subseteq A$ it holds that $[F,\pi(a)]\in \mathcal{L}^p(\bigoplus_{j=1}^\infty L^2(S^{2j-1}))$ for any $p>\dim(\supp(a))+1$. It holds that $K^1\left(\cstar\right)\cong\prod_{j=1}^\infty \Z$ and it is a well known fact that $x|_{C(S^{2j-1})}\in K^1(C(S^{2j-1}))\cong \Z$ is a generator for any $j$, see for instance \cite{venugopalkrishna}. 

Suppose that $x$ admits a $p$-summable representative $(\tilde{\pi}, \mathcal{H}, \tilde{F})$. By \cite[Proposition $3.12$]{cuntzblackadar} it follows that $\mathcal{\cstar}:= \mbox{H\"ol}^p(\tilde{\pi},\He,\tilde{F})$ is not only dense but also holomorphically closed in $\bigoplus_{j=1}^\infty C( S^{2j-1})$. By a standard approximation argument, using that $\mathcal{\cstar}$ is holomorphically closed, the characteristic function $p_k$ for $S^{2k-1}\subseteq \coprod_{j=1}^\infty S^{2j-1}$ belongs to $\mathcal{\cstar}$. In particular $p_k\mathcal{\cstar}\subseteq C(S^{2k-1})$ is a holomorphically closed dense subalgebra. It follows that the degree of summability of $x$ is bounded from below by that of $x|_{S^{2k-1}}$. By \cite[Proposition $3$]{dougvoic} we have a lower estimate for the degree of summability of any $z\in K^1(C(S^{2k-1}))\setminus \{0\}$ given by $2k-2$. Hence we have a contradiction since our assumptions imply that $p\geq 2k-2$ for all $k$.
\end{proof}

\subsection*{Content and organization of the paper} 

The content of the paper can be summarized in the following Theorem. We use the letter $A$ for an $N\times N$-matrix only containing the numbers $0$ and $1$. It will be clear from its context when $A$ is a matrix and when it is a $C^*$-algebra as above.

The notation $\Omega_A$ is for the space of characters of the standard maximal abelian subalgebra in the Cuntz-Krieger algebra $O_A$ associated with the $N\times N$-matrix $A$, see more in Section \ref{groupoidsection} below. There is a $C(\Omega_A)$-valued conditional expectation on $O_A$. We let $\mathpzc{E}^\Omega_A$ denote the $C(\Omega_A)$-Hilbert $C^*$-module closure of $O_A$ and $\pi_A^\Omega:O_A\to \End_{C(\Omega_A)}^*(\mathpzc{E}^\Omega_A)$ the left $O_A$-action.

\begin{thm*} 
The odd degree of summability of a Cuntz Krieger algebra $O_{A}$ is $0$. Moreover, there exists an unbounded bivariant $(O_{A}, C(\Omega_{A}))$-cycle $(\mathpzc{E}^{\Omega}_{A},D)$ with the following property.  For each class $x\in K^{1}(O_A)$ one can find a finite collection $(\omega_{j,x})_{j=1}^{m_x}\subseteq \Omega_A$ and localize $D$ at each of these characters to construct the self-adjoint operator $D_{j,x}$ on $\mathpzc{E}^\Omega_A\otimes_{\omega_{j,x}}\C$ such that  the unbounded Fredholm module 
$$\left(\bigoplus_{j=1}^{m_x}(\pi_{A}^\Omega\otimes_{\omega_{j,x}}\!\!\!\!\id_\C),\;\bigoplus_{j=1}^{m_x}(\mathpzc{E}^\Omega_A\otimes_{\omega_{j,x}}\C),\;\bigoplus_{j=1}^{m_x}D_{j,x}\right),$$
is a $\theta$-summable representative for $x\in K^{1}(O_{A})$. Moreover, each of the triples 
$$\left(\pi_{A}^\Omega\otimes_{\omega_{j,x}}\id_\C,\mathpzc{E}^\Omega_A\otimes_{\omega_{j,x}}\C,D_{j,x}|D_{j,x}|^{-1}\right)$$
form $p$-summable analytic $K$-cycles for any $p>0$ .
\end{thm*}

\begin{remark*}
The first statement of this Theorem should be compared to the results of \cite{EN2}. The intersection of applications for this paper with \cite{EN2} lies in the examples of discrete hyperbolic groups $\Gamma$ such that $C(\partial \Gamma)\rtimes \Gamma$ is a Cuntz-Krieger algebra. E.g. when $\Gamma$ is a free group (see below in Subsubsection \ref{freegroupexample}). 

In these cases, the results of \cite{EN2} are stronger in regards to finite summability of bounded Fredholm modules as they consider also the even $K$-homology. We compare the two approaches in a special case in Subsubsection \ref{dualofunitfreegroup}.
\end{remark*}

\begin{remark*}
Whittaker \cite{whittakerthesis,whittakerpaper}, has carried out constructions similar to those in this paper. The computational approaches differs, but the spirit prevails. It is an interesting question if the results of this paper carry over to general Smale spaces relating to the work of Whittaker.
\end{remark*}

\begin{remark*}
In regards to the discussions above, we can use the same $*$-algebra for all of the Fredholm modules constructed in this paper. Namely, the $*$-algebra generated by the $C^*$-generators of the Cuntz-Krieger algebra.
\end{remark*}

The paper is organized as follows. In Section \ref{groupoidsection} we recall some well known facts about Cuntz-Krieger algebras, focusing on their origin in the dynamics of sub shifts of finite type. This is encoded by means of a groupoid, first studied by Renault \cite{Ren1, Ren2}. The possibility to interchange the groupoid picture of Cuntz-Krieger algebras and the standard generator picture, used in the original definition of Cuntz-Krieger \cite{CK}, is crucial to identifying the $K$-homology classes of the Fredholm modules and spectral triples constructed in this paper. 

Section \ref{sectionfinisumandputnamkam} contains the proof of the fact that the odd degree of summability of $O_A$ is $0$, this is stated in Theorem \ref{finsumthm}. To be precise, we recall the construction from \cite{kaminkerputnam} of Poincar\'e duality $K^*(O_A)\cong K_{*+1}(O_{A^T})$. Using this construction, we identify exactly which odd $K$-homology cycles we need to prove finite summability for. The results of Section \ref{sectionfinisumandputnamkam} implies that any odd $K$-homology class can be represented by a Fredholm module on the GNS-space $L^2(O_A)$ associated with the KMS-state on $O_A$. It would be desirable to prove that the duality class $\Delta\in K^1(O_A\minotimes O_{A^T})$ in fact is finitely summable, in which case it would follow that the even degree of summability of $O_A$ is finite. We can in general only prove $\theta$-summability of $\Delta$ (see Theorem \ref{thetaandfinitekp}). In certain cases, for instance $SU_q(2)$, we obtain finite summability of a $K$-cycle representing $\Delta$.

For the construction of unbounded Fredholm modules on $O_A$ in a given odd $K$-homology class, we consider two approaches. One using the subalgebra of fixed points for the gauge action, an AF-algebra, in Section \ref{oafasection} and one using the standard maximal abelian subalgebra in Section \ref{oacasection}. The approach of using the fixed point algebra is explored first as there is an unbounded $KK$-cycle naturally associated with the gauge action, the \emph{gauge cycle}. We prove that the gauge cycle plays the role of a boundary mapping in the Pimsner-Voiculescu six term exact sequence associated with the gauge action. As such, the possible $K$-homology classes of the unbounded Fredholm modules that can be constructed from the fixed point algebra by means of a Kasparov product with the gauge cycle can be computed. We carry out these computations in some special cases in Subsection \ref{fixedcompsec}. In some cases any $K$-homology class is of this form (see Remark \ref{almostfreegroupremark}), only the odd ones are (e.g. for $SU_q(2)$, see Remark \ref{suq2rem}) and in other cases only the trivial class is (e.g. the Cuntz algebra $O_N$, see Remark \ref{ktheorycomputationsforthealgebraon}).

Before considering the approach of constructing unbounded Fredholm modules from the maximal abelian subalgebra in Section \ref{oacasection}, we recall some spectral triples in Section \ref{sectionbptriples} considered by Bellissard-Pearson \cite{BP}. These spectral triples are interesting since there is an obstruction to extending them to the ambient Cuntz-Krieger algebra coming from the class of the unit $[1_{O_{A^T}}]\in K_0(O_{A^T})$ under Poincar\'e duality $K_0(O_{A^T})\cong K^1(O_A)$. They also provide a natural candidate for constructing spectral triples on $O_A$ with geometric content. In Section \ref{oacasection}, we construct an unbounded bivariant $(O_A,C(\Omega_A))$-cycle. The restrictions of this bivariant cycle to suitable fibres over $\Omega_A$ generate the odd $K$-homology group of $O_A$. We may even identify the phases of these unbounded Fredholm modules with \emph{finitely summable} analytic $K$-cycles very similar to those constructed in Section \ref{sectionfinisumandputnamkam}. In particular, this shows that the Kasparov product
\[KK_{1}(O_{A},C(\Omega_{A}))\otimes K^{0}(C(\Omega_{A}))\rightarrow K^{1}(O_{A}),\]
is surjective. This stands in sharp contrast with the situation where $C(\Omega_{A})$ is replaced with the fixed point algebra $F_{A}$. E.g. for the Cuntz algebra where $K^0(F_N)=0$ but $K^1(O_N)\cong \Z/(N-1)\Z$ (cf. computations in Subsubsection \ref{subsubonon}).

In view of this, we end this paper with a construction of an unbounded Kasparov product between the unbounded bivariant $(O_A,C(\Omega_A))$-cycle with the Bellissard-Pearson spectral triples -- a certain infinite direct sum of point evaluations with dynamical content. 
 
We compute the class of this Kasparov product in rational $K$-homology. The construction uses and extends the techniques of \cite{BMS, KaLe,Mes} to account for naturally ocurring unbounded commutators. This is achieved in the context of \emph{$\epsilon$-unbounded Fredholm modules}, a slight weakening of the notion of unbounded Fredholm modules. This weakening has been hinted at in the literature. We describe the main properties of $\epsilon$-unbounded Fredholm modules in the Appendix.

\begin{remark*} In the Ph.D. thesis of Senior \cite{rogerthesis}, an unbounded Fredholm operator for the quantum group $SU_{q}(2)$ is constructed through the same connection techniques employed here. Although the algebra $C(SU_{q}(2))$ is a Cuntz-Krieger algebra (see Subsection \ref{suqtwoexample}), the base space taken in \cite{rogerthesis} is the noncommutative Podl\'{e}s sphere, and thus the constructions differ fundamentally. In particular, the $\epsilon$-unbounded Fredholm module techniques we employ to prove that our operators represent Kasparov products do not apply there. 
\end{remark*}
\newpage

\Large
\section{Groupoids, $C^{*}$-algebras and dynamics}
\label{groupoidsection}
\normalsize

In this section we will recall some well known facts about the dynamics of sub shifts of finite type, Cuntz-Krieger algebras and the interplay in between them arising from a certain groupoid. The purpose of this section is to set notations and to introduce the underlying classical geometry before describing its noncommutative geometry. Relevant references are provided in each subsection.

\subsection{Subshifts of finite type on the boundary of a tree}
\label{realizingboundarysubsec}
In this section we recall basic facts and introduce notation regarding subshifts of finite type. We let $A=(A_{ij})_{i,j=1}^N$ denote an $N\times N$ matrix with coefficients being $0$ or $1$. Sometimes we write $A(i,j)=A_{ij}$. The matrix $A$ can be thought of as defining the admissible paths in a Markov chain, where a jump from $i$ to $j$ is admissible if and only if $A(i,j)=1$. We always assume that no row nor column of $A$ is zero to guarantee that there is always an allowed jump into as well as out of a letter $j\in \{1,\ldots, N\}$.

There are several well studied geometric objects associated with this Markov chain. The first is the compact space of infinite admissible words:
\[\Omega_A:=\{(x_k)_{k\in \N_+}\in\{1,\ldots, N\}^\N: \forall k : A(x_k,x_{k+1})=1\}, \]
equipped with the topology induced from the compact product topology on $\{1,\ldots, N\}^\N$. The space $\Omega_A$ is totally disconnected. There is a natural shift operator 
\[\begin{split}\Omega_A &\rightarrow \Omega_{A} \\
(x_k)_{k\in \N} &\mapsto (x_{k+1})_{k\in \N}.\end{split}\]
The pair $(\Omega_A,\sigma)$ is called a \emph{subshift of finite type} and is amongst the most well studied systems in dynamics, see for example \cite{Kitchens, Lind, PP}.

We call a sequence of numbers $\mu=(\mu_j)_{j=1}^M$ with $\mu_j\in \{1,\ldots,N\}$ a finite word of length $M$. The length $M$ of $\mu$ is denoted by $|\mu|$. A finite word $\mu=(\mu_j)_{j=1}^M$ is said to be admissible for $A$ if $A(\mu_j,\mu_{j+1})=1$ for $j=1,\ldots, M-1$. To simplify notation we often write $\mu_1\mu_2\cdots \mu_M$ for a finite word $\mu=(\mu_j)_{j=1}^M$. The empty word is defined to be an admissible finite word that we denote by $\circ_A$. The length of the empty word is defined to be $0$. The set of all admissible finite words will be denoted by $\mathcal{V}_A$. For $k\in \N$, we use the notation
\begin{equation}
\label{wordsoflenghtk}
\phi(k):=\#\{\mu\in \mathcal{V}_A:|\mu|=k\}.
\end{equation}

The space $\Omega_A$ splits into \emph{cylinder sets} $C_{\mu}$ associated with finite words $\mu$:
\[C_{\mu}:=\{(x_{k})_{k\in\N}: x_{1}\cdots x_{|\mu|}=\mu\}\quad\mbox{and}\quad C_\circ =\Omega_A.\]
 A finite word $\mu$ is admissible if and and only if $C_{\mu}\neq \emptyset$. The sets $C_{\mu}$ are clopen subsets of $\Omega_A$ and generate the topology. The shift $\sigma$ is injective on each $C_{\mu}$ if $|\mu|>0$. We will abuse the notation by denoting the with $\sigma$ associated endomorphism of the $C^*$-algebra $C(\Omega_A)$ also by $\sigma$. Whenever $X\subseteq \Omega_A$ is a clopen subset, the characteristic function $\chi_X$ of $X$ defines a locally constant continuous function. These observations imply the following Proposition.

\begin{prop}
\label{afstructureonfunctions}
The $C^*$-algebra of continuous functions $C(\Omega_A)$ forms an AF-algebra. The AF-filtration is given by
\[\mathcal{C}_k:=\bigoplus_{|\mu|=k} \C \chi_{C_\mu}\cong \C^{\phi(k)}.\]
The inclusions $\mathcal{C}_k\hookrightarrow \mathcal{C}_{k+1}$ are induced from the partition $C_\mu=\cup_{j=1}^NC_{\mu j}$.
\end{prop}

The space $\Omega_{A}$ can be viewed as the boundary at infinity of the tree $\mathcal{V}_{A}$ of finite $A$-admissible words. The countable set $\mathcal{V}_A$ becomes a tree by allowing an edge between $\mu$ and $\nu$ whenever $\nu=\mu i$ for some $i$. By choosing the empty word $\circ_{A}$ as the base point, $\mathcal{V}_{A}$ becomes a rooted tree. The space $\Omega_{A}$ is naturally identified with the space of infinite paths starting at $\circ_A$. As such, $\Omega_{A}$ carries a natural metric, defined for $x\neq y$ by 
\begin{equation}
\label{metric}
d_{\Omega_A}(x,y):=\e^{-\min\{n: x_{n}\neq y_{n}\}}.
\end{equation}
Informally speaking, two paths are close when they stay on the same track for a long time. In this metric, the cylinder sets satisfy
\[\textnormal{diam}(C_{\mu})=\e^{-|\mu|}.\]
The tree $\mathcal{V}_{A}$ is in particular a Gromov hyperbolic space, and $\overline{\mathcal{V}_{A}}:=\mathcal{V}_{A}\cup \Omega_{A}$ is a compactification of $\mathcal{V}_{A}$ when given the topology generated by that of $\mathcal{V}_{A}$ and the sets
\[C_{\mu}^{\varepsilon}:=C_{\mu}\cup \{\nu\in \mathcal{V}_{A}:\nu\in C^\mathcal{V}_{\mu}, d_{\Omega_A}(\circ_A,\nu)\geq\varepsilon^{-1}\}.\]
Here $C^\mathcal{V}_\mu$ denotes the finite analogue of the cylinder set $C_{\mu}$;
\[C^\mathcal{V}_{\mu}:=\{\nu\in \mathcal{V}_{A}: \nu=\mu\lambda\;\;\mbox{for some}\;\; \lambda\in \mathcal{V}_A\}\subset \mathcal{V}_{A}.\]

\begin{deef}
\label{cylinderchoicefunctionz}
A function $\mathfrak{t}:\mathcal{V}_A\to \Omega_A$ is said to satisfy the cylinder condition if 
$$\mathfrak{t}(\mu)\in C_\mu\quad\forall \mu\in \mathcal{V}_A.$$
\end{deef}

The next proposition shows that functions satisfying the cylinder condition provides a natural candidate for a $*$-homomorphism splitting the following short exact sequence of $C^*$-algebras:
\[0\to C_0(\mathcal{V}_A)\to C\left(\overline{\mathcal{V}_A}\right)\to C(\Omega_A)\to 0.\]

\begin{prop}
\label{cylinderguaranteescont}
If $\mathfrak{t}:\mathcal{V}_A\to \Omega_A$ satisfies the cylinder condition, see Definition \ref{cylinderchoicefunctionz}, then pullback along $\mathfrak{t}^*:C(\Omega_A)\to C_b(\mathcal{V}_A)$ factors over a $*$-homomorphism
$$\mathfrak{t}^*:C(\Omega_A)\to C\left(\overline{\mathcal{V}_A}\right) \quad\mbox{such that}\quad (\mathfrak{t}^*f)|_{\Omega_A}=f.$$
\end{prop}

\begin{proof}
The Proposition follows once proving that the mapping $\bar{\mathfrak{t}}:\overline{\mathcal{V}_A}\to \Omega_A$ given by $\bar{\mathfrak{t}}|_{\mathcal{V}_A}:=\mathfrak{t}$ and $\bar{\mathfrak{t}}|_{\Omega_A}:=\id_{\Omega_A}$ is continuous. This follows from the fact that $\bar{\mathfrak{t}}^{-1}(C_\mu)\subseteq C^\epsilon_\mu$ for some $\epsilon>0$, so $\bar{\mathfrak{t}}$ is continuous.
\end{proof}

The set of finite words comes with a shift mapping defined as 
$$\sigma_{\!\mathcal{V}}:\mathcal{V}_{A}\setminus \circ\rightarrow \mathcal{V}_{A},\quad \mu=\mu_1\mu_2\cdots \mu_N\mapsto \mu_2\cdots \mu_N.$$ 
The endomorphism $\sigma:C(\Omega_{A})\rightarrow C(\Omega_{A})$ has an associated \emph{transfer operator}
\begin{equation}
\label{finitetransfer}
L_{\sigma}(f)(x):=\sum_{y\in\sigma^{-1}(x)}f(y).
\end{equation}
This operator extends to an operator $\bar{L}_{\sigma}:C\left(\overline{\mathcal{V}_A}\right)\rightarrow C\left(\overline{\mathcal{V}_A}\right)$ by setting $\bar{L}_{\sigma}(f)|_{\mathcal{V}_A}(\mu):=\sum_{\nu\in\sigma^{-1}_{\mathcal{V}}(\mu)}f(\nu)$ for $\mu\in \mathcal{V}_A$. Via the Riesz Representation Theorem, the induced operator $\bar{L}_{\sigma}^{*}:C\left(\overline{\mathcal{V}_A}\right)^{*}\rightarrow C\left(\overline{\mathcal{V}_A}\right)^{*}$ can be viewed as an operator on the Borel measures on $\overline{\mathcal{V}_A}$.

\begin{deef} 
A Borel measure $\mu$ on $\Omega_{A}$ is called \emph{conformal of dimension} $\delta_{A}$ if $\bar{L}_{\sigma}^{*}(\mu)=e^{\delta_{A}}\mu$.
\end{deef}

There is a canonical $\sigma$-conformal measure on the space $\Omega_{A}$, which can be constructed explicitly. Denote by $\delta_{A}$ the \emph{upper Minkowski dimension} (sometimes called the \emph{upper box dimension}, see e.g. \cite{Falc}) of $\Omega_{A}$.

\begin{thm}[cf. \cite{BP}, Theorem $2$] 
\label{Poincare}
Let $s>0$. The series $\sum_{\nu\in\mathcal{V}_{A}}e^{-s|\nu|}$ is convergent for all $s> \delta_{A}$ and divergent for $0<s\leq \delta_{A}$. Consequently $\delta_{A}:=\inf \{s: \sum_{\nu\in\mathcal{V}_{A}} e^{-s|\nu|}<\infty\}.$
\end{thm}

A direct corollary of Theorem \ref{Poincare} is the following (recall the definition of $\phi$ from Equation \eqref{wordsoflenghtk}).

\begin{cor}
\label{phiasymptotics}
There is a positive sequence $C_s\in \ell^1(\N)$ such that $\phi(k)\leq C_s(k) \e^{\;sk}$ whenever $s>\delta_A$.
\end{cor}

Consider the measures
\[\mu_{s}:=\frac{\sum e^{-s|\nu|}\delta_{\nu}}{\sum e^{-s|\nu|}},\]
viewed as an element of $C(\overline{\mathcal{V}_{A}})^{*}$. Subsequently define $\mu_{A}=\mathrm{w}^*\mbox{-}\lim_{s\downarrow \delta_{A}} \mu_{s}$, which is to be interpreted as a weak$^*$ limit in $C(\overline{\mathcal{V}_{A}})^{*}$. This is the well-known Patterson-Sullivan construction \cite{Pat, Sul}. Since the series of Theorem \ref{Poincare} diverges at $\delta_{A}$, the measure $\mu_{A}$ is supported only on the boundary $\Omega_{A}$. For $f\in C(\Omega_{A})$, $\int_{\Omega_{A}} f d\mu_{A}$ can be computed by choosing an extension $\tilde{f}$ to $\overline{\mathcal{V}_{A}}$, since any two such extensions differ by a function supported in $\mathcal{V}_{A}$.

\begin{thm}[cf. \cite{coornaert, Doener, Pat, Sul}] 
The measure $\mu_{A}$ is $\sigma$-conformal of dimension $\delta_{A}$. 
\end{thm}

\begin{proof}  
First we compute,
\[\begin{split}\int_{\Omega_{A}} f d L_{\sigma}^{*}\mu_{s}&=\frac{\sum_{\nu\in\mathcal{V}_{A}} e^{-s|\nu|}L_{\sigma}f(\nu)}{\sum_{\nu\in\mathcal{V}_{A}}e^{-s|\nu|}}=\frac{\sum_{\nu\in\mathcal{V}_{A}}e^{-s|\nu|}\sum_{\lambda\in\sigma^{-1}_{\mathcal{V}}(\nu)}f(\lambda)}{\sum_{\nu\in\mathcal{V}_{A}}e^{-s|\nu|}}\\
&=\frac{\sum_{\nu\in\mathcal{V}_{A}\setminus \circ}e^{-s(|\nu|-1)}f(\nu)}{\sum_{\nu\in\mathcal{V}_{A}}e^{-s|\nu|}}=e^{s}\frac{\sum_{\nu\in\mathcal{V}_{A}\setminus \circ}e^{-s|\nu|}f(\nu)}{\sum_{\nu\in\mathcal{V}_{A}}e^{-s|\nu|}},
\end{split}\]
and then, using that $\sum_{\nu\in\mathcal{V}_{A}}e^{-s\nu}$ diverges at $s=\delta_{A}$, we take the limit
\begin{align*}
\lim_{s\downarrow\delta_{A}}\int_{\Omega_{A}} f d L_{\sigma}^{*}\mu_{s} &= \lim_{s\downarrow\delta_{A}}e^{s}\frac{\sum_{\nu\in\mathcal{V}_{A}\setminus\circ}e^{-s|\nu|}f(\nu)}{\sum_{\nu\in\mathcal{V}_{A}}e^{-s|\nu|}}\\
&=\lim_{s\downarrow\delta_{A}}e^{s}\frac{\sum_{\nu\in\mathcal{V}_{A}}e^{-s|\nu|}f(\nu)}{\sum_{\nu\in\mathcal{V}_{A}}e^{-s|\nu|}}=e^{\delta_{A}}\int_{\Omega_{A}} f d\mu.
\end{align*}
\end{proof}

\subsection{Groupoids, $C^{*}$-algebras and modules}
Groupoids are an intermediate structure between spaces and groups. The $C^*$-algebras constructed from groupoids form a rich source of noncommutative $C^{*}$-algebras, and the groupoid origin provides a geometric description of those.

\begin{deef} 
A \emph{groupoid} is a small category $\mathcal{G}$ in which all morphisms are invertible. 
\end{deef}

The requirement of being \emph{small} is of a set-theoretical nature; the objects in $\mathcal{G}$ form a set. We denote the set of objects by $\mathcal{G}^{(0)}$ and the set of morphisms by $\mathcal{G}^{(1)}$. There is an inclusion $\mathcal{G}^{(0)}\rightarrow \mathcal{G}^{(1)}$ as identity morphisms. We often write $\mathcal{G}$ for $\mathcal{G}^{(1)}$. The domain and range maps are denoted $d,r:\mathcal{G}^{(1)}\rightarrow \mathcal{G}^{(0)}$ and the set of \emph{composable pairs} is
\[\mathcal{G}^{(2)}:=\{(\xi,\eta)\in\mathcal{G}\times\mathcal{G}: d(\xi)=r(\eta)\}.\]
This is itself a groupoid with domain and range maps the coordinate projections, and composition
\[(\xi_{1},\eta_{1})\circ (\eta_{1},\xi_{2}):=(\xi_{1},\xi_{2}).\]
If $\mathcal{G}$ carries a locally compact Hausdorff topology for which the maps $r,d$ and composition $\mathcal{G}^{(2)}\rightarrow\mathcal{G}$ are continuous, then $\mathcal{G}$ is said to be a locally compact Hausdorff groupoid.

\begin{deef} 
A locally compact Hausdorff groupoid $\mathcal{G}$ is \emph{\'{e}tale} if the fibers of the range map $r:\mathcal{G}\rightarrow \mathcal{G}^{(0)}$ are discrete.
\end{deef}

An \'{e}tale groupoid $\mathcal{G}$ carries a canonical \emph{Haar system} (see \cite{Ren1}), consisting of counting measure in each fibre of $r$. This allows for the definition of the \emph{convolution product} on $C_{c}(\mathcal{G})$, defined by 
\begin{equation}
\label{theconvolutiononccg}
f*g(\eta)=\sum_{\xi\in r^{-1}(\eta)}f(\xi)g(\xi^{-1}\eta),
\end{equation}
which is a finite sum because $f$ is compactly supported and $r^{-1}(\eta)$ is discrete.\newline

There is a locally compact Hausdorff \'{e}tale groupoid $\mathcal{G}_A$ encoding the dynamics of the totally disconnected compact space $\Omega_A$ and the self mapping $\sigma$. The unit space of $\mathcal{G}_A$ is defined as $\mathcal{G}_A^{(0)}:=\Omega_A$ and the morphism space by
\[\mathcal{G}_A^{(1)}:=\{(x,n,y)\in \Omega_A\times \Z\times \Omega_{A}: \exists k\in \N \;\mbox{s.t.}\; \sigma^{n+k}(x)=\sigma^k(y)\}.\]
The range and source mappings are defined by
\[r(x,n,y)=x\quad\mbox{respectively}\quad d(x,n,y)=y.\]
The composition is given by
\[(x,n,y)(y,m,z)=(x,m+n,z).\]
The groupoid $\mathcal{G}_A$ can be given a locally compact \'{e}tale topology in the following way (see \cite{Ren1,Ren2}). Let $m$ and $n$ be natural numbers, $U\subseteq \Omega_A$ an open set on which $\sigma^{m}$ is injective, and $V\subseteq \Omega_A$ an open set on which $\sigma^{n}$ is injective. The basic open sets for the topology are given by
\begin{equation}\label{etaletop} (U,m,n, V):=\{(x,m-n,y): \sigma^{m}(x)=\sigma^{n}(y)\}.\end{equation}
Since this groupoid is \'{e}tale, it admits a natural Haar system $\nu^{x}$ given by counting measure in the fibers. 

Recall that a measure $\mu$ on $\mathcal{G}^{(0)}$ is called \emph{quasi-invariant} if the induced measure $d\mu(\xi)=d\nu^{x}(\xi)d\mu(x)$ is equivalent to its inverse $d\mu(\xi^{-1})$. The Radon-Nikodym derivative $\Delta:=\frac{d\mu^{-1}}{d\mu}$ is a measurable 1-cocycle on $\mathcal{G}$ called the \emph{modular function}. If $\mathcal{G}$ is an \'{e}tale groupoid and $U\subset\mathcal{G}$ an open set on which both $r$ and $d$ are injective, define $T:r(U)\rightarrow d(U)$ by $x\mapsto d(r^{-1}(x)\cap U)$. A measure $\mu$ on $\mathcal{G}^{(0)}$ is quasi-invariant with modular function $\Delta$ if for every such $U$ we have
\[\frac{dT^{*}\mu}{d\mu}(x)=\Delta(r^{-1}(x)\cap U).\]
See more in \cite[Remark $3.22$]{Ren1}.

\begin{prop}
\label{quasi-inv} The measure $\mu_{A}$ is a quasi-invariant measure on $\Omega_{A}$ with modular function $\Delta(x,n,y)=e^{-\delta_{A}n}$.
\end{prop}

\begin{proof} 
The maps $r$ and $d$ are injective on the basic open sets $(U,m,n,V)$. For $n-m\geq 0$ and $\supp f\subset V$ we have
\[T^{*}f(x)=\sum_{y\in \sigma^{m-n}(x)} f(y)=L_{\sigma}^{n-m}f(x).\]
We conclude that $\int_{V}fdT^{*}\mu_{A}=\int_{V}fdL^{(n-m)*}\mu_{A}=e^{(n-m)\delta_{A}}\int_{V}fd\mu_{A}$. For $n-m<0$
\begin{align*}
\int_{V} f d\mu_{A}&=\int_{U}fdT^{-1*}\mu_{A}=\int_{U}fdL_{\sigma}^{(m-n)*}\mu_{A}\\
&=e^{(m-n)\delta_{A}}\int_{U}fd\mu_{A}=e^{(m-n)\delta_{A}}\int_{V}fdT^{*}\mu_{A},
\end{align*}
so in this case $\int_{V}fd T^{*}\mu_{A}=e^{(n-m)\delta_{A}}\int_{V}fd\mu_{A}$ as well.
\end{proof}

The reduced $C^{*}$-algebra of an \'{e}tale groupoid $\mathcal{G}$ is a certain $C^*$-algebra completion of the algebra that $C_{c}(\mathcal{G})$ forms under the convolution product \eqref{theconvolutiononccg}. There is a conditional expectation $\rho:C_{c}(\mathcal{G})\rightarrow C_{0}(\mathcal{G}^{(0)})$ given by restriction of functions to $\mathcal{G}^{(0)}$. To construct $C^{*}_{r}(\mathcal{G})$, define the $C_{0}(\mathcal{G}^{(0)})$-valued inner product
\begin{equation}
\label{innerprod}
\langle f,g\rangle(x):=\sum_{\xi\in r^{-1}(x)}\overline{f(\xi^{-1})}g(\xi^{-1})=\rho(f^{*}*g),
\end{equation}
which is $C_{0}(\mathcal{G}^{(0)})$-linear for multiplication from the right. The completion of $C_{c}(\mathcal{G})$ in the norm induced from \eqref{innerprod} is a Hilbert $C^*$-module $\mathpzc{E}^{\mathcal{G}}$, the \emph{Haar module}, on which $C_{c}(\mathcal{G})$ acts, via convolution, by adjointable operators. Its completion in the operator norm is $C^{*}_{r}(\mathcal{G})$. The map $\rho$ above extends to a conditional expectation
\begin{equation}\label{rho} \rho:C^{*}_{r}(\mathcal{G})\rightarrow C_{0}(\mathcal{G}^{(0)}).\end{equation}
This intrinsic construction of $C^{*}_{r}(\mathcal{G})$ was first considered in \cite{KS}. \newline

For a closed subgroupoid $\mathcal{H}\subset\mathcal{G}$, we can do a similar construction. Denote by $\rho_{\mathcal{H}}:C_{c}(\mathcal{G})\rightarrow C_{c}(\mathcal{H})$ the restriction map. This extends to a conditional expectation $\rho_{\mathcal{H}}:C^{*}_{r}(\mathcal{G})\rightarrow C^{*}_{r}(\mathcal{H})$, see \cite{Ren1}. Relative to the closed subgroupoid $\mathcal{G}^{(0)}\subset\mathcal{G}$, the inner product \eqref{innerprod} can be expressed as $\langle f,g\rangle=\rho_{\mathcal{G}^{(0)}}(f^{*}*g)$. We distinguish the domain and range mappings of $\mathcal{G}$ respectively $\mathcal{H}$ by an index, e.g. $r_\mathcal{G}:\mathcal{G}^{(1)}\to \mathcal{G}^{(0)}$.  There is a right $C_{c}(\mathcal{H})$-module structure on $C_{c}(\mathcal{G})$ given by
\[g\cdot h (\eta):=\sum_{\xi\in r^{-1}_\mathcal{H}(d_\mathcal{G}(\eta))}g(\eta\xi)h(\xi^{-1}),\quad \eta\in \mathcal{G},\] 
and the formula for the inner product is similar to \eqref{innerprod}:
\[\langle f,g\rangle (\eta):= \rho_{\mathcal{H}}(f^{*}*g)(\eta)=\sum_{\xi\in r^{-1}_\mathcal{G}(r_\mathcal{H}(\eta))} \overline{f(\xi^{-1}\eta)}g(\xi^{-1})\quad \mbox{for}\quad \eta\in \mathcal{H},\]
The completion of $C_{c}(\mathcal{G})$ with respect to this inner product is a Hilbert $C^*$-module $\mathpzc{E}^{\mathcal{G}}_{\mathcal{H}}$ over $C^{*}_{r}(\mathcal{H})$; there is also a left action of $C^{*}_{r}(\mathcal{G})$, which is defined by convolution. The $C^*$-algebra of $C^*_r(\mathcal{H})$-compact operators on such modules can be easily described. We now turn to a brief review of this description.

Consider the right action of $\mathcal{H}$ on $\mathcal{G}$ and its associated quotient space
\begin{equation}
\label{GH}
\mathcal{G}/\mathcal{H}=\{[\xi] : \xi\in\mathcal{G}, [\xi_{1}]=[\xi_{2}] \Leftrightarrow \exists \eta\in\mathcal{H}\; \xi_{1}\eta=\xi_{2}\}.
\end{equation} 
The space 
\[\mathcal{G}\ltimes \mathcal{G}/\mathcal{H}:=\{(\xi,[\eta]): d(\xi)=r(\eta)\},\]
can be made into a groupoid with $(\mathcal{G}\ltimes \mathcal{G}/\mathcal{H})^{(0)}=\mathcal{G}/\mathcal{H}$ by defining 
\begin{align*}
\mbox{range map:}\quad r(\xi,[\eta])&:=[\xi\eta],\\
 \mbox{domain map:}\quad d(\xi,[\eta])&:=[\eta],\\
\mbox{composition:} \quad (\xi_{1},[\eta])&\circ(\xi_{2},[\xi_{2}^{-1}\eta]):=(\xi_{1}\xi_{2},[\xi_{2}^{-1}\eta]),\\
\mbox{and inversion:}\quad(\xi,[\eta])^{-1}&:=(\xi^{-1},[\xi\eta]).
\end{align*}
This groupoid is \'{e}tale because both $\mathcal{G}$ and $\mathcal{H}$ are. The above construction is a special case of an  action of the groupoid $\mathcal{G}$ on a space, which in this case is $\mathcal{G}/\mathcal{H}$. In that context, the map $[\eta]\rightarrow r(\eta)$, viewed as a map $\mathcal{G}/\mathcal{H}\rightarrow \mathcal{G}^{(0)}$ is called the \emph{moment map} of the action. For the general theory of groupoid actions, its relation to  $C^{*}$-algebras and modules, and further references see \cite{Klaas1, Mes1, SO}.

\begin{thm}[cf. \cite{MRW, simswill}] 
\label{cstarisoofsub}
Let $\mathcal{G}$ be an \'etale groupoid and $\mathcal{H}\subset\mathcal{G}$ a closed subgroupoid. 
The mapping 
$$\pi^\mathcal{G}_\mathcal{H}:C^{*}_{r}(\mathcal{G}\ltimes \mathcal{G}/\mathcal{H})\to \Ko_{C^*_r(\mathcal{H})}(\mathpzc{E}_{\mathcal{H}}^{\mathcal{G}})$$ 
defined on $a\in C_c(\mathcal{G}\ltimes \mathcal{G}/\mathcal{H})\subseteq C^{*}_{r}(\mathcal{G}\ltimes \mathcal{G}/\mathcal{H})$ and $f\in \im(C_c(\mathcal{G})\to \mathpzc{E}_{\mathcal{H}}^{\mathcal{G}})$ by 
\[\pi^\mathcal{G}_\mathcal{H}(a)f(\eta):=\sum_{\xi\in r^{-1}(r(\eta))} a(\xi ,[\xi^{-1}\eta])f(\xi^{-1}\eta)  \quad\mbox{for}\quad \eta\in \mathcal{G},\]
is an isomorphism.
\end{thm}

The fact that $C^{*}_{r}(\mathcal{G}\ltimes \mathcal{G}/\mathcal{H})\xrightarrow{\sim} \Ko_{C^*_r(\mathcal{H})}(\mathpzc{E}_{\mathcal{H}}^{\mathcal{G}})$ follows from the Morita equivalence $\mathcal{H}\sim \mathcal{G}\ltimes \mathcal{G}/\mathcal{H}$ of groupoids and the results in \cite{MRW, simswill}. The explicit formula for the isomorphism can also be found in \cite[Equation (11)]{Mes1}.

\subsection{Cuntz-Krieger algebras}
\label{subsectionckalgebras}

Let $O_{A}$ be the Cuntz-Krieger algebra associated with the $N\times N$ matrix $A=(A_{ij})$. Recall our assumption on $A$; no row nor column in $A$ is $0$. The $C^*$-algebra $O_{A}$ was defined in \cite{CK} as the universal $C^{*}$-algebra generated by elements $S_{i}$ satisfying the relations
\begin{align}
\label{eq1}
S_{i}^{*}S_{i}=\sum_{j=1}^{N}A_{ij}S_{j}S_{j}^{*}&, \\
\label{eq2}
\sum_{i=1}^{N}S_{i}S^{*}_{i}&=1,\\
\label{eq3}
&S_{i}S_{i}^{*}S_{j}S_{j}^{*}=S_{j}S_{j}^{*}S_{i}S_{i}^{*}=\delta_{ij}S_{i}S_{i}^{*}.
\end{align}
Following the notation \cite{CK}, for the source projections we write $Q_i:=S_i^*S_i$ and for the range projections $P_i:=S_iS_i^*$. The relations \eqref{eq1}-\eqref{eq3} become
\begin{equation}
\label{CKshort}
P_iP_j=\delta_{ij}P_i \quad\mbox{and}\quad Q_i=\sum_{j=1}^N A_{ij}P_j.
\end{equation}
For any finite word $\mu=\mu_1\mu_2\cdots \mu_M$, we let $S_\mu\in O_A$ denote the element $S_{\mu_1}S_{\mu_2}\cdots S_{\mu_M}$. The relation \eqref{CKshort} guarantees that the element $S_\mu$ is non-zero if and only if $\mu$ is an admissible word.

\begin{prop}[Lemma $1.1$ of \cite{kpr}]
\label{sscomp}
The following computation holds:
\[ S_\nu ^*S_\gamma=
\begin{cases} S_\beta, \quad\mbox{if}\quad &\gamma=\nu\beta, \quad\mbox{for some}\quad \beta,\\
Q_{\nu_k}, \quad\mbox{if}\quad &\nu=\gamma=\nu_1\cdots \nu_k\\
S_\beta^*, \quad\mbox{if}\quad &\nu=\beta\gamma, \quad\mbox{for some}\quad \beta,\\
0, \quad&\mbox{otherwise}.
\end{cases}\]
Every non-zero word in $S_i$ and $S_j^*$ can be written as a finite sum of terms of the form $S_\mu S_\nu^*$ where the admissible $\mu=\mu_1\cdots\mu_k$ and $\nu=\nu_1\cdots \nu_l$ satisfy that $\mu_k=\nu_l$.
\end{prop}

The following fundamental result is due to Renault. 

\begin{thm}[\cite{Ren1, Ren2}]
\label{oaandcga}
There is a canonical isomorphism between the groupoid $C^{*}$-algebra $C_{r}^{*}(\mathcal{G}_A)$ and the universal $C^{*}$-algebra $O_{A}$.
\end{thm}

The isomorphism is implemented by mapping $S_{i}$ to the characteristic function of the set
\begin{equation}
\label{xi}X_{i}:=\left\{\left(x,1,\sigma(x)\right):x\in C_i\right\}.
\end{equation}
As the images of the $S_{i}$:s satisfy the Cuntz Krieger relations, we obtain a $*$-homomorphism $O_{A}\rightarrow C^{*}_{r}(\mathcal{G}_A)$. For more details on the proof see \cite{Ren2}.\newline

Recall the following condition, usually referred to as condition $(I)$, on the $N\times N$-matrix $A$. A finite admissible word $\nu=\nu_1\cdots \nu_R$ is a loop based in $j\in \{1,\ldots, N\}$ if $\nu_1=\nu_R=j$ and $\nu_k\neq j$ for $k=2,\ldots, R-1$. If any $j=1,\ldots, N$ satisfies that there is an admissible finite word $\mu=\mu_1\cdots \mu_M$ with $\mu_1=j$ and there are two different loops based in $\mu_M$, we say that $A$ satisfies condition $(I)$. The matrix $A$ satisfies condition $(I)$ if and only if $\Omega_A$ has no isolated points. An example when condition $(I)$ is satisfied is if $A$ is irreducible but not a permutation matrix. 

\begin{thm}[Theorem $2.14$ of \cite{CK}, Proposition $4.3$ of \cite{anananananantaranamenam}]
\label{variousstructuresonoa}
The Cuntz-Krieger algebra $O_A$ satisfies the following:
\begin{enumerate}
\item If $A$ is irreducible, $O_A$ is simple.
\item If $A$ satisfies $(I)$, then $O_A$ is purely infinite\footnote{Hence $O_A$ is also simple.}.
\end{enumerate}
\end{thm}

The quasi-invariant measure $\mu_{A}$ induces a functional 
\begin{align}
\label{definitionkmsstatephia}
\phi_{A}:C_{c}(\mathcal{G}_{A})&\rightarrow \C, \quad f\mapsto \int_{\Omega_{A}}f|_{\Omega_A}\rd \mu_{A},
\end{align}
which extends to a state on $C^{*}_{r}(\mathcal{G})$. The GNS-representation of $O_{A}$ on $L^{2}(O_{A},\phi_{A})$ is canonically isomorphic to the convolution representation of $C^{*}_{r}(\mathcal{G}_A)$ on $L^{2}(\mathcal{G}_A,\mu_{A})$. We will refer to this as the \emph{fundamental representation}. 

\subsubsection{The algebra $O_N$}
\label{examplethealgebraon}
Also known as the Cuntz algebra, was first introduced in \cite{thealgebraon}. The algebra $O_N$ is the universal $C^*$-algebra generated by $N$ orthogonal isometries. The algebra $O_N$ is the Cuntz-Krieger algebra associated with the symmetric $N\times N$-matrix giving by $A_{ij}=1$ for all $i,j$. The geometry of $\Omega_{O_N}$ takes a very simple form; since any word is admissible, it holds that $\mathcal{V}_{O_N}=\cup_{k\in \N}\{1,\ldots, N\}^k$ and $\phi(k)=N^k$. In this special case, the KMS-state $\phi_{O_N}$ can be computed as
\[\phi_{O_N}(S_\mu S_\nu^*)=\delta_{\mu,\nu}N^{-|\mu|}.\]

\subsection{The fixed point algebra of the circle action}
\label{subsectionfxedpointalgeba}

The Cuntz-Krieger groupoid comes with a natural circle action. We describe the action in both pictures of $O_{A}$. First of all, the map
\begin{equation}
\label{cocycle}
\begin{split}
c_A:\mathcal{G}_A&\rightarrow \Z\\
(x,n,y)&\mapsto n,
\end{split}
\end{equation}
is a continuous homomorphism, or a 1-cocycle. Note that $\ln \Delta=-\delta_{A}c_A$, with $\Delta$ as in Proposition \ref{quasi-inv}. This induces a disjoint union decomposition
\[\mathcal{G}_A=\bigcup_{n\in\Z}\mathcal{G}_{n},\]
where $\mathcal{G}_{n}=c^{-1}_A(n)$. Its kernel 
\[\mathcal{H}_{A}:=\ker c_A=c^{-1}_A(0)=\{(x,0,y):\exists k, \sigma^{k}(x)=\sigma^{k}(y)\},\]
is a closed subgroupoid. We denote $F_{A}:=C^{*}_{r}(\mathcal{H}_{A})\subset C^{*}_{r}(\mathcal{G}_{A})$. We remark that, by the remark on the end of page $3$ of \cite{CK}, the algebra $F_A$ is simple if $A$ is aperiodic. 
There is a $U(1)$-action on $C^{*}_{r}(\mathcal{G}_A)$ (see \cite{Ren1}) constructed from the cocycle $c_A$ via
\[\alpha_{t}(f)(\xi):=e^{itc_A(\xi)}f(\xi).\] We refer to this action as the \emph{gauge action}. The fixed point algebra for this action is exactly $F_{A}$. It is well known that the state $\phi_A$ \eqref{definitionkmsstatephia} satisfies the KMS-condition at inverse temperature $\delta_{A}$ with respect to the gauge action (cf. Definition $3.15$ and Proposition $5.4$ of \cite{Ren1}).

A third way of describing $F_A$ comes from the generators $S_i$. Observe that, in terms of the linearly spanning elements $S_\mu S_\nu^*$ coming from Proposition \ref{sscomp},
\[\alpha_{t}(S_\mu S_\nu^*)=e^{(|\mu|-|\nu|)it}S_\mu S_\nu^*.\]
Hence $F_A$ is the $C^*$-algebra generated by $S_\mu S_\nu^*$ for $|\mu|=|\nu|$. We define $F_A^l$ to be the span of all non-zero $S_\mu S_\nu^*$ where $|\mu|=|\nu|=l+1$. As was computed in the proof of \cite[Proposition $2.3$]{CK}; for a fixed $j$, the elements $S_\mu S_\nu^*$ where $|\mu|=|\nu|=l+1$ and $\mu_{l+1}=\nu_{l+1}=j$ form a set of matrix units; whenever $l+1=|\mu|=|\nu|=|\mu'|=|\nu'|$, 
\begin{equation}
\label{productssmunu}
S_\mu S_\nu^* S_{\mu'} S_{\nu'}^*=\delta_{\nu,\mu'} S_\mu S_{\nu_{l+1}}^*S_{\nu_{l+1}} S_{\nu'}^*=\delta_{\nu,\mu'} S_\mu S_{\nu'}^*.
\end{equation}
These identities follows from Proposition \ref{sscomp}. We can conclude the following Proposition. 

\begin{prop}[Proposition $2.3$ of \cite{CK}]
The space $F_A^l$ is closed under multiplication and adjoint. In particular, 
\[F_A=\overline{\cup_{l\in \N} F_A^l}\] 
is an $AF$-algebra. 
\end{prop}

The stabilization $F_A\minotimes \Ko$ admits yet another description in terms of groupoids. It follows from \cite[Lemma $3.4$]{Mes1} that  $\mathcal{G}_A/\mathcal{H}_A\cong\Omega_A\times \Z$.
The moment mapping $\mathcal{G}_A/\mathcal{H}_A\to \Omega_A$ is the projection onto the first coordinate under the above homeomorphism. Hence we can identify $\mathcal{G}_A\ltimes \mathcal{G}_A/\mathcal{H}_A= \mathcal{G}_A\times \Z$. We will denote elements of $\mathcal{G}_A\ltimes \mathcal{G}_A/\mathcal{H}_A$ by $(x,k,y,l)$. The range and domain mappings 
$$r,d:\mathcal{G}_A\ltimes \mathcal{G}_A/\mathcal{H}_A=\mathcal{G}_A\times \Z\to \mathcal{G}_A/\mathcal{H}_A=\Omega_A\times \Z$$
 are given by 
\[r(x,k,y,l)=(x,l)\quad\mbox{and}\quad d(x,k,y,l)=(x,k+l)\]
The groupoid multiplication in $\mathcal{G}_A\ltimes \mathcal{G}_A/\mathcal{H}_A$ is given by 
\[(x,k,y,l)(y,m,z,k+l)=(x,k+m,z,l)\]
The next Proposition follows by a standard argument, which is left to the reader.

\begin{prop}
\label{betactionfa}
The mapping 
$$\beta:\mathcal{G}_A\ltimes \mathcal{G}_A/\mathcal{H}_A\to \mathcal{G}_A\ltimes \mathcal{G}_A/\mathcal{H}_A, \quad (x,k,y,l)\mapsto (x,k,y,l-1),$$
is a groupoid automorphism. There is an isomorphism $C^*_r(\mathcal{G}_A\ltimes \mathcal{G}_A/\mathcal{H}_A)\cong O_A\rtimes U(1)$ under which $\beta$ corresponds to the dual $\Z$-action. In particular,
\[C^*_r(\mathcal{G}_A\ltimes \mathcal{G}_A/\mathcal{H}_A)\rtimes_\beta \Z\sim_M O_A.\]
\end{prop}

\begin{remark}
A consequence of Theorem \ref{cstarisoofsub} and Proposition \ref{betactionfa} is the well known fact that there is a $\Z$-action on $F_A\minotimes \mathbb{K}$ such that $O_A\minotimes \Ko\cong (F_A\minotimes \Ko)\rtimes \Z$.
\end{remark}

The restriction map $\rho: C_{r}^{*}(\mathcal{G}_{A})\rightarrow C^{*}_{r}(\mathcal{H}_{c})$ is a conditional expectation. The associated Hilbert $C^*$-module is denoted $\mathpzc{E}^{c}$. Under the isomorphism $O_A\cong C^*_r(\mathcal{G}_A)$, $\rho$ corresponds to the map $E: O_{A}\rightarrow F_{A}$ defined by
\[E(a):=\frac{1}{2\pi}\int_{U(1)}\alpha_{t}(a)\rd t.\]
The mapping $E$ defines an $F_{A}$-valued inner product on $O_{A}$. The completion $\mathpzc{E}^{\alpha}$ of $O_{A}$ in the norm associated to this inner product is a $\Z$-graded Hilbert $C^*$-module over $F_{A}$.

\begin{prop}
The isomorphism $C^{*}_{r}(\mathcal{G}_A)\cong O_A$ is $U(1)$-equivariant and induces a $\Z$-graded isomorphism $\mathpzc{E}^{c}\xrightarrow{\sim} \mathpzc{E}^{\alpha}$. 
\end{prop}

\subsubsection{The quantum group $SU_q(2)$}
\label{suqtwoexample}

Consider the matrix $A=\begin{pmatrix} 1& 1\\0&1\end{pmatrix}$. The partial isometries $S_1$ and $S_2$ generating $O_A$ satisfies the relations
$$S_1^*S_2=0,\quad S_2S_2^*=S_2^*S_2,\quad S_1S_1^*+S_2S_2^*=1\quad \mbox{and}\quad S_2^*S_2=1.$$
This condition guarantees that $O_A\cong C(SU_q(2))$ for any $q\in [0,1)$, see more in \cite{hongszym}. The compact quantum group $SU_q(2)$ is well studied and we merely describe it here as an interesting example. We do not derive anything new. Any admissible sequence $\mu\in \mathcal{V}_A$ has the form
\[\mu=11\cdots 122\cdots 2,\]
that is, if the letter $2$ appears in a word, all subsequent letters will be $2$:s. We will identify a point $(k,l)\in \N^2$ with the finite word consisting of $k$ occurrences of $1$ followed by $l$ occurrences of $2$. It holds that 
\begin{equation}
\label{phicomputation}
\phi(l)=\#\{\mu\in \mathcal{V}_A:|\mu|=l\}=l+1.
\end{equation}

\begin{prop}
\label{suqtwosfixedpointalg}
There is an isomorphism $C(SU_q(2))^{U(1)}\cong \tilde{\Ko}$ -- the unitalization of the compact operators on a separable infinite dimensional Hilbert space.
\end{prop}

\begin{proof}
We use the notation $A=\begin{pmatrix} 1& 1\\0&1\end{pmatrix}$ and $F_A=C(SU_q(2))^{U(1)}$. In light of the identification $\mathcal{V}_A=\N^2$, it holds that 
\[F_A^k\cong \C \oplus M_k(\C).\]
The first summand is spanned by $S_{(k+1,0)}S_{(k+1,0)} ^*$ and the second summand spanned by $S_\mu S_\nu^*$ where $\mu$ and $\nu$ are of length $k+1$ and not ending in $1$. Since 
\begin{align*}
S_{(k,0)}S_{(k,0)} ^*&=S_{(k+1,0)}S_{(k+1,0)} ^*+S_{(k,1)}S_{(k,1)} ^*\\
&\mbox{and}\quad S_{(k,l)}S_{(k',l')} ^*=S_{(k,l+1)}S_{(k',l'+1)} ^*, \quad \mbox{for}\quad l,l'>0,
\end{align*}
the embedding of the second factors $M_l(\C)\to M_{l+1}(\C)$ is a corner embedding. Hence the mappings $\C \oplus M_l(\C)\hookrightarrow \C \oplus M_{l+1}(\C)$ are unital. It follows that $\varinjlim \C \oplus M_l(\C)\cong \tilde{\Ko}$.
\end{proof}

\vspace{3mm}

\Large 
\section{Finite summability of Fredholm modules}
\label{sectionfinisumandputnamkam}
\normalsize

In this section we investigate the finite summability of odd $K$-homology classes on Cuntz-Krieger algebras. The central idea when treating the $K$-homology of Cuntz-Krieger algebras is the usage of Kaminker-Putnam's Poincar\'e duality class for Cuntz-Krieger algebras. After recalling its construction we will prove the following Theorem:

\setcounter{thm}{0}

\begin{thm}
\label{finsumthm}
Any class in $K^1(O_A)$ admits a $p$-summable representative for any $p> 0$.
\end{thm}

To be precise, we prove that any class in $K^1(O_A)$ can be represented by a $K$-cycle that is finite rank summable on the $*$-algebra generated by the generators of $O_A$. We return to the proof of this theorem in the end of Subsection \ref{subsectionfinsuminkon}. In the proof, we need to make use of $KK$-theory. The reader unfamiliar with $KK$-theory is referred to the textbook \cite{jeto} or Kasparov's original papers \cite{Kas1,Kas2}. We use the notation $\minotimes$ for the minimal tensor product of $C^*$-algebras.

\subsection{Kaminker-Putnam's Poincar\'e duality class}
\label{recokpsubsec}

Whenever $\mu\in \cup_{k\in \N}\{1,\ldots, N\}^k$, we let $\delta_\mu\in \ell^2(\mathcal{V}_A)$ denote the delta function in $\mu$ if $\mu\in \mathcal{V}_A$ and $\delta_\mu=0$ if $\mu\notin\mathcal{V}_A$. We obtain an ON-basis $\{\delta_\mu|\,\mu\in \mathcal{V}_A\}$ for $\ell^2(\mathcal{V}_A)$. We use the notation $e_1,\ldots, e_N$ for the standard ON-basis of $\C^N$. If $\mu=\mu_1\cdots\mu_k\in \{1,\ldots, N\}^k$, we use the notation $e_\mu:= e_{\mu_1}\otimes \cdots \otimes e_{\mu_k}\in (\C^{N})^{\otimes k}$. Let $\Fg$ denote the Hilbert space completion of $\oplus_{k=0}^\infty (\C^{N})^{\otimes k}$, with $(\C^N)^{\otimes 0}=\C$, in the scalar product 
$$\langle e_\mu,e_\nu \rangle_\Fg=\delta_{\mu,\nu}.$$ 
There is a natural isometric embedding $\ell^2(\mathcal{V}_A)\to \Fg$ whose range is the closed linear span of the set $\{e_\mu|\,\mu\in \mathcal{V}_A\}$. We often identify $\ell^2(\mathcal{V}_A)$ with its image under this embedding; that is, we identify $e_\mu$ with $\delta_\mu$ if $\mu\in \mathcal{V}_A$. We also let $P_A:\Fg\to \ell^2(\mathcal{V}_A)$ denote the orthogonal projection; in particular $P_Ae_\mu=\delta_\mu$ for any finite word $\mu$. Define the bounded operators 
$$L_i^A:\ell^2(\mathcal{V}_A)\to \ell^2(\mathcal{V}_A), \quad \delta_\mu=e_\mu\mapsto P_A(e_{i\mu})=\delta_{i\mu}.$$

There is a bijection of sets $\mathcal{V}_A\to \mathcal{V}_{A^T}$ given by 
$$\mu=\mu_1\mu_2\cdots \mu_{k-1} \mu_k\;\mapsto\; \bar{\mu}:=\mu_k\mu_{k-1}\cdots \mu_2\mu_1,$$
i.e. the word $\mu$ ordered in the opposite way. We define the unitary isomorphism 
$$J_\mathcal{V}:\ell^2(\mathcal{V}_A)\to \ell^2(\mathcal{V}_{A^T}), \quad \delta_\mu\mapsto \delta_{\bar{\mu}}.$$
Consider the operators $R_i^A:=J_\mathcal{V}^*L_i^{A^T}J_\mathcal{V}$, which act as $R^A_i\delta_\mu=\delta_{\mu i}$.

We let $\{S_i|i=1,\ldots, N\}$ and $\{T_i|i=1,\ldots,N\}$ denote the generators of $O_A$ and respectively $O_{A^T}$. We define the $*$-homomorphisms 
\begin{align*}
\beta_A&:=O_{A}\to \mathcal{C}(\ell^2(\mathcal{V}_A)), \quad S_i\mapsto L_i^A\mod \Ko(\ell^2(\mathcal{V}_A))\qquad \mbox{and}\\
\beta_A^T&:=Ad(q(J_\mathcal{V}))(\beta_{A^T}):O_{A^T}\to \mathcal{C}(\ell^2(\mathcal{V}_A)), \quad T_i\mapsto R_i^A\mod \Ko(\ell^2(\mathcal{V}_A)).
\end{align*}
Here $q:\Bo(\ell^2(\mathcal{V}_A),\ell^2(\mathcal{V}_{A^T}))\to \Bo(\ell^2(\mathcal{V}_A),\ell^2(\mathcal{V}_{A^T}))/\Ko(\ell^2(\mathcal{V}_A),\ell^2(\mathcal{V}_{A^T}))$ denotes the quotient mapping. The fact that $\beta_A$ is a $*$-homomorphism for any $A$ is shown in \cite{kaminkerputnam}; it also follows from Lemma \ref{sionell2} and Proposition \ref{lisandsis} below. A short computation shows that 
\begin{equation}
\label{lirirelations}
[L_i^A,R_j^A]=0\quad \mbox{and}\quad [(L_i^A)^*,R_j^A]=\delta_{i,j} P_{\circ_A},
\end{equation}
where $P_{\circ_A}$ denotes the orthogonal projection onto $\C\delta_{\circ_A}$. See more in \cite[Proposition $4.2$]{kaminkerputnam}. It follows that the algebra $\beta_A(O_A)$ commutes with $\beta_A^T(O_{A^T})$ in $\mathcal{C}(\ell^2(\mathcal{V}_A))$. Since $O_A$ and $O_{A^T}$ are nuclear we obtain a $*$-homomorphism
\[\beta_{KP}:=\beta_A\minotimes \beta^T_A:O_A\minotimes O_{A^T}\to \mathcal{C}(\ell^2(\mathcal{V}_A)).\]

By standard constructions, see \cite{jeto}, the $*$-homomorphism $\beta_{KP}$ induces a class 
$$[\beta_{KP}]\in \Ext(O_A\minotimes O_{A^T},\Ko(\ell^2(\mathcal{V}_A)))$$ 
represented by the extension
\begin{equation}
\label{sesoaoat}
0\to \Ko(\ell^2(\mathcal{V}_A))\to E_{KP}\to O_A\minotimes O_{A^T}\to 0,
\end{equation}
where 
\[E_{KP}:=\{a\oplus T\in O_A\otimes O_{A^T}\oplus \Bo(\ell^2(\mathcal{V}_A)): \;\beta_{KP}(a)=T\mod \Ko(\ell^2(\mathcal{V}_A))\in \im \beta_{KP}\}.\]
If $\beta_{KP}$ is injective, for instance if $O_A\otimes O_{A^T}$ is simple, $E_{KP}$ is the $C^*$-algebra generated by $L_i^A$ and $R_i^A$ and the exactness of \eqref{sesoaoat} was in this case verified in the paragraph proceeding \cite[Definition $4.3$]{kaminkerputnam}. The $C^*$-algebras $O_A$ and $O_{A^T}$ are nuclear, so any element in the semi group $\Ext(O_A\minotimes O_{A^T},\Ko(\ell^2(\mathcal{V}_A)))$ is invertible, and
\[\Ext(O_A\minotimes O_{A^T},\Ko(\ell^2(\mathcal{V}_A)))\cong K^1(O_A\minotimes O_{A^T}).\]
This isomorphism can be found in \cite[Chapter $3.3$]{jeto}. The construction of this isomorphism relies on the Choi-Effros Theorem and on the Stinespring Theorem. The two theorems combined guarantee the existence of a completely positive splitting of the short exact sequence \eqref{sesoaoat} that has the following form. There is a Hilbert space $\He$, a representation $\pi:O_A\minotimes O_{A^T}\to \Bo(\He)$ and an isometry $W:\ell^2(\mathcal{V}_A)\to \He$ such that 
\begin{equation}
\label{betaeqwq}
\beta_{KP}(a)=q(W^*\pi(a)W)\quad\mbox{for any}\quad a\in O_A\minotimes O_{A^T}.
\end{equation}
It follows from the fact that $\beta_{KP}$ is a $*$-homomorphism that $[WW^*,\pi(a)]\in \Ko(\He)$ for all $a\in O_A\minotimes O_{A^T}$. The identity \eqref{betaeqwq} guarantees that the image of $[\beta_{KP}]$ under $\Ext (O_A\minotimes O_{A^T},\Ko(\ell^2(\mathcal{V}_A)))\to K^1(O_A\minotimes O_{A^T})$ is represented by the odd analytic $K$-cycle $(\pi,\He,2WW^*-1)$. The data $\pi$, $\He$ and $W$ is difficult to construct in general. Further, the problem of finite summability on a dense subalgebra is not made easier by the abstract construction from the Stinespring Theorem. We will return to this problem in the next section. First we recall the construction of Poincar\'e duality from the image $\Delta\in K^1(O_A\minotimes O_{A^T})$ of the extension class $[\beta_{KP}]$.

\begin{thm}[Consequence of \cite{kaminkerputnam}]
\label{pd}
The mapping 
\[K_*(O_{A^T})\mapsto K^{*+1}(O_A),\quad [e]\mapsto (1_{O_A}\otimes [e])\otimes_{O_A\minotimes O_{A^T}} \Delta\] 
is an isomorphism.
\end{thm}

\begin{remark}
\label{kaspprodremark}
Recall that $KK$-theory comes with a product; for separable $C^*$-algebras $A$, $B$ and $C$, there is a $\Z/2\Z$-graded operation
$$\otimes_B:KK_*(A,B)\otimes KK_*(B,C)\to KK_*(A,C),$$
called the Kasparov product. This product is associative. As such, one often considers $KK$-theory from the perspective of defining an additive category\footnote{It even carries a triangulated structure, see \cite{neme}.} whose objects are the separable $C^*$-algebras and the group of morphisms from $A$ to $B$ is $KK_0(A,B)$ with the composition of morphisms given by the Kasparov product. Further, it coincides with the index pairing 
$$K_*(B)\otimes K^*(B)\to \Z$$
when $A=C=\C$, once identifying $KK_*(\C,B)\cong K_*(B)$, $KK_*(B,\C)=K^*(B)$ and $KK_*(\C,\C)\cong \Z$. The particular Kasparov product $(1_{O_A}\otimes [e])\otimes_{O_A\otimes O_{A^T}} \Delta$ used in Theorem \ref{pd} is that between the class $1_{O_A}\otimes [e]\in KK_*(O_A,O_A\minotimes O_{A^T})$ and $\Delta\in KK_1(O_A\minotimes O_{A^T},\C)$. See more in \cite{emermeyer,kaminkerputnam}.
\end{remark}

In order to use Theorem \ref{pd}, we will need to compute Kasparov products in the case described in Remark \ref{kaspprodremark}. Computations of this type are well known to experts in the field, we include them for the sake of completeness. Throughout this subsection, $A$ and $B$ denote unital $C^*$-algebras and $(\pi,\He,F)$ an odd analytic $K$-cycle for $A\minotimes B$.

\begin{prop} 
\label{psumprop}
Let $e\in B \minotimes M_m(\C)=M_m(B)$ be a projection and set 
\[\He_e:=[\pi\otimes \id_{M_m(\C)}](1_A\otimes e)(\He\otimes \C^m).\] 
There is an odd analytic $K$-cycle $(\pi_e,\He_e,F_e)$ on $A$ defined by
\begin{align*}
\pi_e:A&\to \Bo(\He_e),\quad a\mapsto [\pi\otimes \id_{M_m(\C)}](a\otimes e),\\
&\mbox{and}\quad F_e:=[\pi\otimes \id_{M_m(\C)}](1_A\otimes e)\cdot [F\otimes \id_{\C^m}]\cdot [\pi\otimes \id_{M_m(\C)}](1_A\otimes e).
\end{align*}
\end{prop}

\begin{proof}
We assume $m=1$ to shorten notation. Since $F$ commutes with $\pi(1_A\otimes e)$ up to compacts, 
\[F_e^2-\pi(1_A\otimes e)F^2\pi(1_A\otimes e)\in \Ko(\He_e).\]
Since $F^2-1$ is compact, so is $F_e^2-1$. Furthermore $F_e^*=\pi(1_A\otimes e)F^*\pi(1_A\otimes e)$ so $F^*_e-F_e\in \Ko(\He_e)$. Finally, we have for any $a\in A$ that 
\begin{align*}
[F_e,\pi_e(a)]&=[\pi(1_A\otimes e)F\pi(1_A\otimes e),\pi(a\otimes e)]=\\
&=\pi(1_A\otimes e)[F,\pi(a\otimes 1_B)]\pi(1_A\otimes e)\in \Ko(\He_e).
\end{align*}
\end{proof}

\begin{lem}
\label{kproduct}
If $e\in B\otimes M_m(\C)$ is a projection, the Kasparov product $(1_{A}\otimes [e])\otimes_{A\minotimes B} [\pi,\He,F]$ can be represented by the Fredholm module $(\pi_e,\He_e,F_e)$ (see notation in Proposition \ref{psumprop}).
\end{lem}

\begin{proof}
The $K$-theory class $1_{A}\otimes [e]$ can be represented by the $A-A\minotimes B$ Kasparov module $(A\minotimes eB^m,0)$ with its obvious $A$-action on the left and the structure of an $A\minotimes B$-Hilbert $C^*$-module comes from the inclusion $A\minotimes eB^m\subseteq A\minotimes B^m$. It is clear that as $A-\C$-Hilbert $C^*$-modules
\[\He_e=(A\minotimes eB^m)\otimes_{A\minotimes B} \He.\]

Since $(\pi_e,\He_e,F_e)$ is a Fredholm module on the right Hilbert space, to verify that it is a Kasparov product between $[\pi,\He,F]$ and $(A\minotimes eB^m,0)$ it suffices to prove that $F_e$ is an $F$-connection, see \cite[Definition $2.2.4$]{jeto}. The other conditions on a Kasparov product automatically hold as the Kasparov operator in $(A\minotimes eB^m,0)$ is $0$, see \cite[Definition $2.2.7$]{jeto}. Recall that $F_e$ is an $F$-connection if for $x\in A\minotimes eB^m$, the linear mapping 
\[\xi\mapsto x\otimes_{A\minotimes B} (F\xi)-F_e (x\otimes_{A\minotimes B}\xi)\]
is compact. However, since $(1_A\otimes e)x=x$ this fact follows from the identity
\begin{align*}
x\otimes_{A\minotimes B} (F\xi)-F_e (x\otimes_{A\minotimes B}\xi)&=\pi(x)F\xi-\pi(1_A\otimes e)F\pi(1_A\otimes e)\pi(x)\xi\\
&=\pi(1_A\otimes e)[\pi(x),F]\xi.
\end{align*}
\end{proof}

\begin{remark}
\label{kaspprodbusby}
The natural mapping $K^1(A)\to \Ext (A,\Ko)$ is defined by mapping a cycle $x:=(\pi,\He,F)$ to the extension associated with the Busby invariant
\[\beta_F:A\to \mathcal{C}(\He), \quad\beta_F(a):=q(P_F\pi(x)P_F)\quad \mbox{where}\quad P_F:=(F+1)/2\]
and $q:\Bo(\He)\to \mathcal{C}(\He)$ denotes the quotient mapping. If $F^2=1$, the Hilbert space can be reduced to $P_F\He=\ker(F-1)$. The Busby invariant $\beta_F$ is degenerately equivalent to $\tilde{\beta}_F:A\to \mathcal{C}(P_F\He)$, $\beta_F(a):=q(P_F\pi(x)P_F)$. In particular, the Busby invariant of the $K$-cycle $(\pi_e,\He_e,F_e)$ constructed in Lemma \ref{kproduct} is $\beta_e(a):=\beta_F(a\otimes e)$.
\end{remark}

We end this subsection with a proposition on finite summability concerning Poincar\'e dualities whose proof is carried out mutatis mutandis to that of Proposition \ref{psumprop}. We let $\mathcal{I}$ denote a symmetrically normed operator ideal, see \cite[Chapter $1.7$]{simon}. Assume that $\mathcal{A}\subseteq A$ and $\mathcal{B}\subseteq B$ are unital dense $*$-subalgebras. 

\begin{prop}
\label{finsum}
Let $e\in \mathcal{B}\otimes M_m(\C)$ be a projection and assume that $(\pi,\He,F)$ is $\mathcal{I}$-summable on the $*$-subalgebra $\mathcal{A}\otimes (\C1_B+\C e)\subseteq A\minotimes B$, then $(\pi_e,\He_e,F_e)$ is $\mathcal{I}$-summable on $\mathcal{A}$. 
\end{prop}

\begin{remark}
We note the following important consequence of Proposition \ref{finsum}. Assume that $(\pi,\He,F)$ is $\mathcal{I}$-summable on $\mathcal{A}\otimes^{alg}\mathcal{B}$. Then any element in the image of the mapping
\[K_0(B)\mapsto K^1(A), \quad x\mapsto (1_{A}\otimes x)\otimes_{A\minotimes B} [\pi,\He,F],\]
is $\mathcal{I}$-summable on $\mathcal{A}$. A slight modification of the argument above implies that the same holds true for elements in the image of the analogously defined mapping $K_1(B)\to K^0(A)$. This fact follows from \cite[Proposition $3.12$]{cuntzblackadar} which allows us to assume $\mathcal{B}\subseteq B$ to be holomorphically closed, and the mapping $K_*(\mathcal{B})\to K_*(B)$ induced from the inclusion $\mathcal{B}\hookrightarrow B$ to be an isomorphism.
\end{remark}

\subsection{Finite summability in $K^1(O_A)$}
\label{subsectionfinsuminkon}

To deal with finite summability problems for $O_A$ we note an important relation between linear splittings and finite summability based on \cite{extensionsgoffeng}. The observation will reduce the problem of finite summability for an odd $K$-homology class to finding such $\pi$, $\He$ and $W$ described above, in the paragraph preceding Theorem \ref{pd}, that behaves well on generators. Whenever $\{x_i\}_{i\in I}$ is a set of elements in a $*$-algebra, we let $\C^*[x_i|i\in i]$ denote the $*$-algebra generated by $\{x_i\}_{i\in I}$. 

\begin{thm}
\label{splittingsandsum}
Let $\mathcal{I},\mathcal{J}\subseteq \Bo$ be symmetrically normed operator ideals such that $a^*a\in \mathcal{I}$ implies $a\in \mathcal{J}$. Suppose that 
\[0\to \Ko(\He_0)\to E\to A\to 0\]
is a short exact sequence of $C^*$-algebras with Busby invariant $\beta_E$. Assume the following:
\begin{enumerate}
\item The $C^*$-algebra $A$ contains a dense $*$-subalgebra generated by a set $\{x_i\}_{i\in I}\subseteq A$, where $I$ is an index set, and that there is a set $\{X_i\}_{i\in I}\subseteq \Bo(\He_0)$ of pre images of $\{\beta_E(x_i)\}_{i\in I}$ under the quotient mapping $q:\Bo(\He_0)\to \mathcal{C}(\He_0)$ such that the mapping 
\[\C^*[x_i|i\in I]\to \Bo(\He_0)/\mathcal{I}(\He_0), \quad x_i\mapsto X_i\mod \mathcal{I}(\He_0),\]
is a well defined $*$-homomorphism. 
\item There is a Hilbert space $\He$, a $*$-representation $\pi:A\to \Bo(\He)$ and an isometry $W:\He_0\to\He$ such that
\[X_i-W^*\pi(x_i)W\in \mathcal{I}(\He_0).\]
\end{enumerate}
Then $[\beta_E]$ defines an invertible class in $\Ext (A,\Ko(\He_0))$ whose image in $K^1(A)$ is represented by the $K$-cycle $(\pi, \He,2WW^*-1)$ which is $\mathcal{J}$-summable on the dense $*$-subalgebra $\C^*[x_i|i\in I]\subseteq A$.
\end{thm} 

The proof is closely modeled on the structure in the refined extension invariant of \cite{extensionsgoffeng} that is adapted for extensions of Schatten class ideals. Compare to for instance \cite[Theorem $3.2$]{extensionsgoffeng}. The examples of ideals to keep in mind is the finitely summable case $\mathcal{I}=\mathcal{L}^p$ and $\mathcal{J}=\mathcal{L}^{2p}$ or the $\theta$-summable case $\mathcal{I}=\mathrm{Li}$ and $\mathcal{J}=\mathrm{Li}_{1/2}$.

\begin{proof}
It follows by the construction of the isomorphism $\Ext (A,\Ko)^{-1}\cong K^1(A)$ that $[\beta_E]$ is represented by the $K$-cycle $(\pi, \He,2WW^*-1)$, see the discussion before Theorem \ref{pd}. The $\mathcal{J}$-summability statement requires a more subtle algebraic analysis. 

To simplify notation, we set $\mathcal{A}:=\C^*[x_i|i\in I]$. We can define a linear mapping $\tau:\mathcal{A}\to \Bo(\He_0)$, $a\mapsto W^*\pi(a)W$. The assumptions of the Lemma guarantees that we can define the $*$-algebra
\[\mathcal{E}:=\{(a,T)\in \mathcal{A}\oplus \Bo(\He_0):\tau(a)-T\in \mathcal{I}(\He_0)\}.\]
There is a natural mapping $\sigma_\mathcal{E}:\mathcal{E}\to \mathcal{A}$ given by $(a,T)\mapsto a$ which admits a linear splitting $\tilde{\tau}(a):=(a,\tau(a))$. 

The mapping $\tau$ induces a $*$-homomorphism $\beta_\mathcal{E}:\mathcal{A}\to \Bo(\He_0)/\mathcal{I}(\He_0)$. The pullback of the universal $\mathcal{I}$-summable extension along $\beta_\mathcal{E}$ places $\mathcal{E}$ in a commuting diagram of $*$-algebras with exact rows:
\[
\begin{CD}
0  @>>> \mathcal{I}(\He_0) @>>>\mathcal{E} @>\sigma_\mathcal{E}>>  \mathcal{A}@>>>0 \\
@. @|  @VVV @ VV\beta_\mathcal{E} V@.\\
0  @>>>  \mathcal{I}(\He_0) @>>>\Bo(\He_0) @>>> \Bo(\He_0)/\mathcal{I}(\He_0) @>>>0 \\
\end{CD}, \]
where $\mathcal{E}\to \Bo(\He_0)$ is defined by $(a,T)\mapsto T$. 

We set $P:=WW^*$. The operator $U:=W^*|_{P\He}:P\He\to \He_0$ is a unitary isomorphism. We now turn to the $*$-algebra $\hat{\mathcal{E}}$ defined from the diagram 
\[\begin{CD}
0  @>>>\mathcal{I}(\He_0) @>>>\mathcal{E} @>\sigma_\mathcal{E}>> \mathcal{A}@>>>0 \\
@. @VAd(U) VV  @VV Ad(U)V @|@.\\
0  @>>> \mathcal{I}(P\He)@>>>\hat{\mathcal{E}} @> >> \mathcal{A}@>>>0 \\
\end{CD} \]
By construction, the linear mapping $\hat{\tau}(a):=P\pi(a)P=U^*\tau(a)U\in \hat{\mathcal{E}}$ defines a splitting of the lower row. In particular, for any $a,b\in \mathcal{A}$ it holds that 
\[ \hat{\tau}(ab)-\hat{\tau}(a)\hat{\tau}(b)\in \mathcal{I}(P\He).\]
It follows that $[P,\pi(a)]\in \mathcal{J}(\He)$ for all $a\in \mathcal{A}$ by an algebraic manipulation, see \cite[Lemma $3.7$]{extensionsgoffeng}.
\end{proof}

\begin{remark}
If the mapping $\beta_\mathcal{E}:\mathcal{A}\to \Bo(\He_0)/\mathcal{I}(\He_0)$ in the proof of Theorem \ref{splittingsandsum} is injective, the mapping $\mathcal{E}\to \Bo(\He_0)$ is injective. Hence there is an isomorphism of $*$-algebras
\[\mathcal{E}\cong \{T\in \Bo(\He): T\mod\mathcal{I}(\He_0)\in \im \beta_\mathcal{E}\}.\]
\end{remark}

Let us return to the $C^*$-algebra $O_A$. Recall the definition of the KMS-state $\phi_A$ on $O_A$ from \eqref{definitionkmsstatephia},   the associated GNS-space $L^2(O_A,\phi_A)$ and the fundamental representation $\pi_A$. By the results of Subsection \ref{subsectionckalgebras}, there is an isomorphism $L^2(O_A,\phi_A)\cong L^2(\mathcal{G}_A)$ intertwining the $O_A$-action with the $C^*(\mathcal{G}_A)$-action under the isomorphism $O_A\cong C^*(\mathcal{G}_A)$ from Theorem \ref{oaandcga}. 

Fix a finite admissible word $\lambda\in \mathcal{V}_A$. Define $\He_\lambda$ as the closed linear span of all elements $S_{\mu\lambda}\in L^2(O_A,\phi_A)$. For any two finite words $\mu,\nu\in \mathcal{V}_A$, Proposition \ref{sscomp} implies that 
\[\langle S_\mu,S_\nu\rangle_{L^2(O_A,\phi_A)}=\phi_A(S_\mu^*S_\nu)=\delta_{\mu,\nu}\phi_A(S_{\mu_k}^*S_{\mu_k})=\delta_{\mu,\nu}\sum_{j=1}^N A_{\mu_k,j}\mathrm{vol}(C_j),\]
where $k:=|\mu|$. For any finite word $\mu=\mu_1\cdots \mu_k$, admissible or not, we set 
$$c_\mu:=\left(\sum_{j=1}^N A_{\mu_k,j}\mathrm{vol}(C_j)\right)^{-1/2}.$$
In particular, it holds that $c_\mu$ only depends on the \emph{last} letter of $\mu$. It follows from the computation above that the non-zero elements of $\{c_{\mu\lambda} S_{\mu\lambda}|\mu\in \mathcal{V}_A\}$ form an ON-basis for $\He_\lambda$. Let $\ell^2(\mathcal{V}_\lambda)\subset \ell^{2}(\mathcal{V}_{A})$ be the closed subspace spanned by the basis vectors associated with the words 
\[\mathcal{V}_{\lambda}:=\{\mu\lambda\in \mathcal{V}_A|\mu\in \mathcal{V}_A\}.\]
The operator $\mathfrak{P}_\lambda:=R^A_{\bar{\lambda}} (R^A_{\bar{\lambda}})^*\in \Bo(\ell^2(\mathcal{V}_A))$ is the orthogonal projection onto $\ell^2(\mathcal{V}_\lambda)$. We define the unitary isomorphism $U: \ell^2(\mathcal{V}_A)\to \He_\circ$ by $\delta_\mu\mapsto c_\mu S_\mu$. It follows from the above discussion that the map
\[W_{\lambda}:=U\mathfrak{P}_\lambda:\ell^{2}(\mathcal{V}_A)\rightarrow L^{2}(O_{A},\phi_{A}),\]
is a partial isometry. We recapitulate by noting that the partial isometry $W_\lambda$ maps $\delta_\mu$ to $c_\mu S_\mu$ if $\mu \in \mathcal{V}_\lambda$ and it maps $\delta_\mu$ to $0$ if $\mu\notin\mathcal{V}_\lambda$. Hence, the image of $W_\lambda$ is $\He_\lambda$ and the image of $W_\lambda^*$ is $\ell^2(\mathcal{V}_\lambda)$.

\begin{prop}
\label{wprop}
The partial isometry $W_\lambda$ satisfies the equation: 
\[W_\lambda^*\pi_A(S_i )W_\lambda=L_i^AW_\lambda ^*W_\lambda\equiv L_i^AR^A_{\bar{\lambda}} (R^A_{\bar{\lambda}})^*,\quad i=1,\ldots, N.\]
\end{prop}

\begin{proof}
It suffices to prove that $W_\lambda^*\pi_A(S_i )W_\lambda \delta_{\mu\lambda}=L_i ^A\delta_{\mu\lambda}$ since the vectors $ \delta_{\mu\lambda}$ spans the range of $W_\lambda^*$. A direct computation goes as follows:
\[W_\lambda^*\pi_A(S_i )W_\lambda \delta_{\mu\lambda}=c_{\mu\lambda}W^*_\lambda \pi_A(S_i )S_{\mu\lambda}=c_{\mu\lambda}W_\lambda ^*S_{i\mu\lambda}=\delta_{i\mu\lambda}=L_i^A\delta_{\mu\lambda},\]
since $c_{\mu\lambda}=c_{i\mu\lambda}$.
\end{proof}

\begin{remark}
\label{measureremark} 
The orthogonal projection $W_{\lambda}W_{\lambda}^{*}\in \Bo(L^2(O_A,\phi_A))$ onto $\He_\lambda$ correspond to a projection constructed in the groupoid picture as follows. For any finite word $\mu$, $S_{\mu}$ corresponds to the characteristic function of the set 
\[\{(x,|\mu|,\sigma^{|\mu|}(x))\in \mathcal{G}_A|x\in C_\mu\}.\]
It holds that 
\begin{equation}
\label{measureproj} 
\mathfrak{Q}_{\lambda}f:=\sum_{\mu\in \mathcal{V}_A}c_{\mu\lambda}^2S_{\mu\lambda}\int_{\Omega_{A}} \rho(S_{\mu\lambda}^**f)\rd\mu_{A},\end{equation}
with $\rho$ as in \eqref{rho},
defines a projection in $L^{2}(\mathcal{G}_{A},\mu_{A})$ corresponding to $W_{\lambda}W_{\lambda}^{*}$ under the isomorphism $L^{2}(O_{A},\phi)\cong L^{2}(\mathcal{G}_{A},\mu_{A})$. This is akin to the constructions in \cite{EN2}.
\end{remark}

\begin{prop}
\label{representingbetai}
The extension defined from the Busby invariant 
$$\beta_i:O_A\to \mathcal{C}(\ell^2(\mathcal{V}_A)),\quad a\mapsto \beta_{KP}(a\otimes T_iT_i^*)$$ 
can be represented by the odd analytic $K$-cycle $(\pi_A,L^2(O_A,\phi_A),2W_iW_i^*-1)$ which is $p$-summable for any $p>0$ on the dense $*$-subalgebra $\C^*[S_1,S_2,\ldots, S_N]\subseteq O_A$.
\end{prop}

\begin{proof}
For $j=1,\ldots, n$, the operators $X_j:=L_j^AR_i^A(R_i^A)^*$ lifts $\beta_i(S_j)$. The operators $X_j$ satisfy the Cuntz-Krieger relations modulo finite rank operators, so $S_j\mapsto X_j\mod \mathcal{L}^p(\ell^2(\mathcal{V}_A))$ defines a $*$-homomorphism  $\C^*[S_1,S_2,\ldots, S_N]\to \Bo(\ell^2(\mathcal{V}_A))/ \mathcal{L}^p(\ell^2(\mathcal{V}_A))$ for any $p> 0$, even modulo finite rank operators. It also holds that $W_i^*W_i=R_i^A(R_i^A)^*$. In particular, $X_j=W_i^*\pi_A(S_j )W_i$. Hence, the Proposition follows from Theorem \ref{splittingsandsum}.
\end{proof}

We recall the following description of $K_0(O_{A^T})$ from \cite[Proposition $3.1$]{Cuntz2}.

\begin{prop}[Proposition $3.1$ of \cite{Cuntz2}]
\label{kzerodescription}
The mapping 
$$\Z^N\to K_0(O_{A^T}),\quad (k_j)_{j=1}^N\mapsto \sum_{j=1}^N k_j[T_jT_j^*]$$ 
is surjective with kernel being $(1-A)\Z^N$
\end{prop}

\begin{proof}[Proof of Theorem \ref{finsumthm}]
By Theorem \ref{pd}, Remark \ref{kaspprodbusby} and Proposition \ref{kzerodescription} any $K$-homology class on $O_A$ can be represented by an extension class of the form $\sum_{j=1}^N k_j[\beta_j]$. The Theorem follows from Proposition \ref{representingbetai}.
\end{proof}

\subsection{A representative for $\Delta$} 
\label{representativefordeltasusection}

As previously indicated (see Proposition \ref{finsum}), any summability property of a $K$-cycle representative for $\Delta$ would carry over to any $K$-homology class for $O_A$. The problem is to represent $\Delta$ in a reasonable way. We will in this subsection construct a $\theta$-summable representative for $\Delta$.

We will use the notation $\He_0^T$ for the closed linear span of $\{T_{\bar{\mu}}|\mu\in \mathcal{V}_{A}\}$ in $L^2(O_{A^T},\phi_{A^T})$. Just as for $O_A$, there are constants $c_\mu^T>0$ only depending on the first word of $\mu$ such that $\{c_\mu^T T_{\bar{\mu}}|\mu\in\mathcal{V}_{A}\}$ forms an ON-basis for $\He_0^T$. Define the linear mapping $\mathcal{W}_0$ by 
$$\mathcal{W}_{0}: \ell^{2}(\mathcal{V}_{A})\rightarrow \He_{0}\otimes\He_{0}^{T},\quad \delta_{\lambda}\mapsto \sum_{\mu\nu=\lambda}(|\lambda|+1)^{-\frac{1}{2}}c_{\mu}S_{\mu}\otimes c_{\nu}^{T}T_{\overline{\nu}},$$
whose adjoint equals 
$$\mathcal{W}_0^*:\He_0\otimes \He_0^T\to \ell^2(\mathcal{V}_A),\quad c_\mu c_\nu^T S_\mu\otimes T_{\bar{\nu}}\mapsto (|\mu\nu|+1)^{-1/2}\delta_{\mu\nu}$$
The operator $\mathcal{W}_0$ is an isometry since
$$\mathcal{W}_0^*\mathcal{W}_0\delta_\lambda=\sum_{\mu\nu=\lambda}(|\mu\nu|+1)^{-1}\delta_\lambda=\delta_\lambda.$$
We will make use of the isometry 
$$\mathcal{W}:\ell^2(\mathcal{V}_A)\to L^2(O_A,\phi_A)\otimes L^2(O_{A^T},\phi_{A^T})$$ 
that is defined as the composition of $\mathcal{W}_0$ and with the isometric inclusion $\He_0\otimes \He_0^T\hookrightarrow L^2(O_A,\phi_A)\otimes L^2(O_{A^T},\phi_{A^T})$. 

\begin{lem}
\label{coisometrydelta}
For any $A$ it holds that 
$$L_i^A-\mathcal{W}^*[\pi_A(S_i)\otimes 1_{O_{A^T}}]\mathcal{W}, \; R_i^A-\mathcal{W}^*[1_{O_A}\otimes\pi_{A^T}(T_i)]\mathcal{W}\in \mathrm{Li}(\ell^2(\mathcal{V}_A)).$$
If there is a $p>0$ such that $\phi(l)\lesssim l^p$, it holds that 
$$L_i^A-\mathcal{W}^*[\pi_A(S_i)\otimes 1_{O_{A^T}}]\mathcal{W}, \; R_i^A-\mathcal{W}^*[1_{O_A}\otimes\pi_{A^T}(T_i)]\mathcal{W}\in \mathcal{L}^{p+1,\infty}(\ell^2(\mathcal{V}_A)).$$
\end{lem}

\begin{proof}
This is yet another proof by computation. Choose a finite word $\lambda\in \mathcal{V}_A$. It holds that
\begin{align*}
\mathcal{W}^*[\pi_A(S_i)\otimes 1_{O_{A^T}}]\mathcal{W}\delta_\lambda&=\mathcal{W}^*\left(\sum_{\mu\nu=\lambda} c_\mu c_\nu^T (|\mu\nu|+1)^{-1/2} S_{i\mu}\otimes T_{\bar{\nu}}\right)=\\
&=\frac{(|\lambda|+1)^{-1/2}}{(|i\lambda|+1)^{-1/2}}\delta_{i\lambda}=L_i^A\delta_\lambda+\left(\sqrt{\frac{|\lambda|+2}{|\lambda|+1}}-1\right)L_i^A\delta_{\lambda},\\
\mathcal{W}^*[1_{O_A}\otimes\pi_{A^T}(T_i)]\mathcal{W}\delta_\lambda&=\mathcal{W}^*\left(\sum_{\mu\nu=\lambda} c_\mu c_\nu^T (|\mu\nu|+1)^{-1/2}  S_{\mu}\otimes T_{i\bar{\nu}}\right)=\\
&=\frac{(|\lambda|+1)^{-1/2}}{(|i\lambda|+1)^{-1/2}}\delta_{\lambda i}=R_i^A\delta_\lambda+\left(\sqrt{\frac{|\lambda|+2}{|\lambda|+1}}-1\right)R_i^A\delta_{\lambda},
\end{align*}
since $c_\mu$ only depend on the last letter of $\mu$ and $c^T_\nu$ only depend on the first letter of $\nu$. We define $\Gamma\in \Bo(\ell^2(\mathcal{V}_A))$ by 
$$\Gamma\delta_\lambda := \left(\sqrt{\frac{|\lambda|+2}{|\lambda|+1}}-1\right)\delta_\lambda$$
and reformulate the above identities as
\[\mathcal{W}^*[\pi_A(S_i)\otimes 1_{O_{A^T}}]\mathcal{W}-L_i^A=L_i^A\Gamma \quad \mbox{and}\quad \mathcal{W}^*[1_{O_A}\otimes\pi_{A^T}(T_i)]\mathcal{W}-R_i^A=R_i^A\Gamma.\]
 We recall the elementary asymptotics
\[\sqrt{\frac{|\lambda|+2}{|\lambda|+1}}-1=\frac{1}{2|\lambda|}+O\left(\frac{1}{|\lambda|^2}\right),\quad \mbox{as}\quad |\lambda|\to \infty.\]
It holds in general that $\phi(l)\lesssim \e^{\;s l}$ for $s>\delta_A$ by Corollary \ref{phiasymptotics}, so $\Gamma\in \mathrm{Li}(\ell^2(\mathcal{V}_A))$. On the other hand $\phi(l)\lesssim l^p$ implies $\Gamma\in \mathcal{L}^{p+1,\infty}(\ell^2(\mathcal{V}_A))$.
\end{proof}

From Theorem \ref{splittingsandsum} and Lemma \ref{coisometrydelta} we may conclude a summability result for the duality class $\Delta$. This result is by no means a surprise. There is to the authors' knowledge no known counter examples to the $\theta$-summability problem for unbounded Fredholm modules, so in effect there are no counter examples to representing $K$-homology classes by $\theta$-summable Fredholm modules, cf. \cite[Chapter IV.$8.\alpha$, Theorem 4]{Connesbook}. We nevertheless state it as a Theorem.

\begin{thm}
\label{thetaandfinitekp}
The class $\Delta\in K^1(O_A\minotimes O_{A^T})$ is represented by the analytic $K$-cycle 
$$(\pi_A\otimes \pi_{A^T}, L^2(O_A,\phi_A)\otimes L^2(O_{A^T},\phi_{A^T}),2\mathcal{W}\mathcal{W}^*-1),$$
which is $\theta$-summable on the dense $*$-subalgebra of $O_A\minotimes O_{A^T}$ generated by $S_i\otimes 1$ and $1\otimes T_i$ for $i=1,\ldots, n$. If there is a $p>0$ such that $\phi(l)\lesssim l^p$, this is a $\mathcal{L}^{p+1,\infty}$-summable $K$-cycle.
\end{thm}

Theorem \ref{thetaandfinitekp} and Equation \eqref{phicomputation} imply that any $K$-homology class on $SU_q(2)$ is finitely summable.

\vspace{3mm}

\Large
\section{Unbounded $(O_A,F_A)$-cycles}
\label{oafasection}
\normalsize

We will in this section start approaching the problem of constructing unbounded Fredholm modules on $O_A$. It is natural to try and construct unbounded Fredholm modules as Kasparov products of bivariant cycles with unbounded Fredholm modules on one of the subalgebras $C(\Omega_A)\subseteq F_A\subseteq O_A$.  In this section, we will consider classes in $KK_{1}(O_{A},F_A)$. In Section \ref{oacasection} we construct classes in $KK_{1}(O_{A},C(\Omega_{A}))$. Both constructions provide cycles that behave analogously to those studied in Section \ref{sectionfinisumandputnamkam} apart from the difficulties of being bivariant. A problem with using the fixed point algebra is that, despite there being a well studied bivariant $(O_A,F_A)$-cycle that is naturally constructed, the unbounded Fredholm modules on $F_A$ are more difficult to construct and understand topologically than those on $C(\Omega_A)$. E.g. for the Cuntz algebra $O_N$, it holds that the even $K$-homology of the fixed point algebra vanishes but its odd $K$-homology is an uncountable group  (cf. Proposition \ref{thealgebraonandkhom}). 

We will in the remaining parts of the paper use a great deal of unbounded $KK$-theory. The reader unfamiliar with this material is referred to \cite{BJ, BMS, AKH, KaLe,Kuc, Mes}.

\subsection{The homogeneous components}

In this subsection, we describe the structure of the module $\mathpzc{E}^{c}$ -- the completion of $O_A$ as the pre-$F_A$-Hilbert $C^*$-module associated with the conditional expectation $\rho_{c}:O_{A}\rightarrow F_{A}$ coming from the restriction mapping $C_{c}(\mathcal{G}_{A})\rightarrow C_{c}(\mathcal{H}_A)$. It is clear that $\mathpzc{E}^{c}$ decomposes as a direct sum of $F$-modules:
\[\mathpzc{E}^{c}=\bigoplus_{n\in\Z}\mathpzc{E}_{n}^{c},\]
corresponding to the disjoint union decomposition $\mathcal{G}=\bigcup_{n\in\Z}\mathcal{G}_{n}$, where $\mathcal{G}_{n}=c^{-1}_A(n)$. We show below that each $\mathpzc{E}_{n}^c$ is a finitely generated projective $F_A$-module, and consequently $\mathpzc{E}^c$ is isomorphic to a direct sum of finitely generated projective $F_A$-modules. Since $\mathpzc{E}^c_0=F_A$ it suffices to consider $n\neq 0$.

\begin{lem}\label{positiven}
Let $n>0$. The column vectors $v_{n}:=(S_\mu^{*})_{|\mu|=n}\in \mathrm{Hom}^*_{F_A}(\mathpzc{E}^c_n,F_A^{\phi(n)})$ have the property that $v_{n}^{*}v_{n}=1$. In particular, $\mathpzc{E}^c_n$ is a finitely generated projective $F_A$-module for $n>0$.
\end{lem}

\begin{proof} 
We have
\[v_{n}^{*}v_{n}=\sum_{|\mu|=n}S_\mu S_\mu^*=1,\]
which follows from successively applying the relation \eqref{eq2}.

For an element $a\in O_A$ of degree $n$, the vector $v_{n}a$, constructed by coordinatewise multiplication by $a$, is an element of $F^{\phi(n)}_A$. Therefore, for positive $n$, the map 
\[\begin{split} \mathpzc{E}^c_{n}&\rightarrow F_A^{\phi(n)}\\
a&\mapsto v_{n}a,\end{split}\]
is an isometry onto its image; its image is equal to $p_{n}F^{\phi(n)}_A$, with $p_{n}:=v_{n}v_{n}^{*}$. Hence, $\mathpzc{E}^c_{n}$ is a finitely generated projective $F_A$-module.
\end{proof}

Recall the notation $P_j$ for the projections $S_{j}S_{j}^{*}$. The projections $P_j$ are of degree $0$, and $S_{i}P_{j}$ is of degree $1$ for any $i,j$.  Recall that we assume that neither row nor column of $A$ is composed of only zeroes. Hence the numbers
\[N_{j}:=\sum_{i=1}^{N}A_{ij},\] 
satisfy $0<N_{j}\leq N$. For two finite words $\mu$ and $\nu$, not necessarily admissible, of the same length $n>0$ we set
$$R_{\mu,\nu}:=\frac{1}{\sqrt{N_{\nu_{1}}\cdots N_{\nu_{n}}}}S_{\mu_{1}}P_{\nu_{1}}\cdots S_{\mu_{n}}P_{\nu_{n}}$$
We use the notation $\tilde{\phi}(n):=\#\{(\mu,\nu):|\mu|=|\nu|=n,\; R_{\mu,\nu}\neq 0\}$. It is clear that $\tilde{\phi}(n)\leq \phi(n)^2$.

\begin{lem}\label{negativen}
Let $n>0$. The column vectors $w_{n}:=(R_{\mu,\nu})_{\mu=|\nu|=n}\in \mathrm{Hom}_{F_A}^*(\mathpzc{E}^c_{-n},F_A^{\tilde{\phi}(n)})$ have the property that $w_{n}^{*}w_{n}=1$.  In particular, $\mathpzc{E}^c_{-n}$ is a finitely generated projective $F_A$-module for $n>0$.
\end{lem}

\begin{proof}
We have
\[\begin{split}
w_{n}^{*}w_{n}&=\sum_{|\mu|=|\nu|=n}\frac{1}{N_{\nu_{1}}\cdots N_{\nu_{n}}}P_{\nu_n}S_{\mu_{n}}^{*}\cdots P_{\nu_{1}}S_{\mu_{1}}^{*}S_{\mu_{1}}P_{\nu_{1}}\cdots S_{\mu_{n}}P_{\nu_{n}}\\
&=\sum_{\mu_{n},\nu_{n}=1}^{N}\frac{1}{N_{\nu_{n}}}P_{\nu_{n}}S_{\mu_{n}}^{*}\left(\sum_{|\mu'|=|\nu'|=n-1}\frac{1}{N_{\nu'_{1}}\cdots N_{\nu'_{n-1}}}\cdots P_{\nu_{1}}S_{\mu_{1}}^{*}S_{\mu_{1}}P_{\nu_{1}}\cdots\right)S_{\mu_{n}}P_{\nu_{n}}
\end{split}\]
\[=\sum_{\mu_{n},\nu_{n}=1}^{N}\frac{1}{N_{\nu_{n}}}P_{\nu_{n}}S_{\mu_{n}}^{*}w_{n-1}^*w_{n-1}S_{\mu_{n}}P_{\nu_{n}}.
\]
Hence, the result follows by induction once proven for $n=1$. In that case, the equation becomes
\[\begin{split}w_{1}^{*}w_{1}&=\sum_{i,j=1}^{N}\frac{1}{N_{j}}P_{j}S_{i}^{*}S_{i}P_{j}=\sum_{i,j=1}^{N}\left(\sum_{\ell=1}^{N}\frac{1}{N_{j}}A_{i\ell}P_{j}S_{\ell}S_{\ell}^{*}P_{j}\right)\quad\textnormal{by \eqref{eq1}}\\
&=\sum_{i,j=1}^{N}\frac{1}{N_{j}}A_{ij}S_{j}S_{j}^{*} \quad\textnormal{by \eqref{eq3}}\\
&=\sum_{j}S_{j}S_{j}^{*}=1\quad\textnormal{by \eqref{eq2}}.\end{split}\]
\end{proof}

\subsection{The gauge cycle}

By \cite{Mes1}, pointwise multiplication by the cocycle $c_A$ induces a selfadjoint regular operator $D_c$ on the Hilbert $C^*$-module $\mathpzc{E}^{c}$, giving $(\mathpzc{E}^c,D_c)$ the structure of an odd unbounded $KK$-cycle for $KK_{1}(O_{A},F_A)$. We will refer to this cycle as the \emph{gauge cycle}. The construction of the gauge cycle was considered in a more general setup in \cite{careyetal11} that we make use of in the next subsection. We assume that $A$ is a $C^*$-algebra with a strongly continuous $U(1)$-action satisfying the spectral subspace assumption \cite[Definition $2.2$]{careyetal11}. The gauge action on $O_A$ satisfies the spectral subspace condition because, as we saw above, the graded components of $O_A$ form finitely generated projective $F_A$-modules. 

We let $F\subseteq A$ denote the fixed point algebra for the $U(1)$-action. There is a positive expectation value $E:A\to F$ given by $a\mapsto \frac{1}{2\pi}\int_0^{2\pi} \e^{i\theta}(a)\rd \theta$, this expectation coincides with $\rho_c$ for $O_A$. After completion of $A$ with respect to the associated $F$-valued scalar product we obtain an $A-F$-Hilbert $C^*$-module that we denote by $\mathpzc{E}^R$. We can also define the operator
\[D_R y=i\frac{d}{d\theta} \left(\e^{-i\theta}.y\right)|_{\theta=0},\]
which is a densely defined $F$-linear operator on $\mathpzc{E}^R$. Since $U(1)$ is abelian, $D_R$ commutes with the circle action on $\mathpzc{E}^R$ giving a $U(1)$-equivariant operator.

\begin{prop}
Whenever the $U(1)$-action on $A$ satisfies the spectral subspace assumption, the pair $(D_R,\,\mathpzc{E}^R)$ forms a $U(1)$-equivariant unbounded $(A,F)$-Kasparov module.
\end{prop}

For a proof, see \cite[Proposition $2.9$]{careyetal11}. We can construct $KK$-cycles as in Section \ref{sectionfinisumandputnamkam}. A difference here is that one has to work with partial isometries in Hilbert $C^*$-modules.

\begin{prop}
\label{Fredholmgauge} 
The $F_A$-linear adjointable mapping
\[v: \ell^{2}(\mathcal{V}_{A})\otimes F_{A}\rightarrow \mathpzc{E}^{c},\quad\mbox{defined by}\quad v:\delta_{\mu}\otimes a \mapsto S_{\mu}a,\]
is a partial isometry and the projection $vv^{*}\in \End_{F_A}^*(\mathpzc{E}^c)$ has compact commutators with $O_{A}$. It consequently defines a $U(1)$-equivariant $(O_{A},F_{A})$-Kasparov module $(\mathpzc{E}^c,2vv^*-1)$ whose class in $KK_1^{U(1)}(O_A,F_A)$ coincides with the class $[\mathpzc{E}^{c},D_c]$ of the gauge cycle $(\mathpzc{E}^{c},D_{c})$.
\end{prop}

\begin{proof} 
Observe that $v$ is adjointable by Lemma \ref{positiven} and \ref{negativen}. It is clear that $v$ is a partial isometry because the elements $S_{\mu}$ are mutually orthogonal in the module $\mathpzc{E}^{c}$ and $v^{*}(S_{\mu}a)=\delta_{\mu}\otimes S^{*}_{\mu}S_{\mu}a$, so both $v^{*}v$ and $vv^{*}$ are projections. The statement that the partial isometry defines a Kasparov module is proved as in the previous section. To see that $vv^{*}$ defines the gauge cycle, one only needs to observe that it is exactly the projection onto the positively graded part of the module $\mathpzc{E}^{c}$.
\end{proof}

\begin{remark}
\label{Fredholmgaugedouble} 
In a similar way as in Proposition \ref{Fredholmgauge}, we can define a partial isometry 
\[\begin{split}
w: \ell^{2}(\mathcal{V}_{A})\otimes F_{A}\otimes F_{A^{T}}&\rightarrow \mathpzc{E}^{c}_{A}\otimes\mathpzc{E}_{A^{T}}^c,\\
 \delta_{\mu}\otimes a\otimes b &\mapsto \sum_{\lambda\nu=\mu} \frac{1}{\sqrt{|\mu|+1}}S_{\mu}a\otimes T_{\overline{\nu}}b.
 \end{split}\]
It can be proven, in the same way as in Proposition \ref{Fredholmgauge}, that the projection $ww^{*}$ has compact commutators with $O_{A}\minotimes O_{A^{T}}$. Consequently we obtain an odd $U(1)$-equivariant $(O_{A}\minotimes O_{A^{T}},F_{A}\minotimes F_{A^{T}})$-Kasparov module. Compare to the construction of Subsection \ref{representativefordeltasusection}. 
\end{remark}

\subsection{The Pimsner-Voiculescu sequence for the Cuntz-Krieger algebra}

What we wish to do in this section is to relate the cohomological properties of the Cuntz-Krieger algebra with the fixed point algebra. The standard procedure, found in \cite{Cuntz2} for instance, is to apply the Pimsner-Voiculescu sequence. In this section we briefly recall the proof of the Pimsner-Voiculescu sequence in $KK$ following \cite{cumero} and prove that the gauge cycle appears as the boundary mapping. We summarize the results of this subsection in the following Theorem:

\begin{thm}
\label{pimsnervoic}
The gauge element $[\mathpzc{E}^{c},D_c]\in KK_1(O_A,F_A)$, the $\Z$-action $\beta$ on $\Ko\minotimes F_A$ of Proposition \ref{betactionfa} and the inclusion $\iota:F_A\to O_A$ fits into a distinguished triangle in $KK$:
\[\xymatrix@C=0.8em@R=2.71em{
F_A \ar[rr]^{1-\beta}& & F_A \ar[dl]^\iota\\
  &O_A. \ar[ul]|-\circ^{[\mathpzc{E}^{c},D_c]}
}\]
\end{thm}

The triangulated structure of $KK$ is explained in \cite{neme}; a distinguished triangle is a triangle isomorphic in $KK$ to a semi split short exact sequence of $C^*$-algebras. In practice, it ensures that for any separable $C^*$-algebra $D$ there are the following six term exact sequences:
\[\begin{CD}
KK_0(D,F_A)@>1-\beta_*>>KK_0(D,F_A) @>\iota_*>> KK_0(D,O_A)\\
@A-\otimes [\mathpzc{E}^{c},D_c]AA  @. @VV-\otimes [\mathpzc{E}^{c},D_c]V\\
KK_1(D,O_A) @<\iota_*<<KK_1(D,F_A)@<1-\beta_*<< KK_1(D,F_A)\\
\end{CD} \]
\[\begin{CD}
KK_0(O_A,D)@>\iota^*>>KK_0(F_A,D)@>1-\beta^*>>KK_0(F_A,D) \\
@A[\mathpzc{E}^{c},D_c]\otimes -AA  @. @VV[\mathpzc{E}^{c},D_c]\otimes -V\\
KK_1(F_A,D) @<1-\beta^*<<KK_1(F_A,D)@<\iota^*<<KK_1(O_A,D)\\
\end{CD} \]

The Pimsner-Voiculescu sequence can be derived in many ways. We will here consider the Toeplitz extension approach due to Cuntz. Assume that $B$ is a unital $C^*$-algebra and that $\beta$ is an automorphism of $B$. The restriction that $B$ is unital makes the semantics easier, but can be lifted. Let $\mathcal{T}(B)$ denote the $C^*$-algebra generated by $B$ and an isometry $v_B$ satisfying the relation
\[v_Bbv_B^*=\beta(b).\]
One can represent $\mathcal{T}(B)$ in $\mathrm{End}^*_B(\oplus _{k=0}^\infty B)$ by extending the mappings
\[B\ni b\mapsto \oplus_k \beta^k(b)\in\mathrm{End}^*_B(\oplus _{k=0}^\infty B)\quad\mbox{and}\quad v_B(x_k)_{k\in \N}:=(x_{k-1})_{k\in \N}.\]
There is a $U(1)$-action on $\mathcal{T}(B)$ induced from the grading on $\oplus _{k=0}^\infty B$, i.e. the $U(1)$-action is defined from $z(v_B):=zv_B$. 

Let us realize $B\rtimes \Z$ as the universal $C^*$-algebra generated by $B$ and a unitary $u_B$ satisfying 
\[u_Bbu_B^*=\beta(b).\]
There is a $*$-homomorphism $\sigma_B:\mathcal{T}(B)\to B\rtimes \Z$ given by extending $v_B\mapsto u_B$. Since $\sigma_B$ respects the grading it is clear that $\sigma_B$ is $U(1)$-equivariant with respect to the dual $U(1)$-action on $B\rtimes \Z$. There is an isomorphism of right $B$-Hilbert $C^*$-modules $\oplus _{k=0}^\infty B\cong \ell^2(\N)\otimes B$.

\begin{lem}
The morphism $\sigma_B$ is well defined and fits into a $U(1)$-equivariant semisplit short exact sequence
\begin{equation}
\label{sestoeplitz}
0\to \Ko_B(\ell^2(\N)\otimes B)\to \mathcal{T}(B)\xrightarrow{\sigma_B} B\rtimes \Z\to 0.
\end{equation}
\end{lem}

\begin{proof}
We can identify $B\rtimes \Z$ with the $C^*$-subalgebra of $\mathrm{End}^*_B(\ell^2(\Z)\otimes B)$ generated by the image of 
\[B\ni b\mapsto \oplus_k \beta^k(b)\in \mathrm{End}^*_B(\oplus _{k\in \Z} B)\quad\mbox{and the unitary}\quad u_B(x_k)_{k\in \Z}:=(x_{k-1})_{k\in \Z}.\]
Let $P:\ell^2(\Z)\otimes B\to \ell^2(\N)\otimes B$ denote the orthogonal projection. The adjointable operator $P$ is $U(1)$-equivariant. It is clear that the $B$-linear mapping 
\[\mathcal{T}:B\rtimes \Z\to  \mathcal{T}(B), \; b\mapsto PbP\]
is a $U(1)$-equivariant completely positive splitting of $\sigma_B$. Let 
\[\mathfrak{q}:\mathrm{End}^*_B(\ell^2(\N)\otimes B)\to \mathcal{Q}_B( \ell^2(\N)\otimes B):=\mathrm{End}^*_B(\ell^2(\N)\otimes B)/\Ko_B(\ell^2(\N)\otimes B)\]
denote the quotient mapping. Once we prove that $\mathfrak{q}\circ\mathcal{T}$ is a $*$-homomorphism, the Lemma follows. The operator $P$ commutes with the $B$-action on $ \ell^2(\Z)\otimes B$. Furthermore, if we let $e_k$ denote the standard basis for $\ell^2(\Z)$, then
\[[P,u_B](e_{k}\otimes x)=
\begin{cases}
0, \; k\neq -1\\
e_{0}\otimes x, \;k=-1.
\end{cases}\]
In particular, $[P,u_B]\in \Ko(\ell^2(\Z))\otimes 1_B\subseteq \Ko_B(\ell^2(\Z)\otimes B)$. It follows that $[P,b]\in \Ko_B(\ell^2(\Z)\otimes B)$ for any $b\in B\rtimes \Z$. Hence $\mathfrak{q}\circ\mathcal{T}$ is a $*$-homomorphism.
\end{proof}

\begin{lem}
\label{toeplitzhomotopy}
There is a $U(1)$-equivariant homotopy $\mathcal{T}(B)\sim_h B$ with trivial $U(1)$-action on $B$.
\end{lem}

For a proof, see \cite{cuntztoeplitz}. Let $\iota_B:B\to B\rtimes \Z$ denote the embedding.

\begin{cor}
The morphism $[\mathcal{T}]\in KK^{U(1)}_1(B\rtimes \Z, B)$ defined from the invertible extension class \eqref{sestoeplitz} fits into a distinguished triangle in $KK^{U(1)}$:
\[\xymatrix@C=0.8em@R=2.71em{
B\ar[rr]^{1-\beta} & & B \ar[dl]^{\iota_B}\\
  &B\rtimes \Z, \ar[ul]|-\circ^{[\mathcal{T}]}
}\]
using the homotopy of Proposition \ref{toeplitzhomotopy} and the Morita equivalence $\Ko_B(\ell^2(\N)\otimes B)\sim_M B$. 
\end{cor}

As a consequence of the Corollary, after setting $B=F_A\minotimes \Ko$ and equipping it with its dual $\Z$-action coming from $F_A\minotimes \Ko\cong O_A\rtimes U(1)\minotimes \Ko$, what remains to prove of Theorem \ref{pimsnervoic} is to show that $[\mathcal{T}]$ coincides with the gauge element in $KK_1(O_A,F_A)$ after Takesaki-Takai duality $O_A\minotimes \Ko\cong (F_A\minotimes \Ko)\rtimes \Z$. We first construct an unbounded representative for the Pimsner-Voiculescu element $[\mathcal{T}]\in KK^{U(1)}_1(B\rtimes \Z, B)$.\newline

Before constructing this, let us make a series of minor remarks placing the algebra above in a more analytic framework. The Fourier transform induces an isomorphism $L^2(S^1)\cong \ell^2(\Z)$ which in turn produces an isomorphism $C(S^1)\cong C^*(\Z)$ intertwining the pointwise action of the former with the left regular representation of the latter. The image of $\ell^2(\N)$ under the Fourier transform is $H^2(S^1)$ -- the Hardy space consisting of functions in $L^2(S^1)$ with a holomorphic extension to the interior of $S^1\subseteq \C$. The analogy of the projection $P$ in this picture is the projection of $L^2(S^1)\otimes B$ onto those $B$-valued functions on $S^1$ with a holomorphic extension to the interior. We note that 
\[L^{2}(S^{1})\otimes B\cong \left\{(b_k)_{k\in \Z}\in \prod_{k\in \Z} B\,\bigg|\sum_{k\in \Z} b_k^*b_k<\infty\right\}.\]

\begin{prop}
\label{al2asclosure}
There is a natural unitary $U(1)$-equivariant isomorphism of $B\rtimes \Z-B$-Hilbert $C^*$-modules
\[L^2(S^1)\otimes B\cong \overline{B\rtimes \Z},\]
where the closure is taken in $B$-valued scalar product $\langle a,b\rangle:=E(a^*b)$. 
\end{prop}

We define 
\[W^{1,2}(S^1,B):=\left\{(b_k)_{k\in \Z}\in \prod_{k\in \Z} B:\sum_k k^2b_k^*b_k<\infty\right\},\]
and the $U(1)$-equivariant $B$-linear unbounded operator $D_{B\rtimes \Z}$ on $ L^2(S^1)\otimes B$ on an elementary tensor by
\[D_{B\rtimes \Z}(e_k\otimes x):=k e_k\otimes x\]
and extending it to the domain $W^{1,2}(S^1,B)$ by continuity. If $y\in W^{1,2}(S^1,B)$ then 
\[D_{B\rtimes \Z}y=i\frac{d}{d\theta} \left(\e^{-i\theta}.y\right)|_{\theta=0}.\]

\begin{prop}
\label{dazoperator}
The operator $D_{B\rtimes \Z}$ with domain $W^{1,2}(S^1,B)$ gives a $U(1)$-equivariant unbounded $(B\rtimes \Z,B)$-cycle .
\end{prop}

For a proof, see \cite[Section $2$]{careyetal11}.

\begin{lem}
The bounded transform of $D_{B\rtimes \Z}$ is a compact perturbation of the $(B\rtimes\Z,B)$-Kasparov module $(L^2(S^1)\otimes B, 2P-1)$. Especially; $[D_{B\rtimes \Z}]=[\mathcal{T}]\in KK_1^{U(1)}(B\rtimes \Z,B)$.
\end{lem}

\begin{proof}
We have that 
\[D_{B\rtimes \Z}(1+D_{B\rtimes \Z}^2)^{-1/2}(e_k \otimes x):=k(1+k^2)^{-1/2} e_k\otimes x.\]
Since $k(1+k^2)^{-1/2}-\mathrm{sign}(k)\sim -(2k)^{-1}$ as $|k|\to \infty$ and $B\rtimes \Z$ satisfies the spectral subspace condition (see \cite[Definition $2.2$]{careyetal11}), the Proposition follows from \cite[Lemma $2.4$]{careyetal11}.
\end{proof}

We again turn our attention to the gauge cycle. It is possible to, in the Pimsner-Voiculescu sequence of the Cuntz-Krieger algebra, replace the Toeplitz element $[\mathcal{T}]$ of $F_A\minotimes\Ko$ by the gauge cycle for $O_A$. Recall its definition from above. We denote the class associated with the gauge cycle $(\mathpzc{E}^R,D_R)$ of a $U(1)-C^*$-algebra $A$ satisfying the spectral subspace condition by $[\mathpzc{E}^R,D_R]\in KK_1^{U(1)}(A,F)$. We note the following Proposition whose proof is left to the reader.

\begin{prop}
\label{moritarennie}
The image of $[\mathpzc{E}^R,D_R]$ under the isomorphism
\[j_{\Ko(L^2(S^1))}:KK_1^{U(1)}(A,F)\to KK_1^{U(1)}(\Ko(L^2(S^1))\minotimes A,F)\] 
associated with the $U(1)$-equivariant Morita equivalence $A\sim_M \Ko(L^2(S^1))\minotimes A$, where the right hand side is equipped with the diagonal $U(1)$-action, coincides with the class of 
$$( L^2(S^1)\otimes \mathpzc{E}^R,\,\id_{L^2(S^1)}\otimes D_{R})$$ 
where the $(\Ko(L^2(S^1))\minotimes A,F)$-Hilbert $C^*$-module $L^{2}(S^{1})\otimes \mathpzc{E}^R$ is equipped with the diagonal $U(1)$-action.
\end{prop}

The $A$-action on $\mathpzc{E}^R$ is by construction equivariant, hence there is an action of $A\rtimes U(1)$ on $\mathpzc{E}^R$. By \cite[Lemma $2.4.$ii]{careyetal11}, the spectral subspace assumption guarantees that this action induces a $*$-homomorphism $A\rtimes U(1) \to \Ko_F(\mathpzc{E}^R)$. We let $\mathpzc{G}$ denote the associated $(A\rtimes U(1),F)$-Hilbert $C^*$-module. We denote the associated class in $KK$-theory by $[\mathpzc{G}]\in KK_0^{U(1)}(A\rtimes U(1),F)$. Let us remark that $\mathpzc{G}=\,\mathpzc{E}^R$ as Banach spaces and as $(A,F)$-bimodules we have the equality
\begin{equation}
\label{geiso}
A\otimes_A\mathpzc{G}=\,\mathpzc{E}^R.
\end{equation}

\begin{lem}
If $A$ is a $U(1)-C^*$-algebra satisfying the spectral subspace assumption, and $F:=A^{U(1)}$, then 
\[j_{\Ko(L^2(S^1))}[\mathpzc{E}^R,D_R]=[D_{A\rtimes U(1)\rtimes \Z}]\otimes_{A\rtimes U(1)} [\mathpzc{G}]\quad\mbox{in}\quad KK_1^{U(1)}(\Ko(L^2(S^1))\minotimes A,F),\]
where $[D_{A\rtimes U(1)\rtimes \Z}]\in KK_1^{U(1)}( \Ko(L^2(S^1))\minotimes A,A\rtimes U(1))$ is the element constructed for the $\Z-C^*$-algebra $A\rtimes U(1)$ as in Proposition \ref{dazoperator}.
\end{lem}

\begin{proof}
To simplify notation, we set $B:=A\rtimes U(1)$ which is a $\Z-C^*$-algebra in its dual action. Using that $\left(L^2(S^1)\otimes B\right)\otimes_{B}\,\mathpzc{G}\cong L^2(S^1)\otimes \,\mathpzc{G}$, the class $[D_{B\rtimes \Z}]\otimes _B[\mathpzc{G}] \in KK_1(B\rtimes \Z,F)$ can be represented by the $U(1)$-equivariant $(B\rtimes \Z,F)$-Kasparov cycle 
\[\left( L^2(S^1)\otimes \,\mathpzc{G},D_{B\rtimes \Z}\otimes_{B} \mathrm{id}_{\mathpzc{G}} \right).\]

Define the unitary $U\in \mathrm{End}^*_F( L^{2}(S^{1})\otimes\mathpzc{G})$ by representing the unitary $U_0\in \mathcal{M}(C^*(U(1))\minotimes C(U(1)))$ which is defined as an operator on $L^2(U(1))\otimes L^2(U(1))$ via
\[U_0f(g,h)=f(gh,h).\]
The unitary $U$ implements Takesaki-Takai duality giving an isomorphism of $(B\rtimes \Z,F)$-Hilbert $C^*$-modules
\[L^2(S^1)\otimes\mathpzc{G}\cong L^{2}(S^{1})\otimes\,\mathpzc{E}^R\]
where the left hand carries the structure of a $(B\rtimes \Z,F)$-Hilbert $C^*$-module under Takesaki-Takai duality $B\rtimes \Z\cong \Ko(L^2(S^1))\minotimes A$ and the $U(1)$-action is diagonal. The Lemma now follows from Proposition \ref{moritarennie}.
\end{proof}

\begin{cor}
Under the mapping $KK_1(O_A,F_A)\to KK_1((\Ko\minotimes F_A)\rtimes \Z,\Ko\minotimes F_A)$ induced from the Morita equivalence $O_A\sim_M (\Ko\minotimes F_A)\rtimes \Z$, the gauge element $[D_R]$ is mapped to the Toeplitz element $[D_{(\Ko\minotimes F_A)\rtimes \Z}]$.
\end{cor}

This Corollary follows directly from that $\mathpzc{G}$ is an imprimitivity bimodule implementing the Morita equivalence $O_A\rtimes U(1)\sim_M F_A$, see Proposition \ref{betactionfa}. In general, we can conclude the following Corollary which implies Theorem \ref{pimsnervoic}.

\begin{cor}
If $A$ is a $U(1)-C^*$-algebra satisfying the spectral subspace assumption and $\mathpzc{G}$ is a Morita equivalence, the following triangle is distinguished in $KK^{U(1)}$
\[\xymatrix@C=0.8em@R=2.71em{
F\ar[rr]^{1-\beta} & & F \ar[dl]^{\iota_{F}}\\
  &A, \ar[ul]|-\circ^{[\mathpzc{E}^R,D_R]}
}\]
where $\iota_{F}:F\to B$ denotes the inclusion and $\beta\in KK_0(F,F)$ is Morita equivalent to the $\Z$-action dual to the $U(1)$-action on $A$.
\end{cor}

\subsection{Computations and problems with the approach using the fixed point algebra}
\label{fixedcompsec}

In this subsection we will compute $K$-groups of some examples of Cuntz-Krieger algebras and their fixed point algebras. These computations are known, and are provided only as a basis for discussion regarding possibilities of constructing unbounded Fredholm modules with prescribed $K$-homology classes. In order to do so, we require a Proposition giving a general formula for the $K$-theory and $K$-homology of the fixed point algebra.

\begin{prop}
\label{ktheorymcomputations}
The $K$-theory groups of $F_A$ are given by
\[K_0(F_A)\cong\varinjlim (\Z^N,A^T) \quad\mbox{and}\quad K_1(F_A)\cong 0.\]
The $K$-homology groups of $F_A$ are given by
\[K^0(F_A)\cong \varprojlim A^i\Z^N \quad\mbox{and}\quad K^1(F_A)\cong \hat{\Z}^N_A/(\Z^N/\varprojlim A^i\Z^N).\]
Here $\hat{\Z}^N_A:=\varprojlim \Z^N/A^i\Z^N$ denotes the $A$-adic completion of $\Z^N$.
\end{prop}

The computation of the $K$-homology groups of the fixed point algebra might not be as well known as the corresponding result in $K$-theory so we will sketch the proof. For a detailed proof in a special case, we refer to the notes \cite{hadfield}. The proof relies on a result of Schochet-Rosenberg (see \cite[Theorem $1.14$]{rosscho}) stating that if $B=\varinjlim B_i$ there is a graded short exact sequence
\begin{equation}
\label{sesinverselimits}
0\to\varprojlim \,\!^1\, K^{*+1}(B_i)\to K^*(B)\to\varprojlim \, K^*(B_i)\to 0.
\end{equation}
We can directly conclude from Equation \eqref{sesinverselimits} and the $AF$-structure of $F_A$ that 
\[K^0(F_A)\cong\varprojlim\, (\Z^N,A) \quad\mbox{and}\quad K^1(F_A)\cong\varprojlim\,\!\!^1\, (\Z^N, A).\]
These isomorphisms are simplified further using the explicit construction of derived projective limits in the category of abelian groups, see for instance \cite[Chapter $3.5$]{weibel}.

\subsubsection{The algebra $O_N$}
\label{subsubonon}

The algebra $O_N$ (i.e. the Cuntz algebra), was recalled above in Subsubsection \ref{examplethealgebraon}. We let $F_N$ denote the fixed point algebra in $O_N$.

\begin{prop}
\label{thealgebraonandkhom}
It holds that 
$$K_0(F_N)\cong \Z\left[\frac{1}{N}\right],\quad K^1(F_N)\cong \Z_N/\Z\quad\mbox{and}\quad K_1(F_N)\cong K^0(F_N)\cong 0.$$
\end{prop}

Here we use the notation $\Z\left[\frac{1}{N}\right]$ for the ring generated by $\frac{1}{N}$ and $\Z_N$ for the $N$-adic completion of $\Z$.

\begin{proof}
We let $w:=(1,1,\ldots, 1)^T\in \Z^N$ and $\ell:=(1,1,\ldots, 1)\in \Hom(\Z^N,\Z)$. It holds that $A=w\otimes \ell$. For any $k\in \N_+$ and $x\in \Z^N$, $A^kx=N^{k-1}\ell(x)w$. Hence $K_0(F_N)\cong \varinjlim (\Z,N)= \Z[N^{-1}]$. Similarly, $K^0(F_N)\cong \varprojlim (\Z,N)=0$. It also follows that $\Z^N/A^k\Z^N=\Z^{N-1}\oplus \Z/N^{k-1}\Z$. Hence 
\[\hat{\Z}^N_A=\varprojlim \Z^N/A^k\Z^N=\Z^{N-1}\oplus \Z_N, \quad\mbox{so}\quad K^1(F_N)=(\Z^{N-1}\oplus \Z_N)/\Z^N=\Z_N/\Z.\]
\end{proof}

The isomorphism $K_0(F_N)\cong \Z\left[\frac{1}{N}\right]$ is implemented by the tracial state $\phi_N:F_N\to \C$ given by restricting the KMS-state on $O_N$ to $F_N$.

\begin{remark}
\label{ktheorycomputationsforthealgebraon}
The well known computations 
$$K_0(O_N)\cong K^1(O_N)\cong \Z/(N-1)\Z\quad\mbox{and}\quad K_1(O_N)\cong K^0(O_N)\cong 0,$$
follow from Proposition \ref{thealgebraonandkhom} and the Pimsner-Voiculescu sequence (Theorem \ref{pimsnervoic}). In particular, we arrive at the short exact sequence for the only non-vanishing $K$-homology group $K^1(O_N)$:
\[0\to K^1(O_N)\to \Z_N/\Z\to \Z_N/\Z\to 0.\]
It follows that the Kasparov product with the gauge class of $O_N$ on $K$-homology vanishes. 
\end{remark}

\subsubsection{The quantum group $SU_q(2)$}
Recall that $C(SU_q(2))$ is isomorphic to the Cuntz-Krieger algebra constructed from the matrix $A=\begin{pmatrix} 1& 1\\0&1\end{pmatrix}$, as in Subsubsection \ref{suqtwoexample}.

\begin{prop}
When $A=\begin{pmatrix} 1& 1\\0&1\end{pmatrix}$, it holds that
$$K_0(F_A)\cong K^0(F_A)\cong \Z^2\quad \mbox{and}\quad K_1(F_A)\cong K^1(F_A)\cong 0.$$
\end{prop}

This Proposition follows directly from Proposition \ref{suqtwosfixedpointalg} or the computation of the $K$-groups, in Proposition \ref{ktheorymcomputations}, since $A$ is invertible. The $K$-theory and $K$-homology for $O_A$ is in this case given by 
\begin{equation}
\label{kgroupsforsucue}
K_0(O_A)\cong K_1(O_A)\cong K^0(O_A)\cong K^1(O_A)\cong \Z,
\end{equation}
as can be seen from the Pimsner-Voiculescu sequences
\begin{align*}
0\to K_1(O_A)\xrightarrow{\otimes [D_c]} \Z^2\xrightarrow{\begin{pmatrix} 0& 0\\1&0\end{pmatrix}} \Z^2 \to K_0(O_A)\to 0,\\
0\to K^0(O_A)\to \Z^2\xrightarrow{\begin{pmatrix} 0& 1\\0&0\end{pmatrix}} \Z^2 \xrightarrow{ [D_c]\otimes} K^1(O_A)\to 0.
\end{align*}

\begin{remark}
\label{suq2rem}
We conclude that the Kasparov product with the gauge class surjects onto the odd $K$-homology group of $SU_q(2)$. As the fixed point algebra is the unitalization of the $C^*$-algebra of compact operators, see Proposition \ref{suqtwosfixedpointalg}, it admits unbounded Fredholm modules with both good analytic and topological properties. 
\end{remark}

\subsubsection{The crossed product $C(\partial F_d)\ltimes F_d$}
\label{freegroupexample}

Let $F_d$ denote the free group on $d$ generators that we denote by $\{\gamma_1,\ldots, \gamma_d\}$. The boundary of $F_d$ consists of infinite words in the alphabet given by the generators $\{\gamma_1,\ldots, \gamma_d,\gamma_1^{-1},\ldots, \gamma_d^{-1}\}$ subject to the condition that for any $i$, the letters $\gamma_i$ and $\gamma_i^{-1}$ cannot succeed each other. It is well known that the group $F_d$ act amenably on its boundary $\partial F_d$. Hence $C(\partial F_d)\ltimes F_d\cong C(\partial F_d)\ltimes_r F_d$.

\begin{prop}
\label{isofreegr}
The crossed product $C(\partial F_d)\ltimes F_d$ is a Cuntz-Krieger algebra $O_A$ such that $F_d=\mathcal{V}_A$ and $\partial F_d=\Omega_A$ where $A$ is the symmetric $2d\times 2d$-matrix consisting of $1$'s except for $2\times 2$-identity matrices on the $2\times 2$-diagonal,
\[A_{F_d}:=\begin{pmatrix} 
1&0&1&1&\cdots &1&1\\
0&1&1&1&\cdots &\vdots&\vdots\\
1&1&1&0&\cdots &\vdots&\vdots\\
1&1&0&1&\cdots &\vdots&\vdots\\
\vdots&\vdots&&&\ddots&\\
\vspace{1.5mm}
1&1&\cdots &&&1&0\\
1&1&1&\cdots&&0&1
\end{pmatrix}.\] 
\end{prop}

This result can be found in \cite[Section 2]{spielbergaren}. We just indicate how to prove it using groupoids. It suffices to provide an isomorphism of groupoids $\phi:\partial F_d\rtimes F_d\xrightarrow{\sim} \mathcal{G}_A$ for this specific choice of matrix $A$. Such an isomorphism is given by $\phi(x,\gamma):=(x,n(x,\gamma),x\gamma)$ where 
$$n(x,\gamma)=|\gamma|-2\ell(x,\gamma),$$ 
here $|\gamma|$ is the word length of $\gamma$ and $\ell(x,\gamma)$ is the number of reductions necessary in $x\gamma$ to write it in reduced form. It is well defined because $A$ guarantees that any word $x\in \Omega_A$ corresponds to a reduced word in $\partial F_d$.

\begin{prop}
\label{freecomputations}
It holds that 
\begin{align*}
K^*\left((C(\partial F_d)\ltimes F_d)^{U(1)}\right)\cong &
\begin{cases}
\Z^{2d-1}, \;&*=0,\\
\Z_{2d-1},\;&*=1
\end{cases}
\quad\mbox{and}\\
&K^*\left(C(\partial F_d)\ltimes F_d\right)\cong 
\begin{cases}
\Z^d, \;&*=0,\\
\Z^d\oplus \Z/(d-1)\Z,\;&*=1.
\end{cases}
\end{align*}
\end{prop}

The $K$-homology groups of $(C(\partial F_d)\ltimes F_d)^{U(1)}$ are computed via Proposition \ref{ktheorymcomputations}. The expression for $K^*(C(\partial F_d)\rtimes F_d)$ can either be derived from the Pimsner-Voiculescu sequence of Theorem \ref{pimsnervoic} or found in \cite[Example 33]{emermeyereuler}. The role of the gauge cycle in the Pimsner-Voiculescu sequence in this case is non-trivial.

\subsubsection{A Cuntz-Krieger algebra such that the gauge cycle surjects}
\label{surjectgaugesubsub}

To construct a Cuntz-Krieger algebra such that the Kasparov product with the gauge cycle $[\mathpzc{E}^{c},D_c]\otimes -:K^*(F_A)\to K^*(O_A)$ surjects, we consider the $2d\times 2d$-matrix:
\[A_d:=
\begin{pmatrix} 
0&1&1&1&\cdots &1&1\\
1&0&1&1&\cdots &\vdots&\vdots\\
1&1&0&1&\cdots &\vdots&\vdots\\
1&1&1&0&\cdots &\vdots&\vdots\\
\vdots&\vdots&&&\ddots&\\
\vspace{1.5mm}
1&1&\cdots &&&0&1\\
1&1&1&\cdots&&1&0
\end{pmatrix}.\]

\begin{prop}
\label{almostfreecomputations}
It holds that 
\begin{align*}
K^*\left(F_{A_d}\right)\cong &
\begin{cases}
\Z^{2d-1}, \;&*=0,\\
\Z_{2d-1},\;&*=1
\end{cases}
\quad\mbox{and}\\
&K^*\left(O_{A_d}\right)\cong 
\begin{cases}
0, \;&*=0,\\
(\Z/2\Z)^{2(d-1)}\oplus \Z/4(d-1)\Z,\;&*=1.
\end{cases}
\end{align*}
\end{prop}

The proof of Proposition \ref{almostfreecomputations} follows from Proposition \ref{ktheorymcomputations} and Theorem \ref{pimsnervoic} after a lengthier exercise in linear algebra.

\begin{remark}
\label{almostfreegroupremark}
Writing out the Pimsner-Voiculescu sequence of Theorem \ref{pimsnervoic} using the computations of Proposition \ref{almostfreecomputations} we arrive at a commuting diagram whose rows are exact: 
\begin{align*}
\tiny
\begin{CD}
0  @>>> K^0(F_{A_d}) @>>>K^0(F_{A_d}) @>[\mathpzc{E}_c,D_c]\otimes >> K^1(F_{A_d})@>>> \\
@. @|  @| @ | \\
0  @>>> \Z^{2d-1} @>>>\Z^{2d-1} @>>> (\Z/2\Z)^{2(d-1)}\oplus \Z/4(d-1)\Z @>>>
\end{CD}\\
\tiny
\begin{CD}
@>>> K^1(F_{A_d})@>>> K^1(F_{A_d})@>>>0 \\
 @. @ |@ |@.\\
@>>> \Z_{2d-1}@>>> \Z_{2d-1}@>>>0
\end{CD}.
\normalsize 
\end{align*}
Since the $2d-1$-adic numbers $\Z_{2d-1}$ is a torsion-free group, it follows that the mapping $K^1(O_{A_d})\to K^1(F_{A_d})$ vanishes. We conclude that the Kasparov product with the gauge class in fact surjects onto the $K$-homology of the Cuntz-Krieger algebra $O_{A_d}$.
\end{remark}

\vspace{3mm}

\Large
\section{An even spectral triple on the algebra $C(\Omega_A)$}
\label{sectionbptriples}
\normalsize

In \cite{BP}, a family of even spectral triples were defined for boundaries of trees. While the space of finite words $\mathcal{V}_A$ is a tree and $\Omega_A$ is its boundary, even spectral triples for $C(\Omega_A)$ can be obtained in this way. We will in this section recall the construction of \cite{BP} and prove that the spectral triples obtained in this way pair non-degenerately with many elements in $K_0(C(\Omega_A))$. They encode geometric, measure-theoretic and dynamical data (see \cite{BP,sharpe}), and have the interesting property that the class $2[1_{O_{A^T}}]\in K_0(O_{A^T})$ obstructs the extension of these spectral triples to $O_A$ (see Proposition \ref{twotorsionlifting}). We also interpret these spectral triples as secondary invariants for the triviality of the restriction of the extension class dual to $2[1_{O_{A^T}}]\in K_0(O_{A^T})$ to $C(\Omega_A)$ (see Remark \ref{secondaryremark}). In the next section we will consider generalized unbounded Fredholm modules for $O_A$ constructed through the unbounded Kasparov product, using the present spectral triples as the base.

\subsection{The Bellissard-Pearson spectral triples}

The considerations in \cite{BP} allows one to define a spectral triple on the boundary of a tree by means of interior properties of the tree. In our situation the tree is $\mathcal{V}_A$ and its boundary is $\Omega_A$. The key geometric idea, that transfers geometry on the interior to that on the boundary, is that of a choice function.

\begin{deef}[Choice functions on finite words]
\label{choicefunctionz}
Let $\mathfrak{t},\mathfrak{t}' :\mathcal{V}_A\to \Omega_A$ denote functions. We say that the pair $\tau=(\mathfrak{t},\mathfrak{t}')$ is comparable if there is a constant $C>0$ such that $\tau$ satisfies that 
$$d_{\Omega_A}(\mathfrak{t}(\mu),\mathfrak{t}'(\mu))\leq C \mathrm{diam}(C_\mu).$$
If the inequality is an equality with $C=1$ for all $\mu$, we say that $\tau$ is strictly comparable. A comparable pair of functions satisfying the cylinder condition, see Definition \ref{cylinderchoicefunctionz}, is called a weak choice function. If the comparison is strict we say $\tau$ is a choice function. 
\end{deef}

If $\mathfrak{t}$ and $\mathfrak{t}'$ satisfy the cylinder condition, then $\tau=(\mathfrak{t},\mathfrak{t}')$ is comparable with $C=1$. For a function $\mathfrak{t}:\mathcal{V}_A\to \Omega_A$, we let $\pi_\mathfrak{t}:C(\Omega_A)\to \Bo(\ell^2(\mathcal{V}_A))$ denote the composition of the pullback homomorphism $\mathfrak{t}^*:C(\Omega_A)\to C_b(\mathcal{V}_A)$ with the representation given by pointwise multiplication $C_b(\mathcal{V}_A)\to \Bo(\ell^2(\mathcal{V}_A))$. Compare to Proposition \ref{cylinderguaranteescont} if $\mathfrak{t}$ satisfies the cylinder condition.

\begin{deef}[The Bellissard-Pearson spectral triple \cite{BP}]
Let $\tau=(\tau_+,\tau_-):\mathcal{V}_A\to \Omega_A\times \Omega_A$ be a comparable pair. The associated even Bellissard-Pearson spectral triple $BP^{\exp}(\tau):=(\pi_\tau,\ell^2(\mathcal{V}_A,\C^2),D_{\mathcal{V}}^{BP})$ consists of 
\begin{enumerate}
\item The Hilbert space $\ell^2(\mathcal{V}_A,\C^2)$ graded by the decomposition 
$$\ell^2(\mathcal{V}_A,\C^2)=\ell^2(\mathcal{V}_A)\oplus \ell^2(\mathcal{V}_A).$$
\item The even representation $\pi_\tau:C(\Omega_A)\to \Bo(\ell^2(\mathcal{V}_A,\C^2))$ given by 
$$\pi_\tau:=\pi_{\tau_+}\oplus \pi_{\tau_-}.$$
\item The self-adjoint operator $D_{\mathcal{V}}^{BP}$ defined on its core $C_c(\mathcal{V}_A,\C^2)$ by 
$$D_{\mathcal{V}}^{BP}\begin{pmatrix} \phi_+\\\phi_-\end{pmatrix}(\mu):=\mathrm{diam}(C_\mu)^{-1}\cdot \begin{pmatrix}\phi_-(\mu)\\ \phi_+(\mu)\end{pmatrix}=\e^{|\mu|}\begin{pmatrix}\phi_-(\mu)\\ \phi_+(\mu)\end{pmatrix}.$$
\end{enumerate}

For $s\in (0,1]$, we also define the logarithmic family of Bellissard-Pearson spectral triples 
\[BP_s(\tau):=(\pi_\tau,\ell^2(\mathcal{V},\C^2),D_{\mathcal{V},s}),\]
where the operator $D_{\mathcal{V},s}$ is defined on its core $C_c(\mathcal{V}_A,\C^2)$ by the expression
$$D_{\mathcal{V},s}\begin{pmatrix} \phi_+\\ \phi_-\end{pmatrix}(\mu):=\left(-\log \mathrm{diam}(C_\mu)\right)^s\cdot \begin{pmatrix}\phi_-(\mu)\\ \phi_+(\mu)\end{pmatrix}=|\mu|^s\begin{pmatrix}\phi_-(\mu)\\ \phi_+(\mu)\end{pmatrix}.$$
\end{deef}

\begin{remark}
The construction of spectral triples in \cite{BP} was only carried out for choice functions and the logarithmic version was not considered. The results in \cite{BP} regarding these spectral triples were concerned with metric and measure-theoretic properties. The motivation to lax the conditions on $\tau$ stems from the wish to obtain a larger variety of $K$-homology classes that pair non-degenerately with ``many" $K$-theory elements. The introduction of the logarithmic version of the spectral triple is a matter we return to throughout the section and Subsection \ref{Kaspprod}.
\end{remark}

\begin{prop}
\label{thetaandsummable}
The logarithmic and the ordinary even Bellissard-Pearson spectral triples of a comparable pair $\tau=(\tau_+,\tau_-)$ form even unbounded Fredholm modules. For $s\geq \frac{1}{2}$, $BP_{s}(\tau)$ is $\theta$-summable, whereas $BP^{\exp}(\tau)$ is finitely summable. If the image of $\tau$ is dense\footnote{E.g. when $\tau$ satisfies the cylinder condition.} the Bellissard-Pearson spectral triples indeed form spectral triples.
\end{prop}

\begin{proof}
It was proven in \cite[Proposition $8$]{BP} that $(\pi_\tau,\ell^2(\mathcal{V}_A,\C^2),D_{\mathcal{V}}^{BP})$ is a well defined unbounded Fredholm module. The operator $D_{\mathcal{V}}^{BP}$ admits bounded commutators with elements of the algebra $\mathrm{Lip}(\Omega_A,d_{\Omega_A})$ consisting of functions $f:\Omega_A\to \C$ that are Lipschitz in the metric on $\Omega_A$ defined in \eqref{metric}. From the estimate
\small
\[\begin{split}\|[D_{\mathcal{V},s},\pi_{\tau}(f)]\|_{\Bo(\ell^2(\mathcal{V}_A,\C^2))} &=\sup_{\mu\in\mathcal{V}_{A}}|\mu|^{s}\left\|\left[\begin{pmatrix} 0 & 1\\ 1& 0\end{pmatrix},\begin{pmatrix} f(\tau_+(\mu)) & 0\\ 0& f(\tau_-(\mu))\end{pmatrix}\right]\right\|_{M_2(\C)}\\
&\leq \sup_{\mu\in\mathcal{V}_{A}}e^{|\mu|}\left\|\left[\begin{pmatrix} 0 & 1\\ 1& 0\end{pmatrix},\begin{pmatrix} f(\tau_+(\mu)) & 0\\ 0& f(\tau_-(\mu))\end{pmatrix}\right]\right\|_{M_2(\C)}\\
&=\|[D^{BP}_{\mathcal{V}_{A}},\pi_{\tau}(f)]\|_{\Bo(\ell^2(\mathcal{V}_A,\C^2))},\end{split}\]
\normalsize
it follows that the same holds for $BP_{s}(\tau)$. Since $\mathrm{diam}(C_\mu)= \e^{-|\mu|}$, it follows that 
\begin{equation}
\label{trace}
\tra(e^{-tD_{\mathcal{V},s}^2})= 2\sum_{\mu\in \mathcal{V}_A}\e^{-t|\mu|^{2s}}=2 \sum_{k=0}^\infty \sum_{|\mu|=k}\e^{-tk^{2s}}=2 \sum_{k=0}^\infty \phi(k)\e^{-tk^{2s}}.
\end{equation} 
By Corollary \ref{phiasymptotics}, the operator $e^{-tD_{\mathcal{V},s}^2}$ is trace class if $s\geq \frac{1}{2}$ and $t> \delta_A$.
\end{proof}

\begin{remark}
For $s=\frac{1}{2}$ the trace \eqref{trace} equals the Poincar\'{e} series from Theorem \ref{Poincare}. After introducing a power in the metric defined in Equation \eqref{metric} and in the expression defining $D_{\mathcal{V}}^{BP}$, one can obtain arbitrarily low degree of finite summability. Further, if there are constants $C,p>0$ such that $\phi(k)\leq Ck^p$ for all $k$, then $BP_{s}(\tau)$ is also finitely summable. This holds for instance for $SU_q(2)$ by \eqref{phicomputation}. This is possible only when there is an isolated point in $\Omega_A$ as the following proposition shows. 
\end{remark}

\begin{prop}
If the matrix $A$ satisfies condition $(I)$, there are $C,\epsilon>0$ such that 
$$\phi(k)\geq C\e^{\epsilon k}.$$
\end{prop}

\begin{remark}
\label{decomposingbptriples}
For any point $x\in \Omega_A$, let $\omega_x:C(\Omega_A)\to \C$ denote point evaluation in $x$. Let $[x]\in K^0(C(\Omega_A))$ denote the $K$-homology class associated with $\omega_x$. Formally, we may realize the $K$-homology class that the Bellissard-Pearson spectral triple defines as the formal difference of the sum of all $[x]$, where $x$ ranges over $\tau_+(\mathcal{V}_A)$, and the sum of all $[x]$, where $x$ ranges over $\tau_-(\mathcal{V}_A)$.

This observation can be made sense of in a more rigorous way. For $\mu\in \mathcal{V}_A$, the difference $[\tau_+(\mu)]-[\tau_-(\mu)]\in K^0(C(\Omega_A))$ can be represented by the even unbounded Fredholm module 
\[\mathfrak{S}_{\mu,s}=\left(\omega_{\tau_+(\mu)}\oplus\omega_{\tau_-(\mu)}, \C^2, \begin{pmatrix} 0& |\mu|^{s}\\|\mu|^{s}& 0\end{pmatrix}\right).\]
The direct sum 
$$\bigoplus_{\mu\in \mathcal{V}_A} \mathfrak{S}_{\mu,s}=\left(\bigoplus_{\mu\in \mathcal{V}_A} \omega_{\tau_+(\mu)}\oplus\omega_{\tau_-(\mu)}, \bigoplus_{\mu\in \mathcal{V}_A} \C^2,\bigoplus_{\mu\in \mathcal{V}_A}  \begin{pmatrix} 0& |\mu|^{s}\\|\mu|^{s}& 0\end{pmatrix}\right)=BP_{s}(\tau)$$ 
is well defined once making suitable closures and choices of domains.
\end{remark}

\subsection{Obstructions to extending to Cuntz-Krieger algebras}
\label{subsectionobstructingckext}

Our motivation for introducing the logarithmic version of the Bellissard-Pearson spectral triple is that it extends to a slightly larger algebra related to the Cuntz-Krieger algebra, but not equal to it. The deficiency between that algebra and the Cuntz-Krieger algebra comes from an obstruction in $K_0(O_{A^T})$ (see Proposition \ref{dualofunitprop} and \ref{twotorsionlifting}). 

We let $V_{\sigma}\in \Bo(\ell^2(\mathcal{V}_A))$ be defined by 
$$V_{\sigma}f(v)=
\begin{cases}
 f(\sigma_{\!\mathcal{V}}(v))\quad\mbox{if}\quad v\neq \circ_A\\
0,\quad \mbox{if}\quad v=\circ_A.
\end{cases}.$$
A direct computation gives the identity 
$$V_\sigma^*V_\sigma=S,$$
where $Sf(x)=|\sigma_{\!\mathcal{V}}^{-1}\{x\}|f(x)$. We will henceforth apply the convention that 
\[\sigma_{\!\mathcal{V}}^{-1}(\sigma_{\!\mathcal{V}}(\circ_A))=\emptyset.\]

Assume that $\mathfrak{t}:\mathcal{V}_A\to \Omega_A$ is function satisfying the cylinder condition (see Definition \ref{cylinderchoicefunctionz}). We define the operators $\mathfrak{s}_{i,\mathfrak{t}}\in \Bo(\ell^2(\mathcal{V}_A))$ for $i=1,\ldots, n$ by
\begin{equation}
\label{sitdef}
\mathfrak{s}_{i,\mathfrak{t}}:=\pi_\mathfrak{t} (\chi_{C_i})V_{\sigma}.
\end{equation}
We also let $P_{\circ_A}:\ell^2(\mathcal{V}_A)\to \ell^2(\mathcal{V}_A)$ denote the orthogonal projection onto the space spanned by $\delta_{\circ_A}$. 

\begin{lem}
\label{sionell2}
Let $\mathfrak{t}:\mathcal{V}_A\to \Omega_A$ be a function satisfying the cylinder condition. The operators $\mathfrak{s}_{i,\mathfrak{t}}$ are partial isometries satisfying the relations
\begin{equation}
\label{ckrelationspb}
\mathfrak{s}_{i,\mathfrak{t}}^*\mathfrak{s}_{k,\mathfrak{t}}=\delta_{i,k}\sum_{j=1}^N A_{ij} \mathfrak{s}_{j,\mathfrak{t}}\mathfrak{s}_{j,\mathfrak{t}}^*+P_{\circ_A},
\end{equation}
for any $i$ and $k$. 
\end{lem}

\begin{proof}
If $i\neq j$, $\chi_{C_i}\chi_{C_j}=0$ and it follows that
\[\mathfrak{s}_{i,\mathfrak{t}}^*\mathfrak{s}_{j,\mathfrak{t}}=V_\sigma^*\pi_\mathfrak{t} (\chi_{C_i}\chi_{C_j})V_\sigma=0.\]
Given $f\in \ell^2(\mathcal{V}_A)$, we have that 
\begin{align*}
\mathfrak{s}_j\mathfrak{s}_j^*f(\mu)&=\pi_\mathfrak{t} (\chi_{C_j})V_{\sigma}V_\sigma^*\pi_\mathfrak{t} (\chi_{C_j})f(\mu)=\\
&=\sum_{\nu\in \sigma^{-1}_{\!\mathcal{V}}(\sigma_{\!\mathcal{V}}(\mu))}\chi_{C_j}(\mathfrak{t}(\mu))\chi_{C_j}(\mathfrak{t}(\nu))f(\nu)=
\begin{cases}
\chi_{C_j}(\mathfrak{t}  (\mu))f(\mu),\quad\mbox{if}\quad \mu\neq \circ_A,\\
0\quad\mbox{if}\quad \mu= \circ_A.
\end{cases}
\end{align*}
We conclude that $(\mathfrak{s}_{i,\mathfrak{t}})_{i=1}^N$ forms a collection of partial isometries with orthogonal ranges. On the other hand, 
\begin{align*}
\mathfrak{s}_{i,\mathfrak{t}}^*\mathfrak{s}_{i,\mathfrak{t}}f(\mu)&=V_\sigma^*\pi_\mathfrak{t} (\chi_{C_i})f(\sigma_{\!\mathcal{V}}( \mu))=\sum_{\nu\in \sigma_{\!\mathcal{V}}^{-1}( \mu)} \chi_{C_i}(\mathfrak{t}(\nu))f(\sigma_{\!\mathcal{V}}( \nu))=\\
&=\sum_{\nu\in \sigma_{\!\mathcal{V}}^{-1}(\mu), \mathfrak{t}(\nu)\in C_i} f(\mu)=\sum_{j=1}^N A_{ij}\chi_{C_j}(\mathfrak{t}(\mu))f(\mu),
\end{align*}
since the word $\nu=i\mu\in \sigma^{-1}(\mu)$ is admissible only when $A_{i\mu_1}\neq 0$. 
Rewriting this, we obtain the identity 
\begin{equation}
\label{sisjstarsandshit}
\mathfrak{s}_{i,\mathfrak{t}}\mathfrak{s}_{i,\mathfrak{t}}^*=
\begin{cases}
\pi_{\mathfrak{t}}(\chi_{C_i})-P_{\circ_A},\quad\mbox{if}\quad \mathfrak{t}(\circ_A)\in C_i,\\
\pi_{\mathfrak{t}}(\chi_{C_i}),\quad\mbox{if}\quad \mathfrak{t}(\circ_A)\notin C_i.
\end{cases}
\end{equation}
Since there is only one $i$ for which $\mathfrak{t}(\circ_A)\in C_i$, Equation \eqref{ckrelationspb} holds true.
\end{proof}

\begin{remark}
There is a geometric consequence of Lemma \ref{sionell2} for $\mathcal{G}_A$. Later we will prove that for any function $\mathfrak{t}$ satisfying the cylinder condition, the linear mapping $S_i\mapsto \mathfrak{s}_{i,\mathfrak{t}}$ can not be compactly perturbed to a $*$-homomorphism $O_A\to \Bo(\ell^2(\mathcal{V}_A))$ if $[1]\neq 0$ in the $K$-theory group $K_0(O_{A^T})$, this is related to Kaminker-Putnam's Poincar\'e duality $K^*(O_A)\cong K_{*+1}(O_{A^T})$. See more in Remark \ref{kaspprodbusby}, Proposition \ref{dualofunitprop} and Proposition \ref{twotorsionlifting}. In particular, it proves it impossible for a function $\mathfrak{t}:\mathcal{V}_A\to \Omega_A$ satisfying the cylinder condition to be viewed as the moment map of a $\mathcal{G}_A$-action on the finite words $\mathcal{V}_A$ since if that was the case, it would extend to a $*$-homomorphism $O_A\cong C^*(\mathcal{G}_A)\to \Bo(\ell^2(\mathcal{V}_A))$ extending the $C(\Omega_{A})$-representation coming from $\mathfrak{t}$.
\end{remark}

In order to understand the role of the operators $(\mathfrak{s}_{i,\mathfrak{t}})_{i=1}^N$, we need to relate them to a similar set of operators appearing above in Subsection \ref{recokpsubsec}, cf. \cite{kaminkerputnam}.

\begin{prop}
\label{lisandsis}
If $\mathfrak{t}$ satisfies the cylinder condition, it holds that $L_i^A=\mathfrak{s}_{i,\mathfrak{t}}$.
\end{prop}

\begin{proof}
For any finite word $\mu\in \mathcal{V}_A$, 
$$\mathfrak{s}_{i,\mathfrak{t}}\delta_\mu=\sum_{\nu\in \sigma^{-1}_\mathcal{V}(\mu)}\chi_{C_i}(\mathfrak{t}(\nu))\delta_\nu=\delta_{i\mu},$$
since the cylinder condition (see Definition \ref{cylinderchoicefunctionz}) guarantees that $\nu=i\mu$ is the unique word in $\sigma^{-1}_\mathcal{V}(\mu)$ such that $\chi_{C_i}(\mathfrak{t}(\nu))\neq 0$.
\end{proof}

The computations of Lemma \ref{sionell2} can also be seen from Proposition \ref{lisandsis} and \cite[Proposition $4.2$]{kaminkerputnam}.

\begin{remark} 
A consequence of Proposition \ref{lisandsis} is that the operators $\mathfrak{s}_{i,\mathfrak{t}}$ do not depend on the choice of $\mathfrak{t}$. This does not contradict computations such as that in Equation \eqref{sisjstarsandshit} since this computation merely expresses a cancellation occurring in $\circ_A$. We conclude the following Proposition.
\end{remark}

\begin{prop}
\label{extdualtoone}
The $C^*$-algebra
\[E_{BP}:=C^*(\mathfrak{s}_{i,\mathfrak{t}}|i=1,\ldots,n))\subseteq \Bo(\ell^2(\mathcal{V}_A)),\]
contains $\Ko(\ell^2(\mathcal{V}_A))$ and $\mathfrak{t}^*C(\Omega_A)$ for any function $\mathfrak{t}:\mathcal{V}_{A}\rightarrow \Omega_{A}$ satisfying the cylinder condition.
\end{prop}

We note that if $\tilde{E}_{BP}$ is the $C^*$-subalgebra $\tilde{E}_{BP}\subseteq E_{KP}$ generated by the $L^A_i$, then $E_{BP}$ is the image of $\tilde{E}_{BP}$ in $\Bo(\ell^2(\mathcal{V}_A))$. If $O_A$ is simple, $E_{BP}\cong \tilde{E}_{BP}$. By arguments similar to those in Subsection \ref{recokpsubsec}, $ \tilde{E}_{BP}/\Ko(\ell^2(\mathcal{V}_A))\cong O_A$. We can conclude the following Proposition from Remark \ref{kaspprodbusby}.

\begin{prop}
\label{dualofunitprop}
The extension $\tilde{E}_{BP}$ represents the image of $[1_{O_{A^T}}]\in K_0(O_{A^T})$ under the isomorphism $K_0(O_{A^T})\to K^1(O_A)$ of Theorem \ref{pd}.
\end{prop}

For any pair of functions $\tau=(\tau_+,\tau_-)$ satisfying the cylinder condition, we set $\mathfrak{s}_i:=\mathfrak{s}_{i,\tau_+}\oplus \mathfrak{s}_{i,\tau_-}$. It follows from Proposition \ref{dualofunitprop} that if the element $[1_{O_{A^T}}]\in K_0(O_{A^T})$ is $2$-torsion, the $K$-homological obstruction to lifting the mapping 
$$S_i\mapsto \mathfrak{s}_i\mod \Ko\in \mathcal{C}(\ell^2(\mathcal{V}_A,\C^2))$$ 
to a $*$-homomorphism $O_A\to \Bo(\ell^2(\mathcal{V}_A,\C^2))$, vanishes. In a similar fashion, we conclude the following. 

\begin{prop}
\label{twotorsionlifting}
Assume that $k$ is such that $k[1_{O_{A^T}}]\neq 0$. For functions $\mathfrak{t}_1,\ldots, \mathfrak{t}_k:\mathcal{V}_A\to \Omega_A$ satisfying the cylinder condition,
$$\oplus_{j=1}^k \pi_{\mathfrak{t}_j}:C(\Omega_A)\to \Bo(\ell^2(\mathcal{V}_A,\C^k)),$$ 
does not extend to a representation of $O_A$ and neither does any compact perturbation of it. 
\end{prop}

\begin{remark}
\label{faobstruding}
If $K^0(F_A)=0$, it follows from Theorem \ref{pimsnervoic} that $K^1(O_A)\to K^1(F_A)$ is injective. This happens for instance for the algebra $O_N$ as we saw above in Remark \ref{ktheorycomputationsforthealgebraon}. In this particular case, the obstruction mentioned in Proposition \ref{twotorsionlifting} to lifting the representation of $C(\Omega_A)$ in the Bellissard-Pearson spectral triples remains for $F_A$.
\end{remark}

\begin{prop}
\label{kaatorsionlifting}
If $\mathfrak{t}:\mathcal{V}_A\to \Omega_A$ is a function satisfying the cylinder condition, the representation $\pi_\mathfrak{t}$ of $C(\Omega_A)$ satisfies that 
\[q\circ \pi_\mathfrak{t}=\beta_A|_{C(\Omega_A)},\]
and hence $[\beta_A]|_{C(\Omega_A)}=0$ in $K^1(C(\Omega_A))$.
\end{prop}

\begin{remark}
\label{secondaryremark}
An interesting interpretation of this Proposition is that the Bellissard-Pearson spectral triples should be thought of as an invariant for the choice of two multiplicative liftings of $\beta_A|_{C(\Omega_A)}$, i.e. a \emph{secondary} invariant for the homological triviality of the Toeplitz extension $E_{BP}$ restricted to $C(\Omega_A)$.
\end{remark}

\begin{prop}
\label{dlogprop}
For $i=1,\ldots, N$, the operator $\mathfrak{s}_{i}:=\mathfrak{s}_{i,\tau_+}\oplus \mathfrak{s}_{i,\tau_-}$ and its adjoint 
\begin{enumerate}
\item preserve $C_c(\mathcal{V}_A,\C^2)$;
\item admit bounded commutators with $D_{\mathcal{V},s}$;
\item there is a sequence $(f_k)\subseteq C_c(\mathcal{V}_A,\C^2)$ such that $\|f_k\|=1$ but $\|[D_{\mathcal{V}}^{BP},\mathfrak{s}_i]f_k\|\to \infty$.
\end{enumerate}
\end{prop}

\begin{proof}
Property $(1)$ is clear from the definition $\mathfrak{s}_i:=\mathfrak{s}_{i,\tau_+}\oplus \mathfrak{s}_{i,\tau_-}$ and Equation \eqref{sitdef}. To prove $(2)$, we note that Proposition \ref{lisandsis} implies that
$$[D_{\mathcal{V}, s},\mathfrak{s}_i]\begin{pmatrix}\delta_\mu\\ \delta_\nu\end{pmatrix}=\begin{pmatrix}\left(|i\nu|^{s}-|\nu|^s\right)\delta_{i\nu}\\ \left(|i\mu|^{s}-|\mu|^s\right)\delta_{i\mu}\end{pmatrix}.$$
Since $\mu\mapsto -\log\mathrm{diam}(C_\mu)=|\mu|$ grows linearly in $|\mu|$, $\mu\mapsto |i\mu|^{s}-|\mu|^{s}$ is a bounded function for $0<s\leq 1$. Hence $[D_{\mathcal{V},s},\mathfrak{s}_i]$ is bounded.

Concerning $(3)$, it follows from Proposition \ref{lisandsis} that
$$[D_{\mathcal{V}}^{BP},\mathfrak{s}_i]\begin{pmatrix}\delta_\mu\\ \delta_\nu\end{pmatrix}=
\begin{pmatrix}\left(\mathrm{diam}(C_{i\nu})^{-1}-\mathrm{diam}(C_{\nu})^{-1}\right)\delta_{i\nu}\\ 
\vspace{-2mm}\\
\left(\mathrm{diam}(C_{i\mu})^{-1}-\mathrm{diam}(C_{\mu})^{-1}\right)\delta_{i\mu}\end{pmatrix}.$$
Take a sequence $(\mu_k)_{k=1}^\infty\subseteq \mathcal{V}_A$ such that $|\mu_k|=k$ and $i\mu_k$ is admissible for all $k$. Set $f_k:=(\delta_{\mu_k},0)^T$. It trivially holds that $f_k\in C_c(\mathcal{V}_A,\C^2)$ and that $\|f_k\|=1$. Since $\mathrm{diam}(C_\mu)=\e^{-|\mu|}$,  there is an $\epsilon>0$ for which
\[\frac{\mathrm{diam}(C_{\mu})}{\mathrm{diam}(C_{i\mu})}>1+\epsilon\]
We conclude that $\|[D_{\mathcal{V}}^{BP},\mathfrak{s}_i]f_k\|\geq \epsilon \e^k\to \infty$, as $k\to \infty$.
\end{proof}

As a consequence of Proposition \ref{dlogprop}, the operator $D_{\mathcal{V},s}$ defines a spectral triple on $E_{BP}$ (see Proposition \ref{extdualtoone}) which is $\theta$-summable for $s\geq 1/2$. Yet another consequence is that $D_{\mathcal{V}}^{BP}$ does \emph{not} define a spectral triple on $E_{BP}$ such that $\mathfrak{s}_{i,\mathfrak{t}}$ is in the Lipschitz algebra. In the light of Theorem \ref{variousstructuresonoa} and Proposition \ref{thetaandsummable} this result does not come as a surprise as in that case we would obtain a finitely summable spectral triple on $E_{BP}$. We do however note that there is no obvious obstruction to finitely summable spectral triples on $E_{BP}$ since it is not purely infinite. This fact follows from \cite[Proposition V$.2.2.23$]{blackadarbook} and the existence of the inclusion $\Ko(\ell^2(\mathcal{V}_A))\subseteq E_{BP}$ of Proposition \ref{extdualtoone}.

\subsubsection{Dual of the unit for a free group}
\label{dualofunitfreegroup}

We end this subsection by a comparison of various descriptions of the extension dual to $[1_{O_{A^T}}]\in K_0(O_{A^T})$ in the special case of a free group. This example, described above in Subsection \ref{freegroupexample}, falls into the category of extensions studied by Emerson-Nica \cite{EN2}. The extension constructed in \cite{EN2} is defined from the short exact sequence
\[0\to C_0(F_d)\rtimes F_d\to C(\overline{F_d})\rtimes F_d\to C(\partial F_d)\rtimes F_d\to 0.\]
Using the isomorphism $C_0(F_d)\rtimes F_d\cong \Ko(\ell^2(F_d))$ we obtain an extension $E_{EN}$ whose class was proven in \cite{EN2} to be dual to $[1_{O_{A}}]\in K_0(O_{A})$. In \cite{EN2} an explicit finitely summable analytic $K$-cycle representing this extension class was prescribed. Recall the measure $\mu_A$ on $\partial F_d$ constructed as in Subsection \ref{realizingboundarysubsec}. Let $P_{EN}$ be the orthogonal projection onto the image of the isometric embedding $\ell^2(F_d)\to \ell^2(F_d,L^2(\partial F_d,\mu_A))$ as constant functions on $\partial F_d$. By \cite[Theorem $1.1$]{EN2} the class $[E_{EN}]$ is represented by the finitely summable analytic cycle $(\pi_{F_d}, \ell^2(F_d,L^2(\partial F_d,\mu_A)), 2P_{EN}-1)$, where $\pi_{F_d}$ is the crossed product representation associated with the covariant $C(\partial F_d)$-representation on $\ell^2(F_d,L^2(\partial F_d,\mu_A))$. One can check that this construction of $P_{EN}$ corresponds to the construction of $\mathfrak{Q}_{\circ}$ in Remark \ref{measureremark} thus concluding the following Proposition. 

\begin{prop}
\label{comparingtoen}
If $A$ is the $2d\times 2d$-matrix from Subsection \ref{freegroupexample}, the following diagram with exact rows commute:
\[\begin{CD}
0  @>>>\Ko(\ell^2(F_d)) @>>>E_{BP} @>>> O_A@>>>0 \\
@. @|  @| @|@.\\
0  @>>> \Ko(\ell^2(F_d))@>>>E_{EN} @> >> O_A@>>>0 \\
\end{CD} \]
Furthermore, under the unitary equivalence $L^2(O_A,\phi_A)\cong \ell^2(F_d,L^2(\partial F_d,\mu_A))$ induced by the isomorphism of groupoids $\mathcal{G}_A\cong \partial F_d\rtimes F_d$ it holds that 
\[(\pi_A,L^2(O_A,\phi_A),2W_\circ W^*_\circ-1)=(\pi_{F_d}, \ell^2(F_d,L^2(\partial F_d,\mu_A)), 2P_{EN}-1)\]
\end{prop}

\begin{remark}
For a general $N\times N$-matrix $A$, there are several other equivalent ways of constructing extensions equivalent to $E_{BP}$ in a geometric way from the short exact sequence
\[0\to C_0(\mathcal{V}_A)\to C\left(\overline{\mathcal{V}_A}\right)\to C(\Omega_A)\to 0.\]
For instance, using crossed products by partial actions of the free group $F_N$ on $\Omega_A$ (see \cite{exelpartial}) or a crossed product by the shift endomorphism (see \cite{exel}).
\end{remark}

\subsection{$K$-homology classes}

We now turn to the study of the index theory of the Bellissard-Pearson spectral triples. Whenever $(\pi,\He,D)$ is an unbounded Fredholm module on a $C^*$-algebra $\cstar$, we let $[\pi,\He,D]\in K^*(\cstar)$ denote its $K$-homology class, obtained via the bounded transform. Throughout this subsection, $\tau=(\tau_+,\tau_-)$ denotes a comparable pair of functions $\mathcal{V}_A\to \Omega_A$. For most of the section, $\tau$ will be a weak choice function.

\begin{lem}
\label{bddtransformofpb}
For $0<s\leq 1$, the bounded transforms of the logarithmic and the ordinary even Bellissard-Pearson spectral triples coincide in $K$-homology:
$$[BP^{\exp}(\tau)]=[BP_{s}(\tau)]\in K^0(C(\Omega_A)).$$
Further, the class $[BP_{s}(\tau)]\in K^0(C(\Omega_A))$ of a comparable pair $\tau$ can be represented by the analytic $K$-cycle 
\begin{equation}
\label{kcyclerepbp}
\left(\pi_\tau,\ell^2(\mathcal{V}_A,\C^2),F \right),\quad\mbox{where}\quad F:=\begin{pmatrix} 0& 1\\1&0\end{pmatrix}.
\end{equation}
For any $p>0$ and weak choice function $\tau$, this $K$-cycle is $p$-summable on the dense $*$-subalgebra generated by cylinder functions inside $C(\Omega_A)$.
\end{lem}

\begin{proof}
It is clear that $[BP^{\exp}(\tau)]=[BP_{s}(\tau)]$. That $[BP_{s}(\tau)]\in K^0(C(\Omega_A))$ is represented by the $K$-cycle \eqref{kcyclerepbp} follows from that $F=D_{\mathcal{V}}^{BP}|D_{\mathcal{V}}^{BP}|^{-1}$. To verify the $p$-summability claim, take a finite word $\mu\in \mathcal{V}_A$ and consider the locally constant function $\chi_{C_\mu}\in C(\Omega_A)$. For any $\nu,\nu'\in \mathcal{V}_A$,
\[[F,\pi_\tau(\chi_{C_\mu})]\begin{pmatrix} \delta_\nu\\\delta_{\nu'}\end{pmatrix}=
\begin{pmatrix} \left(\chi_{C_\mu}(\tau_-(\nu'))-\chi_{C_\mu}(\tau_+(\nu'))\right) \delta_{\nu'}\\
\vspace{-2mm}\\
\left(\chi_{C_\mu}(\tau_+(\nu))-\chi_{C_\mu}(\tau_-(\nu))\right) \delta_{\nu}\end{pmatrix}.\]
If both $\tau_+$ and $\tau_-$ satisfies the cylinder condition, then
\[\chi_{C_\mu}(\tau_+(\nu))-\chi_{C_\mu}(\tau_-(\nu))=0\quad\mbox{if} \quad |\nu|\geq |\mu|.\]
The latter statement holds, because $\chi_{C_\mu}(\tau_\pm(\nu))$ is non-zero if and only if $\tau_\pm(\nu)\in C_\mu$ and whenever $|\nu|\geq |\mu|$ the cylinder condition and $\tau_\pm(\nu)\in C_\mu$ implies that there is a finite word $\lambda$ with $\nu=\mu\lambda$, hence $\tau_+(\nu)\in C_\mu$ if and only if $\tau_-(\nu)\in C_\mu$. It follows that $[F,\pi_\tau(\chi_{C_\mu})]$ is an operator of rank at most $2\sum_{k< |\mu|}\phi(k)$ and hence $p$-summable for any $p>0$. The linear span of the cylinder functions $\{\chi_{C_{\mu}}|\mu\in \mathcal{V}_A\}$ forms a dense subalgebra of $C(\Omega_A)$ and the Lemma follows.
\end{proof}

Lemma \ref{bddtransformofpb} gives us a description of the class $[BP(\tau)]$ by means of the quasi-homomorphism $(\pi_{\tau_+},\pi_{\tau_-})$, cf. \cite{cuntzquasi}. To understand the index pairing of the Bellissard-Pearson spectral triples with $K$-theory, we first recall a well known computation of the $K$-theory of $C(\Omega_A)$.

\begin{lem}
\label{ktheoryfunctions}
The $K$-theory group $K_*(C(\Omega_A))$ is given by
\[K_*(C(\Omega_A))=
\begin{cases} 
C(\Omega_A,\Z),\quad\mbox{if}\quad *=0,\\
0,\quad\mbox{if}\quad *=1. \end{cases}\]
\end{lem}

\begin{proof}
We write $C(\Omega_A)=\varinjlim \mathcal{C}_k$ as in Proposition \ref{afstructureonfunctions}. Continuity of $K$-theory under direct limits implies that 
\begin{align*}
K_*(C(\Omega_A))&=\varinjlim K_*(\mathcal{C}_k)=
\begin{cases} 
\varinjlim K_0(\mathcal{C}_k),\quad\mbox{if}\quad *=0,\\
0,\quad\mbox{if}\quad *=1. \end{cases}\\
&=
\begin{cases} 
C(\Omega_A,\Z),\quad\mbox{if}\quad *=0,\\
0,\quad\mbox{if}\quad *=1. \end{cases}.
\end{align*}
\end{proof}

We say that a word $\mu\in \mathcal{V}_A$ is minimal if the following condition holds:
\begin{equation}
\label{minimality}
\mbox{For any} \quad \nu_0,\lambda_0\in \mathcal{V}_A\quad\mbox{such that}\quad \mu=\nu_0\lambda_0,\quad \mbox{we have that} \quad C_\mu\neq C_{\nu_0}.
\end{equation}

\begin{lem}
\label{constructingchoice}
Let $\mu\in \mathcal{V}_A$ and let $\nu_0$ be the longest minimal word such that $\mu=\nu_0\lambda_0$ for some $\lambda_0$. Then there is a weak choice function $\tau=(\tau_+,\tau_-)$ such that whenever $\nu,\lambda\in \mathcal{V}_A\setminus \{\circ_A\}$ are such that $\mu=\nu\lambda$ then 
\begin{enumerate}
\item $\tau_+(\nu)\in C_\mu$ if and only if $\tau_-(\nu)\in C_\mu$ for $|\lambda|\neq |\lambda_0|+1$
\item $\tau_-(\nu)\notin C_\mu$ and $\tau_+(\nu)\in C_\mu$ if $|\lambda|=|\lambda_0|+1$.
\end{enumerate}
\end{lem}

\begin{proof}
Let $\tau^0$ be any weak choice function. We will redefine $\tau^0$ on the set of $\nu$:s such that there exists a $\lambda\in \mathcal{V}_A\setminus \{\circ_A\}$  with $\mu=\nu\lambda$. Since $C_\mu=C_{\nu_0}$ we can equally well assume $\mu=\nu_0$ and $|\lambda_0|=0$.

Whenever $\mu=\nu\lambda$, we divide into the four cases
\begin{enumerate}
\item[A)] $|\lambda|>1$ and $\tau_+^0(\nu)=\tau_-^0(\nu)$.
\item[B)] $|\lambda|=1$ and $\tau_+^0(\nu)\neq \tau_-^0(\nu)$.
\item[C)] $|\lambda|>1$ and $\tau_+^0(\nu)\neq \tau_-^0(\nu)$.
\item[D)] $|\lambda|=1$ and $\tau_+^0(\nu)=\tau_-^0(\nu)$.
\end{enumerate}
If $\nu$ satisfies $A)$, we do not alter $\tau^0(v)$. If $\nu$ satisfies $B)$, and $\tau_+(\nu)\in C_\mu$ we do not alter $\tau^0(\nu)$. If $\nu$ satisfies $B)$, and $\tau_-(\nu)\in C_\mu$ we redefine $\tau_\pm(\nu):=\tau^0_\mp(\nu)$. If $\nu$ satisfies $B)$ and $\tau_+^0(\nu),\tau_-^0(\nu)\notin C_\mu$ we do not alter $\tau_-^0(\nu)$ but define $\tau_+(\nu):=\tau_+^0(\mu)\in C_\mu$. If $C)$ holds, then we set $\tau_\pm(\nu):=\tau^0_-(\nu)$. If $D)$ holds, then the minimality assumption \eqref{minimality} guarantees that there is a finite word $\lambda'$ such that $|\lambda'|=|\lambda|$, $\lambda'\neq \lambda$ and $\nu\lambda'$ is admissible. Define $\tau_+(\nu):=\tau_+^0(\mu)\in C_\mu$ and $\tau_-(\nu):=\tau_-^0(\nu\lambda')\notin C_\mu$. The constructed $\tau$ satisfy the cylinder condition, hence $\tau$ is a comparable pair.
\end{proof}

Our main result of this subsection indicates the topological importance of the Bellissard-Pearson spectral triples.

\begin{lem}
For any non-empty word $\mu\in \mathcal{V}_A\setminus \{\circ_A\}$ there is a weak choice function $\tau_\mu$ such that 
\[\langle [\chi_{C_\mu}],[BP_s(\tau_\mu)]\rangle=1.\]
\end{lem}

\begin{proof}
It suffices to prove the Lemma for finite words $\mu$ satisfying the minimality assumption \eqref{minimality}. A straight forward index manipulation gives the identities
\begin{align*}
\langle [\chi_{C_\mu}],[BP_s(\tau_\mu)]\rangle&=\ind (\tau_+^*(\chi_{C_\mu}):\tau_-^*(\chi_{C_\mu})\ell^2(\mathcal{V}_A)\to \tau_+^*(\chi_{C_\mu})\ell^2(\mathcal{V}_A))\\
&=\ind(\tau_+^*(\chi_{C_\mu}),\tau_-^*(\chi_{C_\mu})),
\end{align*}
where the last index denotes the relative index of the Fredholm pair of projections given by $(\tau_+^*(\chi_{C_\mu}),\tau_-^*(\chi_{C_\mu}))$. Using \cite[Proposition $2.2$]{avronseilersimon}, it follows that 
\begin{align}
\nonumber
\langle [\chi_{C_\mu}],[BP_s(\tau_\mu)]\rangle&=\tra_{\ell^2(\mathcal{V}_A)}\left(\tau_+^*(\chi_{C_\mu})-\tau_-^*(\chi_{C_\mu})\right)\\
\nonumber&=\sum_{|\nu|< |\mu|} [\chi_{C_\mu}(\tau_+(\nu))-\chi_{C_\mu}(\tau_-(\nu))]\\
\label{mucomputations}
&=\#\left\{\nu\big| \;\tau_+(\nu)\in C_\mu,\; \tau_-(\nu)\notin C_\mu\right\}\\
&\nonumber\qquad\qquad-\#\left\{\nu\big| \;\tau_-(\nu)\in C_\mu,\; \tau_+(\nu)\notin C_\mu\right\}.
\end{align}
The Lemma follows from Equation \eqref{mucomputations} and Lemma \ref{constructingchoice}.
\end{proof}

\vspace{3mm}

\Large
\section{Unbounded $(O_A,C(\Omega_A))$-cycles and the associated spectral triples}
\label{oacasection}
\normalsize

In this section we will construct classes over the commutative base by combining the philosophies of Section \ref{sectionfinisumandputnamkam} and Section \ref{oafasection}. The advantage of using $C(\Omega_A)$ is that there are several well behaved $K$-homology classes, e.g. point evaluations and Bellissard-Pearson spectral triples. We will use these to construct unbounded Fredholm modules on Cuntz-Krieger algebras $O_A$ and prove that such unbounded Fredholm modules exhaust $K^1(O_A)$. For this purpose, point evaluations suffices. We consider the products with Bellissard-Pearson spectral triples in the next section. The reader unfamiliar with unbounded $KK$-theory is referred to the references listed in the beginning of Section \ref{oafasection}.

\subsection{An unbounded $(O_A,C(\Omega_A))$-cycle}
\label{definingoacoa}
We start this subsection with a structure analysis for the Haar module $\mathpzc{E}^\Omega_A$ over the commutative algebra $C(\Omega_A)$. Consider the filtration of $\mathcal{G}_{A}$ given by
\begin{equation}
\label{Gfilter}
\mathcal{G}^{k}_A:=\{(x,n,y)\in\mathcal{G}_{A}: \sigma^{k+n}(x)=\sigma^{k}(y)\},
\end{equation}
which forms a filtration by subsets such that:
\begin{enumerate}
\item Each set $\mathcal{G}^{k}_A$ is closed under under composition.
\item Inversion is a filtered operation in the sense that if $\xi=(x,n,y)\in \mathcal{G}^{k}_A$, then $\xi^{-1}\in \mathcal{G}^{n+k}_A$.
\item The filtering respects the cocycle grading; $\mathcal{G}^k_A=\cup_{n\in \Z}\mathcal{G}^{k}_n$ where $\mathcal{G}^{k}_n:=\mathcal{G}^{k}_A\cap c_A^{-1}(n)$. 
\end{enumerate}

We can further decompose this filtration into a grading. 

\begin{lem} 
The function 
\begin{equation}
\begin{split}
\label{kappa}
\kappa:\mathcal{G}_{A} &\rightarrow  \N\\
(x,n,y) &\mapsto  \min\{k:\sigma^{k+n}(x)=\sigma^{k}(y)\},
\end{split}
\end{equation}
is locally constant and hence continuous.
\end{lem}

\begin{proof} 
Recall the definition of the basic open sets from \eqref{etaletop}. Let $\kappa(x,n,y)=k$ and take $(U, n+k, k, V)$ with $U:=C_{x_{1}\cdots x_{n+k}}$ and $V=C_{y_{1}\cdots y_{k}}$. Since $k$ is minimal, it follows that $x_{n+k}\neq y_{k}$. Therefore it is clear that for any $(x',n,y')\in (U, n+k, k, V)$, $\kappa(x',n,y')=k$. So $\kappa$ is locally constant.
\end{proof}

Because $\kappa$ is continuous, the sets $\kappa^{-1}(k)$ are clopen in $\mathcal{G}_A$. Therefore each $\mathcal{G}^{k}_{A}$ decomposes as a disjoint union
\[\mathcal{G}^{k}_{A}=\bigcup_{i=0}^{k}\kappa^{-1}(i),\]
compatible with the cocycle grading. Writing
\[\mathcal{X}_{n}^{k}:=\{(x,n,y)\in\mathcal{G}_{A}: \kappa(x,n,y)=k\},\]
this gives decompositions  
\[\mathcal{G}_A=\bigcup_{n\in\Z}\bigcup_{k\in\N}\mathcal{X}_{n}^{k}\quad\mbox{and}\quad C_{c}(\mathcal{G}_A)=\bigoplus_{n\in\Z}\bigoplus_{k\in\N} C_{c}(\mathcal{X}_{n}^{k}),\]
where the former is a disjoint union and the latter is a decomposition into $C(\Omega_{A})$-submodules. For if $f\in C_{c}(\mathcal{G}_A)$ and $g\in C(\Omega_{A})$, then
\[f*g(x,n,y)=f(x,n,y)g(y),\]
so $\supp(f*g)\subset \supp f$. For $n+k<0$, $\mathcal{X}_n^k=\emptyset$ hence we use the convention $C_c(\mathcal{X}_n^k)=0$ if $n+k<0$. After completion this gives a decomposition of the Hilbert $C^*$-module $\mathpzc{E}^{\Omega}_A$ as
\[\mathpzc{E}^{\Omega}_A=\bigoplus_{n\in\Z}\bigoplus_{k\in\N}\mathpzc{E}^{k}_{n}.\] 
We will now proceed to show that each $\mathpzc{E}^{k}_{n}$ is a finitely generated projective $C(\Omega_{A})$-module.
Define the sets
\[X_{n,\mu}^{(k)}:=\{(x,n,y):\kappa(x,n,y)=k, x\in C_{\mu}, |\mu|=k+n\},\]
whose characteristic function we denote by $\chi^{k}_{n,\mu}\in C_c(\mathcal{X}_n^k)$. We set 
\[\phi_\lambda(l):=\#\{\mu\in \mathcal{V}_A:\mu\lambda\in \mathcal{V}_A, \; |\mu|=l-|\lambda|\}.\]
Recall the conditional expectation $\rho:O_{A}\rightarrow C(\Omega_{A})$ defined in \eqref{rho}.

\begin{lem}
\label{kisom} 
For any finite word $\lambda\in \mathcal{V}_A$ and $n+k\geq |\lambda|$, the column vectors 
$$v_{n,k,\lambda}:=\left(\left(\chi^{k}_{n,\mu\lambda}\right)^{*}\right)_{|\mu|=n+k-|\lambda|}\in \left(C_c(\mathcal{X}_{n}^{k})^{\phi_\lambda(n+k)}\right)^*\subseteq C_c(\mathcal{X}_{-n}^{n+k})^{\phi_\lambda(n+k)},$$ 
satisfy 
$$v_{n,k,\lambda}^{*}\rho(v_{n,k,\lambda}*f)=\chi_{(\sigma^{n+k-|\lambda|})^{-1}(C_\lambda)}*f\quad \forall f\in C_{c}(\mathcal{X}^{k}_{n}).$$
In particular, under the inclusion 
$$\left(C_c(\mathcal{X}_{n}^{k})^{\phi(n+k)}\right)^*\subseteq \mathrm{Hom}^*_{C(\Omega_A)}(\mathpzc{E}^{k}_{n}, C(\Omega_A)^{\phi(n+k)}),$$ 
the following $C(\Omega_A)$-linear operators define isometries 
$$v_{n,k}:=v_{n,k,\circ}\in \mathrm{Hom}^*_{C(\Omega_A)}(\mathpzc{E}^{k}_{n}, C(\Omega_A)^{\phi(n+k)}).$$
\end{lem}

\begin{proof} 
We must show that for $f\in C_{c}(\mathcal{X}^{k}_{n})$ 
\[\sum_{|\mu|=k+n-|\lambda|}\chi^{k}_{n,\mu\lambda}*\rho\left(\left(\chi^{k}_{n,\mu\lambda}\right)^* *f\right)=\chi_{(\sigma^{n+k-|\lambda|})^{-1}(C_\lambda)}* f.\] 
First, for arbitrary $\mu$, we compute
\[\begin{split}\rho\left(\left(\chi^{k}_{n,\mu}\right)^{*}*f\right)(x,0,x) & =\sum\left(\chi^{k}_{n,\mu}\right)^{*}(x,\ell,z)f(z,-\ell,x) \\
&=A_{\mu_{n+k},x_{k+1}}\delta_{k,\kappa(\mu\sigma^{k}(x),n,x)}f(\mu\sigma^{k}(x),n,x),\end{split}\]
and subsequently
\[\begin{split}\chi^{k}_{n,\mu}*\rho\left(\left(\chi^{k}_{n,\mu}\right)^{*}*f\right)&(x,m,y)  =\sum \chi^{k}_{n,\mu}(x,\ell,z)\rho\left(\left(\chi^{k}_{n,\mu}\right)^{*}*f\right)(z,m-\ell,y) \\
&=\sum  \chi^{k}_{n,\mu}(x,m,y)\left(\left(\chi^{k}_{n,\mu}\right)^{*}*f\right)(y,0,y)\\
&=\chi^{k}_{n,\mu}(x,m,y)A_{\mu_{n+k},y_{k+1}}\delta_{k,\kappa(\mu\sigma^{k}(y),n,y)}f(\mu\sigma^{k}(y),n,y)\\
&=\chi^{k}_{n,\mu}(x,m,y)f(x,m,y).\end{split}\]
Therefore, we have
\[\begin{split}\sum_{|\mu|=n+k-|\lambda|}\chi^{k}_{n,\mu\lambda}*\rho\left(\left(\chi^{k}_{n,\mu\lambda}\right)^{*}*f\right)(x,m,y) & =\sum_{|\mu|=n+k-|\lambda|}\chi^{k}_{n,\mu\lambda}(x,m,y)f(x,m,y)\\ &=\chi_{(\sigma^{n+k-|\lambda|})^{-1}(C_\lambda)}(x)f(x,m,y)\\ &=(\chi_{(\sigma^{n+k-|\lambda|})^{-1}(C_\lambda)}*f)(x,m,y).\end{split}\]

\end{proof}

\begin{prop} 
\label{projsum}
The Haar module $\mathpzc{E}_{A}^{\Omega}$ is the direct sum of the finitely generated projective $C(\Omega_A)$-modules $\mathpzc{E}^k_n$.
\end{prop}

\begin{proof} 
It is clear from Lemma \ref{kisom} that the image of the isometries $v_{n,k}$ equals the range of the projections $p_{n,k}=v_{n,k}v_{n,k}^{*}$. Hence, $\mathpzc{E}^k_n\cong p_{n,k}C(\Omega_A)^{\phi(k+n)}$ are finitely generated projective $C(\Omega_A)$-modules. The Proposition follows from the fact that $\mathpzc{E}^\Omega_A=\bigoplus_{n\in\Z}\bigoplus_{k\in\N}\mathpzc{E}^{k}_{n}$.
\end{proof}

Define an operator $D_{\kappa}:C_{c}(\mathcal{G}_A)\rightarrow C_{c}(\mathcal{G}_A)$ via pointwise multiplication $D_{\kappa}f(\xi):=\kappa(\xi)f(\xi)$.

\begin{prop} 
\label{kappaprop}
The operator $D_{\kappa}$ is essentially selfadjoint and regular in $\mathpzc{E}_{A}^{\Omega}$. Moreover, it commutes up to bounded operators with the generators $S_{i}$.
\end{prop}

\begin{proof} 
The operator $D_{\kappa}$ is obviously symmetric. Moreover $D_{\kappa}\pm i$ maps the submodule $C_{c}(\mathcal{G}^{k}_A)$ surjectively onto itself, and the union $\bigcup_{k} C_{c}(\mathcal{G}^{k}_A)$ is dense in $\mathpzc{E}_{A}^{\Omega}$. Therefore $D_{\kappa}\pm i$ have dense range, and the closure of $D_{\kappa}$ is selfadjoint and regular in $\mathpzc{E}^{\Omega}_{A}$. That $D_{\kappa}$ commutes up to bounded operators with the operators $S_{i}$ follows by direct computation:
\[\begin{split} [D_{\kappa},S_{i}]&g(x,n,y)\\
 &=\sum \kappa(x,n,y) S_{i}(x,\ell, z)g(z,n-\ell, y)-S_{i}(x,\ell,z)(\kappa g)(z,n-\ell,y)\\ &=\sum \left(\kappa(x,n,y)-\kappa(z,n-\ell,y)\right)S_{i}(x,\ell,z)g(z,n-\ell,y) \\
&=(\kappa(x,n,y)-\kappa(\sigma(x),n-1,y))\chi_{C_i}(x)g(\sigma(x),n-1,y)\\
&=S_{i}p_{-} g(x,n,y),
\end{split}\]
where $p_{-}$ denotes the projection onto  $\oplus_{k\in \N}\mathpzc{E}^k_{-k}$, because 
$$\kappa(x,n,y)=
\begin{cases}
\kappa(\sigma(x),n-1, y),\quad &\mbox{when}\quad n+\kappa(x,n,y)>0,\\
\kappa(\sigma(x),n-1,y)-1,\quad &\mbox{when}\quad n+\kappa(x,n,y)= 0.
\end{cases}$$
\end{proof}

The generator of the gauge action $D_{c}$, extends to a selfadjoint regular operator on $\mathpzc{E}^{\Omega}_{A}$. However, instead of a naive combination of the operators $D_{c}$ and $D_{\kappa}$, we need to assemble the two with a little more care in order to construct unbounded Kasparov modules that will eventually allow us to obtain nontrivial unbounded Fredholm modules on $O_{A}$.
We define the subset
\[Y_\lambda:=\{(x,n,y)\in \mathcal{G}^0_A:\;|\lambda|\leq n\;\mbox{and} \; \sigma^{n-|\lambda|}(x)=\lambda y\}.\]
We note that $Y_{\circ_A}=\mathcal{G}^0_A$ and $c_A|_{Y_\lambda}\geq |\lambda|$. We let $p_\lambda\in \End^*_{C(\Omega_A)}(\mathpzc{E}^\Omega_A)$ denote the projection given by pointwise multiplication by the characteristic function of $Y_\lambda$. We write $\mathpzc{E}_{n,\lambda}^{0}$ for the completion of the submodule $C_{c}(Y_{\lambda}\cap c^{-1}_A(n))$, and $\mathpzc{E}^{\perp 0}_{n,\lambda}$ for the completion of $C_{c}(\mathcal{G}^{0}_{A}\setminus Y_{\lambda}\cap c^{-1}_A(n))$. Recall the notation $\mathcal{V}_\lambda=\{\mu\lambda\in \mathcal{V}_A\}$.

\begin{prop}
\label{rangevlambda}
The projection $p_\lambda$ projects onto the closed $C(\Omega_A)$-submodule of $\mathpzc{E}^\Omega_A$ generated by $\{S_\mu|\mu\in \mathcal{V}_\lambda\}$, and can be written as $p_{\lambda}f=\sum_{n=0}^{\infty}v_{n,0,\lambda}^{*}\rho(v_{n,0,\lambda}*f)$. In particular, for any finite word $\lambda$,  the Haar module $\mathpzc{E}^{\Omega}_{A}$ decomposes as a direct sum of finitely generated projective $C(\Omega_{A})$-modules
\begin{equation}\label{lambdadecomp} E^{\Omega}_{A}=\bigoplus_{n=0}^{\infty}\mathpzc{E}^{0}_{n,\lambda}\oplus\bigoplus_{n=0}^{\infty}\mathpzc{E}^{\perp 0}_{n,\lambda}\oplus\bigoplus_{k=1}^{\infty}\bigoplus_{n\geq -k}\mathpzc{E}^{k}_{n}.\end{equation}
\end{prop}

\begin{proof}
It suffices to prove that  $p_\lambda S_{\mu\lambda}=S_{\mu\lambda}$ and that $p_\lambda S_\mu S_\nu^*= 0$ if and only if $\mu\neq \mu_0\nu$ for all $\mu_0\in \mathcal{V}_\lambda$. Since $S_{\mu\lambda}$ is defined from the characteristic function of the set
\[\{(x,|\mu|+|\lambda|,y)\in \mathcal{G}_A|x\in C_{\mu\lambda}, \,\sigma^{|\mu|+|\lambda|}(x)=y\},\]
it follows that $p_\lambda S_{\mu\lambda}=S_{\mu\lambda}$. The element $S_\mu S_\nu^*$ is defined from the characteristic function of the set
\[\{(x,|\mu|-|\nu|,y)\in \mathcal{G}_A:x\in C_{\mu}, y\in C_\nu, \, \sigma^{|\mu|}(x)=\sigma^{|\nu|}(y)\}.\]
The proposition follows from the fact that $p_\lambda S_\mu S_\nu^*$ is the characteristic function of the set
\begin{align*}
\big\{(x,&|\mu|-|\nu|,y)\in \mathcal{G}_A:x\in C_{\mu}, y\in C_\nu,\, |\lambda|+|\nu|\leq |\mu|,\\
&\qquad\qquad\qquad\qquad\qquad\qquad \, \sigma^{|\mu|}(x)=\sigma^{|\nu|}(y),\sigma^{|\mu|-|\nu|-|\lambda|}(x)=\lambda y \big\}\\
&=
\begin{cases}
\emptyset,&\mbox{if}\; \mu \neq\mu_0\nu \;\forall \mu_0\in \mathcal{V}_\lambda\\
\\
\{(x,|\mu|-|\nu|,y)\in& \mathcal{G}_A|x\in C_{\mu}, y\in C_\nu, \, \sigma^{|\mu|}(x)=\sigma^{|\nu|}(y)\}, \\ &\mbox{if for some}\;\mu_0\in \mathcal{V}_\lambda, \, \mu =\mu_0\nu. 
\end{cases}
\end{align*}
The factorization and decomposition statements now follow directly from Lemma \ref{kisom}.
\end{proof}

Recall that $\kappa|_{\mathcal{G}^k_A}=k\in \{0,1,2,\ldots\}$ and that $c_A|_{\mathcal{G}^k_A}+ k\geq 0$. Now consider the function $\psi_\lambda:\mathcal{G}_A\rightarrow\Z$ given by
\begin{equation}
\label{psilambda} 
\psi_\lambda(x,n,y)=
\begin{cases} 
\quad n &\textnormal{when } (x,n,y)\in Y_\lambda \\ 
-n &\textnormal{when } (x,n,y)\in \mathcal{G}^{0}_A\setminus Y_\lambda \\ 
-|n|-\kappa(x,n,y)&\textnormal{when } (x,n,y)\in \mathcal{G}_A\setminus \mathcal{G}^0_A, 
\end{cases}.
\end{equation}
The function $\psi_\lambda$ is clearly locally constant and continuous. Define an operator $D_{\lambda}:C_{c}(\mathcal{G}_A)\rightarrow C_{c}(\mathcal{G}_A)$ by pointwise multiplication by $\psi_\lambda$, i.e. $D_{\lambda}f(x,n,y)=\psi_\lambda(x,n,y)f(x,n,y)$. We wish to show that $D_{\lambda}$ has bounded commutators with the generators $S_{i}$. We compute
\begin{align} 
\nonumber
[D_{\lambda},S_{i}]f(x,n,y) & = (\psi_{\lambda}(x,n,y)-\psi_{\lambda}(\sigma(x),n-1,y))\chi_{C_i}(x)f(\sigma(x),n-1,y) \\
\label{commcompwithdlam}
&=(\psi_{\lambda}(x,n,y)-\psi_{\lambda}(\sigma(x),n-1,y))(S_{i}f)(x,n,y).
\end{align}

\begin{lem}
\label{compmatrkfak}
The function $\mathfrak{a}_\lambda(x,n,y):=\psi_\lambda(x,n,y)-\psi_\lambda(\sigma(x),n-1,y)$ satisfies the estimate $|\mathfrak{a}_\lambda(x,n,y)|\leq \max(2,2|\lambda|-1)$ for any $(x,n,y)\in \mathcal{G}_A$ and belongs to $C_b(\mathcal{G}_A)$. More precisely, 
$$\mathfrak{a}_\lambda(x,n,y)=$$
$$\begin{cases}
(2|\lambda|-1)&\chi_{Y_\lambda\cap \{c_A=|\lambda|\}}+\chi_{Y_\lambda\cap \{c_A>|\lambda|\}}-\chi_{\mathcal{G}^0_{>0}\setminus Y_\lambda}\\
&+2\chi_{\mathcal{G}_{\leq 0}\cap \{c_A+\kappa=0\}}-\chi_{\mathcal{G}_{>0}\setminus \mathcal{G}^0}+\chi_{\mathcal{G}_{\leq 0}\cap \{c_A+\kappa>0\}},\quad|\lambda|>0,\\
\\
\chi_{\mathcal{G}^0}&+2\chi_{\mathcal{G}_{< 0}\cap \{c_A+\kappa=0\}}-\chi_{\mathcal{G}_{>0}\setminus \mathcal{G}^0}+\chi_{\mathcal{G}_{\leq0}\cap \{c_A+\kappa>0\}},\quad\lambda=\circ_A.
\end{cases}$$
\end{lem}

\begin{proof}
We prove the Lemma by dividing into cases. Assume first that $\lambda$ is non-empty. Consider the following statements involving $(x,n,y)\in \mathcal{G}_A$:
\begin{enumerate}
\item[a.] $(x,n,y)\in Y_\lambda$;
\item[b.] $(x,n,y)\in \mathcal{G}^0\setminus Y_\lambda$;
\item[c.] $(x,n,y)\in \mathcal{G}\setminus \mathcal{G}^0$;
\item[$\alpha$.] $(\sigma(x),n-1,y)\in Y_\lambda$;
\item[$\beta$.] $(\sigma(x),n-1,y)\in \mathcal{G}^0\setminus Y_\lambda$;
\item[$\gamma$.] $(\sigma(x),n-1,y)\in \mathcal{G}\setminus \mathcal{G}^0$.
\end{enumerate}

\noindent We start by excluding the cases that do not occur:\\

\noindent(a and $\gamma$) cannot hold simultaneously, for this would be the case if and only if $n=0$ and $(x,n,y)\in Y_\lambda$ which is not possible for $n\geq |\lambda|>0$. \\

\noindent(b and $\alpha$) can not hold simultaneously, because b implies $n>0$ in which case $\alpha$ implies a. \\

\noindent(c and $\alpha$ or $\beta$) can not hold simultaneously, because c. implies $n<0$ or $n\geq 0$ and $\sigma^n(x)\neq y$ while $\alpha$. or $\beta$. implies $n>0$ and $\sigma^n(x)=y$. \\

\noindent We thus have five cases to consider:\\

\noindent(a and $\alpha$) This holds if and only if $n>|\lambda|$ and $(x,n,y)\in Y_\lambda$. In this case, $\mathfrak{a}_\lambda(x,n,y)=1$. The contribution from a. and $\alpha$. is therefore $\chi_{Y_\lambda\cap \{c_A>|\lambda|\}}$. \\

\noindent(a and $\beta$) This holds if and only if $n=|\lambda|>0$ and $(x,n,y)\in Y_\lambda$. In this case, $\mathfrak{a}_\lambda(x,n,y)=2|\lambda|-1$, and this contributes $(2|\lambda|-1)\chi_{Y_\lambda\cap \{c_A=|\lambda|\}}$ if $|\lambda|>0$.\\

\noindent(b and $\beta$) This holds if and only if $(x,n,y)\in \mathcal{G}^0\setminus Y_\lambda$ and $n>0$ in which case $\mathfrak{a}_\lambda(x,n,y)=-1$. As such, the contribution to $\mathfrak{a}_\lambda$ is  $-\chi_{\mathcal{G}^0_{>0}\setminus Y_\lambda}$. \\

\noindent(b and $\gamma$) This holds if and only if $n=0$ and $(x,n,y)\in \mathcal{G}^0\setminus Y_\lambda$ and if this is the case, $\mathfrak{a}_\lambda(x,n,y)=2$. Thus, b. and $\gamma$ contribute $2\chi_{\mathcal{G}^0_0\setminus Y_\lambda}=2\chi_{\mathcal{G}^0_0}$ to $\mathfrak{a}_\lambda$.\\

\noindent(c and $\gamma$) This case can be divided into four sub cases: 
\begin{enumerate}
\item $n> 0$ and $\sigma^n(x)\neq y$;
\item $n= 0$ and $x\neq y$;
\item $n<0$ and $n+\kappa(x,n,y)=0$;
\item $n<0$ and $n+\kappa(x,n,y)>0$.
\end{enumerate}
In the first case $\mathfrak{a}_\lambda(x,n,y)=-1$, contributing $-\chi_{\mathcal{G}_{>0}\setminus \mathcal{G}^0}$ to $\mathfrak{a}_\lambda$. In the second case, $\mathfrak{a}_\lambda(x,n,y)=1$, contributing $\chi_{\mathcal{G}_{0}\setminus \mathcal{G}^0_0}$ to $\mathfrak{a}_\lambda$. In the third case $\mathfrak{a}_\lambda(x,n,y)=2$, contributing $2\chi_{\mathcal{G}_{<0}\cap\{c_A+\kappa=0\}}$ to $\mathfrak{a}_\lambda$. In the fourth case  $\mathfrak{a}_\lambda(x,n,y)=1$, contributing $\chi_{\mathcal{G}_{<0}\cap \{c_A+\kappa>0\}}$ to $\mathfrak{a}_\lambda$. The total contribution from the case c. is therefore $-\chi_{\mathcal{G}_{>0}\setminus \mathcal{G}^0}+\chi_{\mathcal{G}_{\leq 0}\cap \{c_A+\kappa>0\}}+2\chi_{\mathcal{G}_{<0}\cap\{c_A+\kappa=0\}}$. \\

We only sketch the case when $\lambda$ is the empty word. If $\lambda$ is the empty word, the case-by-case analysis is similar but without the cases b. and $\beta$. The conditions a. and $\alpha$ hold if and only if a. holds and $n>0$, contributing $\chi_{\mathcal{G}_{>0}^0}$. Further, a. and $\gamma$ hold if and only if a. holds and $n=0$, contributing $\chi_{\mathcal{G}_{0}^0}$. So the contributions from the case a. is exactly $\chi_{\mathcal{G}^0}$. If c. holds, then $\gamma$ follows contributing in the same fashion as above the terms $-\chi_{\mathcal{G}_{>0}\setminus \mathcal{G}^0}+\chi_{\mathcal{G}_{\leq 0}\cap \{c_A+\kappa>0\}}+2\chi_{\mathcal{G}_{<0}\cap\{c_A+\kappa=0\}}$.
\end{proof}

\begin{thm}
\label{unbdd2} 
The operator $(\mathpzc{E}^{\Omega}_{A},D_{\lambda})$ is an odd unbounded $KK$-cycle for $(O_{A},C(\Omega_{A}))$, which defines the same class as the $(O_A,C(\Omega_A))$-Kasparov module $(\mathpzc{E}^{\Omega}_{A},2p_\lambda-1)$ does.
\end{thm}

\begin{proof} 
The operator $D_{\lambda}$ is $C(\Omega_{A})$-linear by construction. The operators $D_{\lambda}\pm i:C_{c}(\mathcal{G}_A)\rightarrow C_{c}(\mathcal{G}_A)$ are bijective since $D_{\lambda}$ is defined via multiplication by a real valued function. Thus, $D_{\lambda}$ extends to a selfadjoint regular operator in the module $\mathpzc{E}^{\Omega}_{A}$.  To prove that $D_{\lambda}$ has compact resolvent, we observe that the restriction of $D_{\lambda}^{2}$ to $\mathpzc{E}^{k}_{n}$ acts as multiplication by $(|n|+k)^{2}$, so since $\mathpzc{E}^{k}_{n}$ is finitely generated and projective, the resolvent $(1+D_{\lambda}^{2})^{-1}$ is compact.

It remains to show that $D_{\lambda}$ has bounded commutators with the generators $S_{i}$. This fact follows from Equation \eqref{commcompwithdlam} and Lemma \ref{compmatrkfak}. Since $\psi_\lambda$ is positive exactly on $Y_\lambda$, the class of this unbounded cycle coincides with that of $p_\lambda$ using Proposition \ref{rangevlambda}.
\end{proof}

\begin{remark}
It is also possible to construct even classes over $C(\Omega_A)$ from $c_A$ and $\kappa$. On the direct sum $\mathpzc{E}_{A}^{\Omega}\oplus \mathpzc{E}_{A}^{\Omega}$, consider the $O_{A}$ representation determined by $S_{i}\mapsto S_{i}\oplus S_i$ and the unbounded symmetric operator
\[D^{ev}:=\begin{pmatrix} 0 & D_{c} + iD_{\kappa} \\ D_{c}-iD_{\kappa} & 0\end{pmatrix}.\]
The pair $(\mathpzc{E}^{\Omega}_{A}\oplus\mathpzc{E}^{\Omega}_{A} ,D^{ev})$ defines a cycle for $KK_{0}(O_{A},C(\Omega_A))$. 
\end{remark}

\subsubsection{The operator $D_{\circ_A}$ on the free group}
Let us consider the construction of Theorem \ref{unbdd2} in the example of the free group, recalled above in Subsection \ref{freegroupexample}. The reader can verify that the function $\phi^*\kappa:\partial F_d\rtimes F_d\to \N$, where $\varphi$ denotes the groupoid isomorphism implementing the isomorphism of Proposition \ref{isofreegr}, is given by
$$\phi^*\kappa(x,\gamma)=\ell(x,\gamma).$$
In particular, for the empty word $\lambda=\circ_A$, it holds that 
$$\phi^*\psi(x,\gamma)=
\begin{cases}
\quad |\gamma| &\textnormal{when } \ell(x,\gamma)=0 \\ 
\vspace{-2mm}\\
-\left||\gamma|-2\ell(x,\gamma)\right|-\ell(x,\gamma)&\textnormal{when } \ell(x,\gamma)>0.
\end{cases}$$

\subsubsection{Quick computation for $SU_q(2)$} Recall the construction from Subsubsection \ref{suqtwoexample}.

\begin{prop}
\label{suq2prop}
If $\tau=(\tau_+,\tau_-):\mathcal{V}_{SU_q(2)}\to \Omega_{SU_q(2)}\times \Omega_{SU_q(2)}$ is a weak choice function such that $\tau_+(\circ_A)\in C_2$ and $\tau_-(\circ_A)\in C_1$, then the class $[\mathpzc{E}^\Omega_{SU_q(2)},D_2]\otimes_{C(\Omega_{SU_q(2)})} [BP(\tau)]$ generates $K^1(C(SU_q(2)))$.
\end{prop}

We use the identification $\mathcal{V}_{SU_q(2)}\cong \N\times \N$ given by the mapping that maps $(k,l)$ to the word $1\cdots 12\cdots 2$ of $k$ $1$:s and $l$ $2$:s. 

\begin{proof}
It is well-known (see more in Equation \eqref{kgroupsforsucue}), that $K_1(C(SU_q(2)))\cong \Z\cong K^1(C(SU_q(2)))$. Hence, the Universal Coefficient Theorem for $KK$-theory implies that the index pairing 
$$K_1(C(SU_q(2)))\otimes K^1(C(SU_q(2)))\to \Z$$ 
is non-degenerate and in fact an isomorphism. Thus, it suffices to construct a unitary $u\in C(SU_q(2))$ such that the class $x:=[u]\otimes_{C(SU_q(2))} [\mathpzc{E}^\Omega_{SU_q(2)},D_2]\in K_0(C(\Omega_{SU_q(2)}))$ satisfies that $x\otimes_{C(\Omega_{SU_q(2)})} [BP(\tau)]=-1$.

Consider the unitary $u:=S_2+1-S_2S_2^*=S_2+S_1S_1^*$. We set $T:=p_2 u p_2\in \mathrm{End}^*_{C(\Omega_{SU_q(2)})}(p_2\mathpzc{E}^\Omega_{SU_q(2)})$, so $x=\ind_{C(\Omega_{SU_q(2)})}(T)$. It holds that $p_2\mathpzc{E}^\Omega_{SU_q(2)}$ is generated over $C(\Omega_{SU_q(2)})$ by the elements $\{S_{(k,l)}:l>0\}$. A direct computation gives that 
\[TS_{(k,l)}=
\begin{cases}
S_{(k,l)},\quad k>0,\\
S_{(0,l+1)},\quad k=0,
\end{cases}
\quad \mbox{and}\quad 
T^*S_{(k,l)}=
\begin{cases}
S_{(k,l)},\quad k>0,\\
S_{(0,l-1)},\quad k=0,\; l>1,\\
0, \quad k=l-1=0.
\end{cases}
\quad
\]
It follows that $\ker T=0$ and $\ker T^*=S_2C(\Omega_{SU_q(2)})\cong \chi_{C_2}C(\Omega_{SU_q(2)})$. Hence $x=\ind_{C(\Omega_{SU_q(2)})}(T)=-[\chi_{C_2}]$. It follows that $x\otimes_{C(\Omega_{SU_q(2)})} [BP(\tau)]=-1$ from the computation \eqref{mucomputations}.

\end{proof}

\subsection{Restricting to a fiber} 

In this subsection we will exhaust all the odd $K$-homology classes of $O_A$ by the unbounded Fredholm modules that are restrictions of the unbounded $(O_A,C(\Omega_A))$-cycles $(\mathpzc{E}^{\Omega}_{A},D_{\lambda})$ of Theorem \ref{unbdd2} to ``fibers" over points in $\Omega_A$. Whenever $\omega$ is a character on $C(\Omega_A)$, we say that $\omega$ starts in $j$ if the word that $\omega$ corresponds to starts in $j$, i.e. $\omega(\chi_{C_{k\mu}})=\delta_{k,j}\omega(\chi_{C_{k\mu}})$ for any $\mu\in \mathcal{V}_A$. Before formulating the precise result on these unbounded Fredholm modules, we need a lemma whose notation will come in handy. Recall the notation $\mathcal{V}_\lambda=\{\mu\lambda\in \mathcal{V}_A\}$.

\begin{lem}
\label{iotacomp}
For any character $\omega$ starting in $j$, there is a partial isometry $\iota_\omega:\ell^2(\mathcal{V}_A)\to \mathpzc{E}^{\Omega}_{A}\otimes_\omega \C$ such that the source projection is the orthogonal projection onto $\C\delta_{\circ_A}\oplus \bigoplus_{A(k,j)\neq 0} \ell^2(\mathcal{V}_k)$ and 
$$\iota_\omega(\delta_\mu)=
\begin{cases}
1_{O_A}\otimes_\omega 1_\C,\quad&\mbox{if}\; \mu=\circ_A,\\
S_\mu\otimes_{\omega} 1_\C,\quad&\mbox{if}\; \mu\in \mathcal{V}_A\setminus \{\circ_A\}.
\end{cases}$$
\end{lem}

\begin{proof}
The identity $\iota_\omega(\delta_\mu)=S_\mu\otimes_{\omega} 1$ and $\iota_\omega(\delta_{\circ_A})=1_{O_A}\otimes_\omega 1_\C$ determines a linear mapping $C_c(\mathcal{V}_A)\to  \mathpzc{E}^{\Omega}_{A}\otimes_\omega \C$. Let $\mathfrak{P}_k:=R^A_k(R^A_k)^*$ denote the orthogonal projection onto $\ell^2(\mathcal{V}_k)$. Since $1_{O_A}\otimes_\omega 1_\C$ is a unit vector in $\mathpzc{E}^{\Omega}_{A}\otimes_\omega \C$, it suffices to prove that for arbitrary $\mu,\nu\in \mathcal{V}_A$, with $\mu=\mu_0k$, it holds that 
\begin{equation}
\label{iotaepsident}
\langle \iota_\omega \delta_\mu,\iota_\omega \delta_\nu\rangle_{\mathpzc{E}^{\Omega}_{A}\otimes_\omega \C}=
\langle \sum_{A(j,k)\neq 0}\mathfrak{P}_k\delta_\mu,\delta_\nu\rangle=
\begin{cases}
\delta_{\mu,\nu},\;& \mbox{if}\;\; A(k,j)\neq 0,\\
0, \;& \mbox{otherwise}.
\end{cases}
\end{equation}
Let $\mu=\mu_0k$. A direct computation shows that 
\begin{align*}
\langle \iota_\omega \delta_\mu,\iota_\omega \delta_\nu\rangle_{\mathpzc{E}^{\Omega}_{A}\otimes_\omega \C}&=\omega(S_\mu
^*S_\nu)=\delta_{\mu,\nu} \omega(S_k^*S_k)\\
&=\delta_{\mu,\nu} \sum_{i=1}^N A_{ki} \omega (S_iS_i^*)=\begin{cases}
\delta_{\mu,\nu},\;& \mbox{if}\;\; A(k,j)\neq 0,\\
0, \;& \mbox{otherwise}.
\end{cases}\end{align*}
\end{proof}

\begin{remark}
\label{algecomput}
Already on an algebraic level, 
\begin{align*}
\iota_\omega(\delta_{\mu k})=S_{\mu k}\otimes _\omega 1=S_{\mu k}S_k^*S_k\otimes _\omega 1=\sum_{l=1}^N A_{kl}S_{\mu k}S_lS_l^*\otimes _\omega 1\\
=\sum_{l=1}^N A_{kl}S_{\mu k}\otimes _\omega \omega(\chi_{C_l})=A_{kj}S_{\mu k}\otimes_\omega1.
\end{align*}
\end{remark}

Let $\lambda\in \mathcal{V}_A$ be a finite word, if $\lambda$ is non-empty we let $\lambda_\ell$ denote the last letter of $\lambda$. We define the partial isometry $W_{\lambda,\omega}:\ell^2(\mathcal{V}_\lambda)\to \mathpzc{E}^{\Omega}_{A}\otimes_\omega \C$ by 
$$W_{\lambda,\omega}:=\iota_\omega|_{\ell^2(\mathcal{V}_\lambda)}.$$ 
By Lemma \ref{iotacomp} it holds that $W_{\lambda,\omega}$ is an isometry if $\lambda$ is non-empty and $A(\lambda_\ell,j)=1$. If $\lambda$ is non-empty and $A(\lambda_\ell,j)=0$, then $W_{\lambda,\omega}$ is a partial isometry of rank $1$ with source projection being the one-dimensional space $\C\delta_{\circ_A}$. If $\lambda=\circ_A$, the partial isometry $W_{\lambda,\omega}$ is precisely $\iota_\omega$.

We let $\pi_A^\Omega:O_A\to \End^*_{C(\Omega_A)}(\mathpzc{E}^\Omega_A)$ denote the left $O_A$-action. Let $P_\omega\in \Ko(\mathpzc{E}^\Omega_A\otimes_\omega \C)$ denote the orthogonal projection onto the one-dimensional space 
$$\ker D_\lambda\otimes_\omega 1=\C\iota_\omega(\delta_{\circ_A})=\C 1_{O_A}\otimes_\omega 1_\C.$$
These identities follow from the definition of $\psi_\lambda$, see Definition \ref{psilambda}.

\begin{thm}
\label{computingkhomclasseslikeaboss}
Let $\omega:C(\Omega_A)\to \C$ be a character starting in $j$. For a finite word $\lambda\in \mathcal{V}_A$, the unbounded Fredholm module
\begin{equation}
\label{pushdefspe}
\omega_*(\mathpzc{E}^{\Omega}_{A},D_{\lambda})=(\pi_A^\Omega\otimes_\omega\id_\C, \mathpzc{E}^{\Omega}_{A}\otimes_\omega \C, D_{\lambda}\otimes_\omega 1),
\end{equation}
is $\theta$-summable. If $\phi(l)\leq C l^p$ for some $C,p>0$, then $\omega_*(\mathpzc{E}^{\Omega}_{A},D_{\lambda})$ is $\mathcal{L}^{p+1,\infty}$-summable. Furthermore, it holds that the phase of the unbounded Fredholm module \eqref{pushdefspe} coincides with the finitely summable analytic $K$-cycle: 
$$(\pi_A^\Omega\otimes_\omega \id_\C,\mathpzc{E}^{\Omega}_{A}\otimes_\omega \C,2W_{\lambda,\omega}W_{\lambda,\omega}^*\pm P_\omega-1),$$
where the sign is $+$ is $\lambda\neq \circ_A$ and the sign is $-$ if $\lambda=\circ_A$. On the level of $K$-homology, it holds that 
\begin{equation}
\label{khomcomputationmoreexactly}
\omega_*[\mathpzc{E}^{\Omega}_{A},D_\lambda]=
\begin{cases}
[\beta_j], \quad &\lambda=\circ_A,\\
\vspace{-3mm}\\
A(\lambda_\ell,j)[\beta_{\lambda_1}],\quad &\lambda=\lambda_1\cdots \lambda_\ell\in \mathcal{V}_A\setminus \{\circ_A\}
\end{cases}
\qquad \mbox{in} \quad K^1(O_A).
\end{equation}
\end{thm}

\begin{remark}
In fact, it follows from Lemma \ref{iotacomp} that if $\lambda$ is non-empty and $A(\lambda_\ell,j)=0$ then $2W_{\lambda,\omega}W_{\lambda,\omega}^*+P_\omega-1=P_\omega-1$. Hence, the computations of Theorem \ref{computingkhomclasseslikeaboss} imply that the phase of the unbounded Fredholm module \eqref{pushdefspe} is modulo $P_\omega$ a degenerate cycle, as such it is $K$-homologically trivial in a very strong sense. 
\end{remark}

Recall the notation $\beta_k$ from Proposition \ref{representingbetai}. We wish to remark\footnote{For the sake of mental peace of the reader.} that since $(\mathpzc{E}^{\Omega}_{A},D_{\lambda})$ is an unbounded $KK$-cycle, functoriality of unbounded $KK$-cycles guarantees that $\omega_*(\mathpzc{E}^{\Omega}_{A},D_{\lambda})$ is an unbounded Fredholm module. As such, the proof consists of proving $\theta$-summability and identifying its bounded transform. We structure the proof of the later in a Proposition.

\begin{prop}
\label{posspeccom}
Let $\omega$ be a character on $C(\Omega_A)$, $\lambda\in\mathcal{V}_A$ and define $\mathpzc{K}^{\,\lambda}_{\;\;\omega}$ as the closed linear span of $\{S_\mu\otimes_\omega 1|\mu=\mu_0\lambda\}\subseteq  \mathpzc{E}^{\Omega}_{A}\otimes_\omega \C$. It holds that the positive spectral projection of $D_\lambda\otimes_\omega 1$ is the orthogonal projection onto $(1-P_\omega)\mathpzc{K}^{\,\lambda}_{\;\;\omega}\subseteq  \mathpzc{E}^{\Omega}_{A}\otimes_\omega \C$. In particular, if $\omega$ starts in $j$ and $A(\lambda_\ell,j)=0$, where $\lambda_\ell$ is the last letter of $\lambda$, then $\mathpzc{K}^{\,\lambda}_{\;\;\omega}=0$.
\end{prop}

The proof of the first part of Proposition \ref{posspeccom} is clear from Proposition \ref{rangevlambda} and the proof of Theorem \ref{unbdd2}. The second part follows from the first part and Lemma \ref{iotacomp} (cf. Remark \ref{algecomput}).

\begin{proof}[Proof of Theorem \ref{computingkhomclasseslikeaboss}]
It follows from Proposition \ref{posspeccom} and Lemma \ref{iotacomp} that if $\lambda$ is non-empty, the projection onto the positive spectrum of $D_\lambda\otimes_\omega 1$ coincides with $W_{\lambda,\omega}W_{\lambda,\omega}^*$. If $\lambda$ is empty, the projection onto the non-negative spectrum of $D_\circ\otimes_\omega 1$ coincides with $W_{\circ,\omega}W_{\circ,\omega}^*$. In our convention, declaring $|D_\lambda\otimes_\omega 1|^{-1}$ to be $0$ on $\ker D_\lambda\otimes_\omega 1$, it holds that
$$\frac{D_\lambda\otimes_\omega 1}{|D_\lambda\otimes_\omega 1|}=
\begin{cases}
2W_{\lambda,\omega}W_{\lambda,\omega}^*+P_\omega-1, \quad&\mbox{if}\; \lambda\neq \circ_A\\
\vspace{-2mm}\\
2W_{\circ,\omega}W_{\circ,\omega}^*-P_\omega-1, \quad&\mbox{if}\; \lambda= \circ_A
\end{cases}.$$
Hence, if $\lambda$ is non-empty and $A(\lambda_\ell,j)=0$, Equation \eqref{khomcomputationmoreexactly} follows. To prove Equation \eqref{khomcomputationmoreexactly} for a non-empty $\lambda$ with $A(\lambda_\ell,j)=1$, we apply the ideas of Subsection \ref{subsectionfinsuminkon} after computing 
\[W_{\lambda,\omega}^*\left[(\pi_A\otimes_\omega \id_\C)(S_i )\right]W_{\lambda,\omega}=L_i^A|_{\ell^2(\mathcal{V}_\lambda)},\quad i=1,\ldots, N.\] 
The identity \eqref{khomcomputationmoreexactly} and finite summability follows mutatis mutandis to the proof of Proposition \ref{representingbetai} using the fact that $\ell^2(\mathcal{V}_\lambda)=R_{\bar{\lambda}}^A(R_{\bar{\lambda}}^A)^*\ell^2(\mathcal{V}_A)$ and in the $K$-theory of $O_{A^T}$ it holds that 
$$T_{\bar{\lambda}} T_{\bar{\lambda}}^*\sim T_{\bar{\lambda}}^*T_{\bar{\lambda}}=T_{\lambda_1}^*T_{\lambda_1}\sim T_{\lambda_1}T_{\lambda_1}^*.$$ 

If $\lambda=\circ_A$, it follows from Proposition \ref{iotacomp} that 
\[W_{\circ,\omega}^*\left[(\pi_A\otimes_\omega \id_\C)(S_i )\right]W_{\circ,\omega}=W_{\circ,\omega}^*W_{\circ,\omega}L_i^A|_{\C\delta_{\circ_A}\oplus \bigoplus_{A(k,j)\neq 0} \ell^2(\mathcal{V}_k)},\quad i=1,\ldots, N.\] 
Hence $W_{\circ,\omega}^*\left[(\pi_A\otimes_\omega \id_\C)(S_i )\right]W_{\circ,\omega}-L_i^A|_{\bigoplus_{A(k,j)\neq 0} \ell^2(\mathcal{V}_k)}$ is of finite rank. An argument similar to that in Subsection \ref{subsectionfinsuminkon} shows that
$$\left[\pi_A^\Omega\otimes_\omega \id_\C,\mathpzc{E}^{\Omega}_{A}\otimes_\omega \C,2W_{\circ,\omega}W_{\circ,\omega}^* -P_\omega-1\right]=\sum_{l=1}^N A(l,j)[\beta_l]=[\beta_j]$$

It remains to prove $\theta$-summability, i.e. that $\e^{-(D_\lambda\otimes_\omega 1)^2}$ is trace class. Applying the computations of Proposition \ref{projsum} and the definition of $D_\lambda$, we have that 
$$\mathpzc{E}^{\Omega}_{A}\otimes_\omega \C=\bigoplus_{n\in\Z}\bigoplus_{\substack{l\in \N\\l+n\geq  0}}\omega(p_{n,l})\C^{\phi(l+n)}$$
and in this decomposition 
$$(D_\lambda\otimes_\omega 1)^2=\bigoplus_{n\in\Z}\bigoplus_{\substack{l\in \N\\l+n\geq  0}}(|n|+l)^2\omega(p_{n,l}).$$ 
It follows from Corollary \ref{phiasymptotics} that $\e^{-(D_\lambda\otimes_\omega 1)^2}$ is trace class. Assuming that $\phi(l)\leq Cl^p$ for some $p$ implies that $|D_\lambda \otimes_\omega 1|^{-1}\in \mathcal{L}^{p+1,\infty}(\mathpzc{E}^{\Omega}_{A}\otimes_\omega \C)$; in this case, $\omega_*(\mathpzc{E}^{\Omega}_{A},D_{\lambda})$ is a $\mathcal{L}^{p+1,\infty}$-summable unbounded Fredholm module.
\end{proof}

\begin{remark}
In particular, Theorem \ref{computingkhomclasseslikeaboss} implies that for a choice of characters $\omega_1, \omega_2, \ldots, \omega_N$ such that each $\omega_k$ starts in a letter $k$, the mapping 
$$\Z^N\to K^1(O_A), \quad  (l_1,l_2,\ldots, l_N)\mapsto \sum_{k=1}^N l_k \left[(\omega_k)_*(\mathpzc{E}^{\Omega}_{A},D_\circ)\right] \quad\mbox{is surjective}.$$
This gives an explicit proof of the fact that the Kasparov product 
\[KK_1(O_A,C(\Omega_A))\otimes K^0(C(\Omega_A))\to K^1(O_A)\quad\mbox{is surjective}.\]
\end{remark}

\begin{remark}
If the matrix $A$ is irreducible or has property $(I)$, Theorem \ref{variousstructuresonoa} implies that the unbounded Fredholm modules $\omega_*(\mathpzc{E}^\Omega_A,D_\lambda)$ in fact are spectral triples on $O_A$. 
\end{remark}

\Large
\section{Kasparov products with the Bellissard-Pearson spectral triples}
\label{bpkaspprodsection}
\normalsize

The point localizations of the previous section form a simple case of the Kasparov product in $KK$-theory. We will describe the Kasparov products of the $(O_{A},C(\Omega_{A}))$-cycles with the Bellissard-Pearson spectral triples, via the operator space approach to connections \cite{BMS, KaLe, Mes}. It turns out that, by naively applying these techniques, we obtain a $1-s$-unbounded Fredholm module (see the appendix) from any cycle $(\mathpzc{E}^\Omega_A,D_{\lambda})$, with $\lambda$ a finite word, and any Bellissard-Pearson spectral triple $(\pi_{\tau},\ell^{2}(\mathcal{V}_{A},\C^{2}), D_{\mathcal{V},s})$ for $s\in (0,1)$. The case $s=1$ is excluded as the theory of $\epsilon$-unbounded Fredholm modules breaks down at $\epsilon=0$. First, we will briefly recall the techniques developed in \cite{BMS}.

\setcounter{thm}{0}

\begin{deef} 
Let $(\pi,\mathpzc{H},D)$ be a unbounded Fredholm module. Its \emph{Lipschitz algebra} is as in Definition \ref{unbddfred} (see page \pageref{unbddfred}) defined to be the $*$-algebra
\begin{equation}
\label{lipdeef}
\mathcal{\cstar }_{D}=\textnormal{Lip}(\pi,\mathpzc{H},D):=\{a\in \cstar : [D,a]\in\mathbb{B}(\He)\}.
\end{equation}
\end{deef} 

This algebra is the maximal subalgebra of $\cstar $ such that $[D,a]$ is bounded for any $a$. The algebra $\mathcal{\cstar }_D$ can be topologized by the representation
\vspace{1mm}
\begin{align*}
\tilde{\pi}_D:=\id\oplus \pi_{D}:\mathcal{\cstar }_D&\rightarrow \cstar \oplus\mathbb{B}(\mathpzc{H}\oplus\mathpzc{H}) \\
\mbox{where} \quad&\pi_D:a \mapsto \begin{pmatrix} \pi(a) & 0 \\ [D,\pi(a)] & \pi(a)\end{pmatrix},
\end{align*}
realizing $\mathcal{\cstar }_D$ as a closed subalgebra of $\cstar \oplus\mathbb{B}(\mathpzc{H}\oplus\mathpzc{H})$. As such it is an \emph{operator algebra}. The reader can consult \cite{BlecherleM} for an exposition of the general theory of nonselfadjoint operator algebras. The involution in $\cstar $ induces an involution in $\mathcal{\cstar }_D$, which is well behaved with respect to the representation $\pi_{D}$. Indeed,
\[\pi_{D}(a^{*})=v^{*}\pi_{D}(a)^{*}v,\quad \mbox{where}\quad v=\begin{pmatrix} 0 & -1 \\ 1 & 0\end{pmatrix},\]
which implies that the involution is \emph{completely isometric} for the norm induced by $\tilde{\pi}_{D}$. Operator algebras equipped with a completely bounded involution are called \emph{involutive operator algebras} \cite{BMS, Mes} and \emph{operator $*$-algebras} in \cite{KaLe}. The main feature of involutive operator algebras is that there is a class of modules over them, which in many ways behave like Hilbert $C^*$-modules. We recall the theory for Lipschitz algebras.

\begin{deef}[\cite{KaLe, Mes}] 
Let $\mathcal{\cstar}_{D}$ be a unital Lipschitz algebra. The \emph{standard free module over} $\mathcal{\cstar}_D$ is the module
\[\mathpzc{H}_{\mathcal{\cstar}}:=\left\{(a_{i})_{i\in\Z}\in \prod_{i\in \Z}\mathcal{\cstar}_D:\sum_{i\in\Z}\tilde{\pi}_{D}(a_{i})^{*}\tilde{\pi}_{D}(a_{i})<\infty\right\}.\]
\end{deef}

The module $\mathpzc{H}_{\mathcal{\cstar}}$ carries an $\mathcal{\cstar}_D$-valued inner product, but this inner product does not define the norm.  The \emph{algebra of adjointable operators} $\End^{*}_{\mathcal{\cstar}}(\mathpzc{H}_{\mathcal{\cstar}})$ consists of those completely bounded operators $T:\mathpzc{H}_{\mathcal{\cstar}}\to \mathpzc{H}_{\mathcal{\cstar}}$ that admit an adjoint with respect to the inner product. The existence of unbounded projections in $\mathpzc{H}_{\mathcal{\cstar}}$ is due to the fact that norm and inner product are not related in the same way as they are in Hilbert $C^*$-modules.  A \emph{projection} is a closed densely defined operator satisfying $p^{2}=p^{*}=p$. In \cite[Definition 2.27]{BMS}, a \emph{Lipschitz module} over $\mathcal{\cstar}_D$ is defined to be a closed submodule $\mathcal{E}\subset\mathpzc{H}_{\mathcal{\cstar}}$ which is the range of a densely defined (possibly unbounded) projection $p:\Dom p\rightarrow\mathpzc{H}_{\mathcal{\cstar}}$ that decomposes as a direct sum $p=\bigoplus_{i\in I} p_{i}$ of projections $p_{i}\in\End^{*}_{\mathcal{\cstar}}(\mathpzc{H}_{\mathcal{\cstar}})$ for some countable set $I$. The algebra $\mathbb{K}(\mathcal{E})$ is defined to be the cb-norm closure of the $\mathcal{\cstar}_D$-linear finite rank operators on $\mathcal{E}$. 
\begin{prop}[\cite{BMS}] \label{projdirectsum}
For each $i\in\Z$, let $p_{i}\in M_{n_{i}}(\mathcal{\cstar}_D)$ be a projection and $\mathcal{E}_{i}:=p_{n_{i}}\mathcal{\cstar}_D^{n_{i}}\subset\mathcal{\cstar}_D^{n_{i}}$. Then the direct sum $\bigoplus_{i\in\Z}\mathcal{E}_{i}$ is a Lipschitz module.
\end{prop} 
The main feature of Lipschitz modules is the existence of connections on them.  Recall that the space of $1$-forms associated to $(\pi,\mathpzc{H},D)$ is 

\[\Omega^{1}_{D}:=\left\{\sum_{i}\pi(a_{i})[D,\pi(b_{i})]: a_{i}\in \cstar,b_{i}\in \mathcal{\cstar}_{D}\right\}\subset \B(\mathpzc{H}),\]
where the sums converges in operator norm. The operator space $\Omega^{1}_{D}$ is a left $\cstar$-module and a right $\mathcal{\cstar}_D$-module. The map $a\mapsto [D,a]$ is a completely bounded derivation $\mathcal{\cstar}_D\to \Omega^1_D$.
A $D$-\emph{connection} on a Lipschitz module $\mathcal{E}$ is a completely bounded map \[\nabla:\mathcal{E}\rightarrow\mathpzc{E}\hotimes_{\cstar}\Omega^{1}_{D},\] where $\hotimes$ denotes the  \emph{Haagerup module tensor product} (see \cite{BlecherleM} for the general construction and  \cite[Section $3.2$]{Mes} and the papers \cite{BMS, KaLe} for its use in the context of $KK$-theory), satisfying the Leibniz rule 
$$\nabla(ea)=\nabla(e)a+e\otimes [D,a],$$ 
for $e\in \mathcal{E}$ and $a\in \mathcal{\cstar}_D$. By \cite{BMS}, connections on Lipschitz modules always exist, since the \emph{Grassmann connection} $p[D,p]$ is completely bounded by construction.

\subsection{A connection on the Haar module}\label{connsubsec}

We now employ the machinery described above to construct a Lipschitz submodule $\mathcal{E}^{\Omega}_{A}\subset \mathpzc{E}^{\Omega}_{A}$ for any given logarithmic Bellissard-Pearson spectral triple $BP(\tau)$. By Proposition \ref{projdirectsum} it suffices to show that the Haar module $\mathpzc{E}^{\Omega}_{A}$ is a direct sum of finitely generated projective modules over $C(\Omega_{A})$, which is the content of Proposition \ref{projsum}. The following lemma serves in making the associated Lipschitz structure explicit.

\begin{lem}
\label{lipproj} 
Let $(\pi_{\tau}, \ell^{2}(\mathcal{V}_{A},\C^{2}), D_{\mathcal{V},s})$ be a logarithmic Bellissard-Pearson spectral triple. The projections $p_{n,k,\lambda}:=v_{n,k,\lambda}v^{*}_{n,k,\lambda}\in M_{\phi(n+k)}(C(\Omega_{A}))$ are in fact elements of $M_{\phi(n+k)}(\textnormal{Lip}(\Omega_{A},d_{\Omega_A}))$, and therefore $[D_{\mathcal{V},s},p_{n,k,\lambda}]\in \B(\ell^{2}(\mathcal{V}_{A},\C^{2}))$. 
\end{lem}

\begin{proof} 
The projection $v_{n,k,\lambda}v_{n,k,\lambda}^{*}\in M_{\phi(n+k)}(C(\Omega_{A}))$ has entries 
$$\left[v_{n,k,\lambda}v_{n,k,\lambda}^{*}\right]_{\mu,\nu}=\rho\left(\left(\chi^{k}_{n,\mu\lambda}\right)^{*}\chi^{k}_{n,\nu\lambda}\right)$$ 
which equal $0$ if $\mu\neq \nu$. For $\mu\in \mathcal{V}_A$ of length $n+k$ the convolution product gives
\[\left(\chi^{k}_{n,\mu}\right)^{*}\chi^{k}_{n,\mu}(x)=\sum\chi^{k}_{n,\mu}(z,n,x)=\left\{\begin{matrix} 1 & \textnormal{if } A_{\mu_{n+k},x_{k+1}}=1\textnormal{ and } \mu_{n+k}\neq x_{k}\\ 0 &\textnormal{otherwise}\end{matrix}\right.\]
Thus, for $k=0$, this function equals the projection
\[\sum_{i=1}^{N} A_{\mu_{n},i}\chi_{C_{i}},\]
whereas, for $k>0$, we get
\[\sum_{j=1}^{N}\sum_{i\neq\mu_{n+k}}A_{\mu_{n+k},i}(\sigma^{k-1})^*\chi_{C_{ij}}.\]
Since these are sums of shifted cylinder functions, it is Lipschitz in the metric $d_{\Omega_A}$. It follows that the projection $v_{n,k,\lambda}v_{n,k,\lambda}^{*}$ is a matrix of functions that are Lipschitz in the metric $d_{\Omega_A}$. The proposition follows from Proposition \ref{thetaandsummable}.
\end{proof}

In view of this fact, Proposition \ref{projsum} and Lemma \ref{lipproj} imply that the module $\mathpzc{E}^{k}_{k}$ admits a submodule $\mathcal{E}^{n}_{k}$ with the structure of a projective operator module over the involutive operator algebra
\begin{align*}
\textnormal{Lip}_{\tau,s}&(\Omega_{A}):=\textnormal{Lip}(\pi_\tau,\ell^2(\mathcal{V}_A,\C^2),D_{\mathcal{V},s})\\
&=\left\{f\in C(\Omega_A):\begin{pmatrix} \pi_{\tau}(f) & 0 \\ [D_{\mathcal{V},s},\pi_{\tau}(f)] &  \pi_{\tau}(f)\end{pmatrix}\in \Bo(\ell^2(\mathcal{V}_A,\C^2)\oplus\ell^2(\mathcal{V}_A,\C^2)) \right\}.
\end{align*}
This uses the fact that Proposition \ref{thetaandsummable} implies that there is a continuous inclusion $\textnormal{Lip}(\Omega_{A},d_{\Omega_A})\hookrightarrow \textnormal{Lip}_{\tau,s}(\Omega_{A})$ for any $s\in (0,1]$. Denote by $\mathcal{E}^{\Omega}_{A}\subset\mathpzc{E}^\Omega_A$ the submodule
\[\mathcal{E}^{\Omega}_{A}:=\left\{f\in\mathpzc{E}^\Omega_A:\sum_{\substack{n,k,\mu\\|\mu|=n+k}}\tilde{\pi}_{D}(\rho(\chi^{k*}_{n,\mu}f))^{*}\tilde{\pi}_{D}(\rho(\chi_{n,\mu}^{k*}f))<\infty\right\},\]
which is complete in the norm
\begin{equation}
\label{Lipnorm}
\|f\|^{2}_{\mathcal{E}}:=\left\|\sum_{n,k,\mu}\tilde{\pi}_{D}(\rho(\chi^{k*}_{n,\mu}f))^{*}\tilde{\pi}_{D}(\rho(\chi_{n,\mu}^{k*}f))\right\|_{C(\Omega_A)\oplus \Bo(\ell^2(\mathcal{V},\C^4))}.
\end{equation}
To reduce notation, we suppress the dependence on $s$ in $\mathcal{E}^{\Omega}_{A}$ in our notation. We reduce notation further by setting $\Omega^1_\tau:=\Omega^1_{D_{\mathcal{V},s}}$, which depends on $\tau$ through the representation of $C(\Omega_A)$. The norm in \eqref{Lipnorm} is compatible with the projective module decomposition \eqref{lambdadecomp}. There is a connection 
\[\begin{split}\nabla_{n}^{k}:\mathcal{E}^{k}_{n}&\rightarrow \mathpzc{E}^{k}_{n}\hotimes_{C(\Omega_{A})}\Omega^{1}_{\tau}\\
f&\mapsto v^{*}_{n,k}\otimes[D_{\mathcal{V},s},\rho(v_{n,k}*f)],\end{split}\] whose direct sum extends to a connection
\[\nabla:\mathcal{E}^{\Omega}_{A}\rightarrow \mathpzc{E}^{\Omega}_{A}\hotimes_{C(\Omega_{A})}\Omega^{1}_{\tau}.\]

\begin{lem} 
\label{lipresolv}
The module $\mathcal{E}^{\Omega}_{A}$ is dense $\mathpzc{E}^{\Omega}_{A}$ and $\mathcal{E}^{\Omega}_{A}$ is a Lipschitz module in the norm \eqref{Lipnorm}. The operator $D_{\lambda}$ restricts to a selfadjoint regular operator in $\mathcal{E}^{\Omega}_{A}$, and $(D_{\lambda}\pm i)^{-1}\in\K(\mathcal{E}^{\Omega}_{A})$. Moreover, $[D_{\lambda},\nabla]=0$.
\end{lem}

\begin{proof} 
To see that $\mathcal{E}^{\Omega}_{A}$ is dense in $\mathpzc{E}^{\Omega}_{A}$, observe that the finitely generated projective $\Lip_{\tau,s}(\Omega_{A})$-module
\[\mathcal{E}_{n}^{k}:=\left\{f\in \mathpzc{E}^{n}_{k}:v_{n,k}f\in \Lip_{\tau,s}(\Omega_{A})^{\phi(n+k)}\right\}\subset\mathcal{E}^{\Omega}_{A},\]
is dense in $\mathpzc{E}^{k}_{n}$. The $\Lip_{\tau,s}(\Omega_{A})$-module $\mathcal{E}^{\Omega}_{A}$ contains the algebraic direct sum of the $\mathcal{E}^{k}_{n}$ as a dense submodule. Since the norm \eqref{Lipnorm} comes from the embedding
\[\begin{split}v:\mathcal{E}^{\Omega}_{A}&\rightarrow \bigoplus_{n,k,\mu}\Lip_{\tau,s}(\Omega_{A})^{\phi(n+k)}\cong\mathpzc{H}_{\Lip_{\tau}^s(\Omega_{A})}\\
f&\mapsto (\rho(\chi^{k*}_{n,\mu}f))_{n,k,\mu},\end{split}\]
$\mathcal{E}^{\Omega}_{A}$ is a Lipschitz module. We now prove that the resolvents $(D_{\lambda}\pm i)^{-1}$ are completely bounded for the Lipschitz norm. The Lipschitz norm is given by \eqref{Lipnorm}, for $f\in \mathcal{E}^{\Omega}_{A}$ we have
\small
\[\|(D_{\lambda}\pm i)^{-1}f\|^{2}_{\mathcal{E}}=\left\|\sum_{n,k,\mu}\tilde{\pi}_{D}(\rho(\chi^{k*}_{n,\mu}(D_{\lambda}\pm i)^{-1}f))^{*}\tilde{\pi}_{D}(\rho(\chi_{n,\mu}^{k*}(D_{\lambda}\pm i)^{-1}f))\right\|_{C(\Omega_A)\oplus \Bo(\ell^2(\mathcal{V},\C^4))},\]
\normalsize
and this norm identity is compatible with the projective module decomposition \eqref{lambdadecomp}. Thus (although $D_{\lambda}$ depends on whether $k=0$ or $k>0$ and $\mu\in\mathcal{V}_{\lambda}$ or not) for fixed $n$, $k$, $\mu$, we have 
\begin{align*}
\tilde{\pi}_{D}(\rho(\chi^{k*}_{n,\mu}&(D_{\lambda}\pm i)^{-1}f))^{*}\tilde{\pi}_{D}(\rho(\chi_{n,\mu}^{k*}(D_{\lambda}\pm i)^{-1}f))\\
&\leq (1+n^{2}+k^{2})^{-1}\tilde{\pi}_{D}(\rho(\chi^{k*}_{n,\mu}f))^{*}\tilde{\pi}_{D}(\rho(\chi_{n,\mu}^{k*}f)),
\end{align*}
by definition of $\psi_{\lambda}$, see Equation \eqref{psilambda}. This shows that $\|(D_{\lambda}\pm i)^{-1}f\|^{2}_{\mathcal{E}}\leq \|f\|^{2}_{\mathcal{E}}$. The same computation shows that the resolvent $(D_{\lambda}\pm i)^{-1}$ is completely contractive. Moreover, they also show that the resolvents are cb-norm limits of finite rank operators (see Proposition \ref{rangevlambda} and Lemma \ref{lipproj}), and hence $(D_{\lambda}\pm i)^{-1}\in\K(\mathcal{E}^{\Omega}_{A})$. By construction, the connection satisfies $[\nabla,D_{\lambda}]=0$.
\end{proof}

The operator $1\otimes_{\nabla}D_{\mathcal{V},s}$ acts on elementary tensors $e\otimes (\phi_+,\phi_-)^T\in \mathcal{E}^{\Omega}_{A}\otimes_{ \Lip_{\tau,s}(\Omega_{A})}^{alg}C_c(\mathcal{V}_A,\C^2)$ as
\[(1\otimes_{\nabla}D_{\mathcal{V},s})\left(e\otimes\begin{pmatrix}\phi_{+} \\ \phi_{-}\end{pmatrix}\right)(v)=\sum_{k=0}^{\infty}\sum_{n=-k}^{\infty}\sum_{|\mu|=n+k}\chi^{k}_{n,\mu}\otimes\begin{pmatrix}|v|^{s}\pi_{-}(\rho(\chi^{k*}_{n,\mu}e))\phi_{-} \\ |v|^{s}\pi_{+}(\rho(\chi^{k*}_{n,\mu}e)) \phi_{+}\end{pmatrix}(v).\]

\begin{thm} 
\label{selfcpt} 
For any logarithmic Bellissard-Pearson spectral triple with grading operator $\gamma$ and any finite word $\lambda$, the operator
\[D_{\lambda,\tau,s}:=D_{\lambda}\otimes \gamma + 1\otimes_{\nabla}D_{\mathcal{V},s},\]
is selfadjoint and has compact resolvent in $\He(\tau):=\mathpzc{E}^\Omega_A\otimes_{C(\Omega_{A})}\ell^{2}(\mathcal{V}_{A},\C^{2})$.
\end{thm}

\begin{proof}
The unbounded $KK$-cycle $(\mathpzc{E}^\Omega_A,D_{\lambda})$ admits the compatible Lipschitz structure $(\mathcal{E}^\Omega_A,D_{\lambda},\nabla)$ (described above) associated with a Bellissard-Pearson spectral triple $(\pi_\tau,\ell^2(\mathcal{V}_A,\C^2),D_{\mathcal{V},s})$. Therefore, the operator $1\otimes_{\nabla}D_{\mathcal{V},s}$ is essentially selfadjoint by \cite[Theorem 2.30]{BMS}. Since $(D_{\lambda}\pm i)^{-1}\in\mathbb{K}(\mathcal{E}^\Omega_A)$, 
\[\textnormal{im}(D_{\lambda}\otimes \gamma \pm i)^{-1}(1\otimes_{\nabla}D_{\mathcal{V},s}\pm i)^{-1}=\textnormal{im}(1\otimes_{\nabla}D_{\mathcal{V},s}\pm i)^{-1}(D_{\lambda}\otimes \gamma \pm i)^{-1},\]
and $D_{\lambda}\otimes \gamma$ and $1\otimes_{\nabla}D_{\mathcal{V},s}$ anticommute on this subspace by the proof of \cite[Theorem 2.35]{BMS}. From \cite[Theorem 2.33]{BMS}, and the discussion in \cite[Example 2.39]{BMS}, it follows that $D_{\lambda,\tau,s}$ is selfadjoint on the intersection of the domains of $D_{\lambda}\otimes \gamma$ and $1\otimes_{\nabla}D_{\mathcal{V},s}$. The products of the resolvents $(1\otimes_{\nabla}D_{\mathcal{V},s}\pm i)^{-1}$ and $(D_{\lambda}\otimes \gamma \pm i)^{-1}$ are compact by construction; hence by \cite[Lemma 6.3.2]{Mes}, the resolvent of the sum is compact as well.
\end{proof}

\begin{remark} 
In this section and Theorem \ref{selfcpt}, contrary to the constructions in \cite{BMS, KaLe, Mes}, we have not discussed any left module structure for a dense subalgebra of $O_{A}$ on $\mathcal{E}^{\Omega}_{A}$. The existence of a left module structure as in \cite{BMS, KaLe, Mes} would imply that the product operator has bounded commutators with the dense subalgebra of $O_{A}$, and thus represents the Kasparov product of the unbounded modules involved. In view of not having a well behaved left module structure, we cannot conclude bounded commutators with the left action of the dense subalgebra of $O_A$ from Theorem \ref{selfcpt}. Due to the lack of bounded commutators, we are required to use the broader setting of $\epsilon$-unbounded Fredholm modules in order to identify this operator as the Kasparov product.
\end{remark}

\subsection{A family of $\epsilon$-unbounded Fredholm modules}
\label{Kaspprod} 

We now proceed to show that $(\He(\tau),D_{\lambda,\tau,s})$ constitutes an $\epsilon$-unbounded Fredholm module representing the Kasparov product
\[\begin{split} KK_{1}(O_{A},C(\Omega_{A}))\times K^{0}(C(\Omega_{A}))&\rightarrow K^{1}(O_{A})\\
[D_{\lambda}]\times [BP_{s}(\tau)]&\mapsto [D_{\lambda}]\otimes_{C(\Omega_{A})} [BP_{s}(\tau)].\end{split}\]
The classes $[D_{\lambda}]\in KK_1(O_A,C(\Omega_A))$ are described in Subsection \ref{definingoacoa}, and $[BP_{s}(\tau)]\in K^0(C(\Omega_A))=KK_0(C(\Omega_A),\C)$ are the classes associated with the logarithmic Bellissard-Pearson spectral triples, with $s<1$, from Section \ref{sectionbptriples}. The reader is referred to the appendix for the notion of $\epsilon$-unbounded Fredholm modules.

\begin{lem}
\label{nkpos} 
Let $k+n>0$ and $\mu$ be a nonempty word starting in $\mu_1$. Then
\begin{enumerate}\item $S_{i}\chi^{k}_{n,\mu}=A_{i,\mu_{1}}\chi^{k}_{n+1,i\mu}$;
\item $\left(\chi^{k}_{n,\mu}\right)^{*}S_{i}=\delta_{i,\mu_{1}}\left(\chi^{k}_{n-1,\sigma_\mathcal{V}(\mu)}\right)^{*}$.
\end{enumerate}
\end{lem}

\begin{proof} 
We compute
\[\begin{split}S_{i}\chi^{k}_{n,\mu}(x,m,y) & =\sum S_{i}(x,\ell, z)\chi^{k}_{n,\mu}(z,m-\ell,y) \\
&= \chi_{C_{i}}(x)\chi^{k}_{n,\mu}(\sigma(x),m-1,y),\end{split}\]
which is nonzero only if $m=n+1$, $x_{1}=i$, $\sigma(x)\in C_{\mu}$ and $\kappa(\sigma(x),n,y)=k$. This holds if and only if $x\in C_{i\mu}$ and $\kappa (x,n+1,y)=k$, proving 1.). For 2.) we compute again
\[\begin{split}\left(\chi^{k}_{n,\mu}\right)^{*}S_{i}(x,m,y)&=\sum \left(\chi^{k}_{n,\mu}\right)^{*}(x,\ell,z)S_{i}(z,m-\ell,y)\\
&=\sum\chi^{k}_{n,\mu}(z,-\ell,x)S_{i}(z,m-\ell,y)\\
&=A_{i,y_{1}}\chi^{k}_{n,\mu}(iy,1-m,x),\end{split}\]
and this is nonzero only if $m=-(n-1)$, $\mu_{1}=i$, $y\in C_{\sigma_\mathcal{V}(\mu)}$ and $\kappa(iy,n,x)=k$. This holds only if $\kappa(y, n-1, x)=k$, proving 2.)
\end{proof}

\begin{lem}
\label{nknull}
Let $k\geq 0$ and $i=1,\ldots, N$. Then
\begin{enumerate}
\item  $\left(\chi^{k}_{-k,\circ}\right)^{*}S_{i}=(\chi_{C_{i}}\circ\sigma^{k})*\left(\chi^{k+1}_{-k-1,\cemptyset}\right)^{*}$;
\item $S_{i}\chi^{k}_{-k,\cemptyset}=\chi^{k}_{-k+1,i}+\chi_{C_{i}}*\chi^{k-1}_{-k+1,\cemptyset}$;
\item $\chi_{C_{i}}*\chi^{k}_{-k,\cemptyset}=\chi^{k}_{-k,\cemptyset}*\left(\chi_{C_{i}}\circ\sigma^{k}\right)$.
\end{enumerate}
\end{lem}

\begin{proof} 
For 1.) compute
\[\begin{split}&\left(\chi^{k}_{-k,\cemptyset}\right)^{*}S_{i}(x,m,y)  = \sum \left(\chi^{k}_{-k,\cemptyset}\right)^{*}(x,\ell, z)S_{i}(z,m-\ell,y) \\
&=\sum \chi^{k}_{-k,\cemptyset}(z,-\ell,x)S_{i}(z,m-\ell,y)\\
&=A_{iy_{1}}\chi^{k}_{-k,\cemptyset}(iy, 1-m, x)\\
&=\left\{\begin{matrix} 1 & \textnormal{when } m=k+1, \;A_{iy_{1}}=1,\; \kappa(iy,-k,x)=k\\ 0 & \textnormal{otherwise }\end{matrix}\right. \\
&=\left\{\begin{matrix} 1 & \textnormal{when } m=k+1, \;A_{iy_{1}}=1,\; \kappa(y,-(k+1),x)=k+1,\;  \sigma^{k}(x)\in C_{i} \\ 0 & \textnormal{otherwise }\end{matrix}\right.\\
&=\chi_{C_{i}}(\sigma^{k}(x))\chi^{k+1}_{-k-1,\cemptyset}(y,-m,x) =\chi_{C_{i}}(\sigma^{k}(x))\left(\chi^{k+1}_{-k-1,\cemptyset}\right)^{*}(x,m,y).
\end{split}\]
For 2.)
\[\begin{split}S_{i}\chi^{k}_{-k,\cemptyset}&(x,m,y)=\sum S_{i}(x,\ell,z)\chi^{k}_{-k,\circ}(z,m-\ell,y)\\
&=\delta_{i,x_{1}}\chi^{k}_{-k,\cemptyset}(\sigma(x),m-1,y)\\
&=\left\{\begin{matrix}1 & \textnormal{when } x\in C_{i}, \;m=-(k-1),\; \kappa(\sigma(x),-k,y)=k \\ 0 & \textnormal{otherwise}\end{matrix}\right. \\
&=\left\{\begin{matrix}1 & \textnormal{when } x\in C_{i}, \;m=-(k-1), \;\kappa(x,-(k-1),y)\in\{k,k-1\} \\ 0 & \textnormal{otherwise}\end{matrix}\right.\\
&=( \chi^{k}_{-(k-1),i}+\chi_{C_{i}}*\chi^{k-1}_{-(k-1),\cemptyset } )(x,m,y). \end{split}\]

Also 3.) is verified by direct computation.
\[\begin{split}\chi_{C_{i}}*\chi^{k}_{-k,\cemptyset}(x,m,y) & =\chi_{C_{i}}(x)\chi^{k}_{-k,\cemptyset}(x,m,y) \\
&=\left\{\begin{matrix} 1 & \textnormal{when } x\in C_{i}, \;m=-k, \;\kappa(x,-k,y)=k \\ 0 & \textnormal{otherwise}\end{matrix}\right.\\
&=\left\{\begin{matrix} 1 & \textnormal{when } x\in C_{i}, \;m=-k, \;x=\sigma^{k}(y) \\ 0 & \textnormal{otherwise}\end{matrix}\right.\\
&=\left\{\begin{matrix} 1 & \textnormal{when } \sigma^{k}(y)\in C_{i}, \;m=-k, \;  \kappa(x,-k,y)=k\\ 0 & \textnormal{otherwise}\end{matrix}\right.\\
&=\chi^{k}_{-k,\cemptyset}*(\chi_{C_{i}}\circ\sigma^{k})(x,m,y).
\end{split}\]
\end{proof}

\begin{prop} 
\label{Scomm} 
Let $(\pi_{\tau}, \ell^{2}(\mathcal{V}_{A},\C^{2}), D_{\mathcal{V},s})$ be a logarithmic Bellissard-Pearson spectral triple. The operators $S_{i}$ preserve the algebraic tensor product
\[\left(\bigoplus_{n,k}^{alg}\mathcal{E}^{k}_{n}\right)\otimes_{C(\Omega_{A})}^{alg} C_{c}(\mathcal{V}_{A},\C^{2}),\]
 which is a core for $D_{\lambda}\otimes\gamma+1\otimes_{\nabla}D_{\mathcal{V},s}$ and $[1\otimes_{\nabla}D_{\mathcal{V},s},S_{i}]$ is given on an elementary tensor $e\otimes (\phi_+ \; \phi_-)^T$ by the sum
\begin{align}
\label{redcomm}
\left[1\otimes_{\nabla}D_{\mathcal{V},s},S_{i}\right]&e\otimes\begin{pmatrix}\phi_{+} \\  \phi_{-}\end{pmatrix}\\
\nonumber&= -\sum_{k=0}^{\infty}\chi^{k}_{-k,\cemptyset}\otimes 
\left[D_{\mathcal{V},s},\pi_{\tau}\left(\chi_{C_i}\circ\sigma^{k}\right)\right]
\pi_{\tau}\left(\rho(\chi^{(k+1)*}_{-k-1,\cemptyset}e)\right)
\begin{pmatrix}\phi_{+} \\  \phi_{-}\end{pmatrix}.
\end{align}
The operator $[D_{\mathcal{V},s},\pi_{\tau}(\chi_{C_i}\circ\sigma^{k})]$ on $\ell^2(\mathcal{V}_A,\C^2)$ is given by multiplication by a compactly supported matrix valued function on $\mathcal{V}_A$ satisfying the estimate
\begin{equation}
\label{ks}
\left\|[D_{\mathcal{V},s},\pi_{\tau}(\chi_{C_i}\circ\sigma^{k})]\right\|_{\Bo(\ell^2(\mathcal{V}_A,\C^2))}\leq k^{s}.
\end{equation}
\end{prop}

\begin{proof}
Since $S_{i}(\mathcal{E}^{k}_{n})\subseteq  \mathcal{E}^{k-1}_{n+1}\oplus \mathcal{E}^{k}_{n+1}$, the operator $S_{i}$ preserves the algebraic direct sum of the $\mathcal{E}^{n}_{k}$ and hence a common core for $D_{\lambda}\otimes 1$ and $1\otimes_{\nabla}D_{\mathcal{V},s}$. The  commutator $[S_{i},1\otimes_{\nabla}D_{\mathcal{V},s}]$ is computed as
\begin{align}
\label{commutinpropsixtwo}
&[S_{i}, 1\otimes_{\nabla}D_{\mathcal{V},s}]\left(e\otimes
\begin{pmatrix}
\phi_{+} \\ 
\vspace{-3mm}\\ 
\phi_{-}\end{pmatrix}
\right)(v) \\
& = \sum_{n,k,\mu} S_{i}\chi^{k}_{n,\mu}\otimes
\begin{pmatrix}
|v|^{s}\pi_{-}\left(\rho(\chi^{k*}_{n,\mu}e)\right)\phi_{-} \\
\vspace{-3mm}\\
\nonumber
|v|^{s}\pi_{+}\left(\rho(\chi^{k*}_{n,\mu}e)\right) \phi_{+}
\end{pmatrix}(v)\\
\nonumber&\qquad\qquad-\chi^{k}_{n,\mu}\otimes 
\begin{pmatrix}
|v|^{s}\pi_{-}\left(\rho(\chi^{k*}_{n,\mu}S_{i}e)\right)\phi_{-} \\
\vspace{-3mm}\\
|v|^{s}\pi_{+}\left(\rho(\chi^{k*}_{n,\mu}S_{i}e)\right) \phi_{+}
\end{pmatrix}(v). 
\end{align}
This expression can by Lemma \ref{nkpos} (for $n+k>0$) and by Lemma \ref{nknull} (for $n+k=0$) be written as:
\[\begin{split}
=\sum_{n+k>0}A_{i,\mu_{1}}&\chi^{k}_{n+1,i\mu}\otimes 
\begin{pmatrix}
|v|^{s}\pi_{-}\left(\rho(\chi^{k*}_{n,\mu}e)\right)\phi_{-} \\
\vspace{-3mm}\\
|v|^{s}\pi_{+}\left(\rho(\chi^{k*}_{n,\mu}e)\right) \phi_{+}
\end{pmatrix}(v)\\
&-\chi^{k}_{n,\mu}\otimes 
\begin{pmatrix}
|v|^{s}\pi_{-}\left(\rho(\delta_{i,\mu_{1}}\chi^{k*}_{n-1,\sigma(\mu)}e)\right)\phi_{-} \\
\vspace{-3mm}\\
|v|^{s}\pi_{+}\left(\rho(\delta_{i,\mu_{1}}\chi^{k*}_{n-1,\sigma(\mu)})e)\right) \phi_{+}
\end{pmatrix}(v)
\\
\\
&+\sum_{k=0}^{\infty}\left(\chi^{k}_{-k+1,i}+\chi_{C_i}*\chi^{k-1}_{-k+1,\cemptyset})\right)\otimes
\begin{pmatrix}
|v|^{s}\pi_{-}\left(\rho(\chi^{k*}_{-k,\cemptyset}e)\right)\phi_{-} \\ 
\vspace{-3mm}\\
|v|^{s}\pi_{+}\left(\rho(\chi^{k*}_{-k,\cemptyset}e)\right) \phi_{+}
\end{pmatrix}(v)\\
\\
&\quad-\sum_{k=0}^{\infty}\chi^{k}_{-k,\cemptyset}\otimes 
\begin{pmatrix}
|v|^{s}\pi_{-}\left(\rho((\chi_{C_{i}}\circ \sigma^{k})\chi^{(k+1)*}_{-k-1,\cemptyset}e)\right)\phi_{-} \\ 
\vspace{-3mm}\\
|v|^{s}\pi_{+}\left(\rho((\chi_{C_{i}}\circ \sigma^{k})\chi^{(k+1)*}_{-k-1,\cemptyset}e)\right) \phi_{+}
\end{pmatrix}(v)
\end{split}\]
We regroup these expressions as follows.
\begin{equation}\label{expandedcomm1}\sum_{n+k>0}A_{i,\mu_{1}}\chi^{k}_{n+1,i\mu}\otimes 
\begin{pmatrix}
|v|^{s}\pi_{-}\left(\rho(\chi^{k*}_{n,\mu}e)\right)\phi_{-} \\
\vspace{-3mm}\\
|v|^{s}\pi_{+}\left(\rho(\chi^{k*}_{n,\mu}e)\right) \phi_{+}
\end{pmatrix}(v)\qquad\qquad\end{equation}
\begin{equation}\label{expandedcomm2} \qquad\qquad -
\chi^{k}_{n,\mu}\otimes 
\begin{pmatrix}
|v|^{s}\pi_{-}\left(\rho(\delta_{i,\mu_{1}}\chi^{k*}_{n-1,\sigma(\mu)}e)\right)\phi_{-} \\
\vspace{-3mm}\\
|v|^{s}\pi_{+}\left(\rho(\delta_{i,\mu_{1}}\chi^{k*}_{n-1,\sigma(\mu)}e)\right) \phi_{+}
\end{pmatrix}(v)
\end{equation}
\begin{equation}\label{expandedcomm3} \qquad\qquad\qquad\qquad\qquad\qquad
+\sum_{k=0}^{\infty}\chi^{k}_{-k+1,i}\otimes
\begin{pmatrix}
|v|^{s}\pi_{-}\left(\rho(\chi^{k*}_{-k,\cemptyset}e)\right)\phi_{-} \\ 
\vspace{-3mm}\\
|v|^{s}\pi_{+}\left(\rho(\chi^{k*}_{-k,\cemptyset}e)\right) \phi_{+}
\end{pmatrix}(v)\end{equation}
\begin{align}
\label{rest}
+\sum_{k=0}^{\infty}\chi_{C_{i}}*\chi^{k}_{-k,\cemptyset}\otimes &\begin{pmatrix}
|v|^{s}\pi_{-}\left(\rho(\chi^{(k+1)*}_{-k-1,\cemptyset}e)\right)\phi_{-} \\ 
\vspace{-3mm}\\
|v|^{s}\pi_{+}\left(\rho(\chi^{(k+1)*}_{-k-1,\cemptyset}e)\right) \phi_{+}
\end{pmatrix}(v)\\
\nonumber
&-\chi^{k}_{-k,\cemptyset}\otimes
\begin{pmatrix}
|v|^{s}\pi_{-}\left(\rho((\chi_{C_{i}}\circ \sigma^{k})\chi^{(k+1)*}_{-k-1,\cemptyset}e)\right)\phi_{-} \\ 
\vspace{-3mm}\\
|v|^{s}\pi_{+}\left(\rho((\chi_{C_{i}}\circ \sigma^{k})\chi^{(k+1)*}_{-k-1,\cemptyset}e)\right) \phi_{+}
\end{pmatrix}(v).\end{align}
We claim that \eqref{expandedcomm1}, \eqref{expandedcomm2} and \eqref{expandedcomm3} add up to $0$. To see this, consider a nonempty word  $\mu$ with $A_{i\mu_{1}}=1$. Each nonzero term in \eqref{expandedcomm1} is cancelled by a nonzero term in \eqref{expandedcomm2}. All of \eqref{expandedcomm1} is cancelled in this way. What remains in \eqref{expandedcomm2} are the terms with $n+k=|\mu |=1$ and $\delta_{i,\mu_{1}}=1$. The remaining terms in \eqref{expandedcomm2} correspond to $\mu=i$ and $n=-k+1$. These are exactly the terms occuring in \eqref{expandedcomm3}, with the opposite sign. As such, the remainder of \eqref{expandedcomm2} is cancelled by \eqref{expandedcomm3}, as claimed, and the entire commutator in \eqref{commutinpropsixtwo} equals \eqref{rest}.

Subsequently, we handle \eqref{rest} by exchanging $\pi_{+}$ and $\pi_{-}$ at the expense of a commutator to obtain
\begin{align}
\nonumber
[S_{i}, 1\otimes_{\nabla}D_{\mathcal{V},s}]&\left(e\otimes
\begin{pmatrix}
\phi_{+} \\ 
\vspace{-3mm}\\ 
\phi_{-}\end{pmatrix}
\right)(v) \\
\label{rel}&=\sum_{k=0}^{\infty}\chi_{C_i}*\chi^{k}_{-k,\cemptyset}\otimes
\begin{pmatrix}
|v|^{s}\pi_{+}\left(\rho(\chi^{(k+1)*}_{-k-1,\cemptyset}e)\right)\phi_{-} \\ 
\vspace{-3mm}\\
|v|^{s}\pi_{-}\left(\rho(\chi^{(k+1)*}_{-k-1,\cemptyset}e)\right) \phi_{+}
\end{pmatrix}(v) \\
\nonumber
\\
\nonumber
&\qquad\qquad-\sum_{k=0}^{\infty}\chi^{k}_{-k,\cemptyset}\otimes 
\begin{pmatrix}
|v|^{s}\pi_{+}\left(\rho((\chi_{C_{i}}\circ \sigma^{k})\chi^{(k+1)*}_{-k-1,\cemptyset}e)\right)\phi_{-} \\
\vspace{-3mm}\\ 
|v|^{s}\pi_{-}\left(\rho((\chi_{C_{i}}\circ \sigma^{k})\chi^{(k+1)*}_{-k-1,\cemptyset}e)\right) \phi_{+}
\end{pmatrix}(v)\end{align}
\begin{align}
\label{Lip}
+\sum_{k=0}^{\infty}\chi_{C_i}*\chi^{k}_{-k,\cemptyset}\otimes &\left[D_{\mathcal{V},s},\pi_{\tau}\left(\rho(\chi^{(k+1)*}_{-k-1,\cemptyset}e)\right)\right]
\begin{pmatrix}\phi_{+} \\
\vspace{-3mm}\\
\phi_{-}
\end{pmatrix}(v)\\
\nonumber\\
\nonumber
&\quad-\sum_{k=0}^{\infty}\chi^{k}_{-k,\cemptyset}\otimes\left[D_{\mathcal{V},s},\pi_{\tau}\left(\rho((\chi_{C_{i}}\circ \sigma^{k})\chi^{(k+1)*}_{-k-1,\cemptyset}e)\right)\right] 
\begin{pmatrix}
\phi_{+} \\
\vspace{-3mm}\\ 
\phi_{-}\end{pmatrix}(v),
\end{align}
and to the term \eqref{rel} we apply Lemma \ref{nknull} 3.) to obtain
\begin{align*}
\sum_{k=0}^{\infty}\chi^{k}_{-k,\cemptyset}(\chi_{C_{i}}\circ\sigma^{k})\otimes&
\begin{pmatrix}
|v|^{s}\pi_{+}\left(\rho(\chi^{(k+1)*}_{-k-1,\cemptyset}e)\right)\phi_{-} \\ 
\vspace{-3mm}\\
|v|^{s}\pi_{-}\left(\rho(\chi^{(k+1)*}_{-k-1,\cemptyset}e)\right) \phi_{+}
\end{pmatrix}(v)\\
&-\sum_{k=0}^{\infty}\chi^{k}_{-k,\cemptyset}\otimes 
\begin{pmatrix}
|v|^{s}\pi_{+}\left(\rho((\chi_{C_{i}}\circ \sigma^{k})\chi^{(k+1)*}_{-k-1,\cemptyset}e)\right)\phi_{-} \\
\vspace{-3mm}\\
|v|^{s}\pi_{-}\left(\rho((\chi_{C_{i}}\circ \sigma^{k})\chi^{(k+1)*}_{-k-1,\cemptyset}e)\right) \phi_{+}
\end{pmatrix}(v)
\end{align*}
\begin{align*}
=\sum_{k=0}^{\infty}\chi^{k}_{-k,\cemptyset}\otimes&
\begin{pmatrix}
|v|^{s}\pi_{+}\left(\rho((\chi_{C_{i}}\circ\sigma^{k})\chi^{(k+1)*}_{-k-1,\cemptyset}e)\right)\phi_{-} \\ 
\vspace{-3mm}\\
|v|^{s}\pi_{-}\left(\rho((\chi_{C_{i}}\circ\sigma^{k})\chi^{(k+1)*}_{-k-1,\cemptyset}e)\right) \phi_{+}
\end{pmatrix}(v)\\
\\
&\qquad -\sum_{k=0}^{\infty}\chi^{k}_{-k,\cemptyset}\otimes 
\begin{pmatrix}|v|^{s}\pi_{+}\left(\rho((\chi_{C_{i}}\circ \sigma^{k})\chi^{(k+1)*}_{-k-1,\cemptyset}e)\right)\phi_{-} \\ 
\vspace{-3mm}\\
|v|^{s}\pi_{-}\left(\rho((\chi_{C_{i}}\circ \sigma^{k})\chi^{(k+1)*}_{-k-1,\cemptyset}e)\right) \phi_{+}
\end{pmatrix}(v)=0.
\end{align*}
The remaining term \eqref{Lip} further simplifies to \eqref{redcomm} using Lemma \ref{nknull} 3.) once more. 

The commutator $[D_{\mathcal{V},s},\pi_{\tau}(\chi_{C_i}\circ\sigma^{k})]$ appearing in \eqref{redcomm} vanishes whenever $|v|>k$ because in that case $\tau_{+}(v)_{k+1}=v_{k+1}=\tau_{-}(v)_{k+1}$ hence $\chi_{C_i}(\tau_{+}(v))=\chi_{C_i}(\tau_{-}(v))$. Here the subscript $k+1$ indicates the $(k+1)$-st letter. The estimate \eqref{ks} follows readily from the same observation. 
\end{proof}

Recall the notation $\He(\tau)$ and $D_{\lambda,\tau,s}$ from Theorem \ref{selfcpt}.

\begin{thm} 
\label{comegaandbp}
For $s\in (0,1)$ the pair $(\He(\tau), D_{\lambda,\tau,s})$ is a well defined $(1-s)$-unbounded Fredholm module on $O_A$ that represents the Kasparov product $[\mathpzc{E},D_{\lambda}]\otimes_{C(\Omega_{A})}[BP_{s}(\tau)]$. 
\end{thm}

\begin{proof} 
The operator $D_{\lambda,\tau,s}=D_{\lambda}\otimes \gamma+ 1\otimes_{\nabla}D_{\mathcal{V},s}$ is selfadjoint with compact resolvent by Theorem \ref{selfcpt}. We will show that the operators 
\[[1\otimes_{\nabla}D_{\mathcal{V},s},S_{i}](1+D_{\lambda,\tau,s}^{2})^{-\frac{s}{2}},\quad (1+D_{\lambda,\tau,s}^{2})^{-\frac{s}{2}}[1\otimes_{\nabla}D_{\mathcal{V},s},S_{i}]\] 
are bounded. By Proposition \ref{Scomm}, $S_{i}$ preserves a core for $D_{\lambda}\otimes \gamma+1\otimes_{\nabla}D_{\mathcal{V},s}$. The operator $1+D_{\lambda,\tau,s}^2=1+D_{\lambda}^2\otimes 1+(1\otimes_{\nabla}D_{\mathcal{V},s})^{2}$ preserves subspaces of the form $\mathcal{E}^{k}_{n}\otimes_{\Lip_{\tau}(\Omega_{A})}C_{c}(\mathcal{V}_{A},\C^{2})$. From the form of \eqref{redcomm}, it follows that 
\begin{equation}
\label{comrestitok}
(1+D_{\lambda,\tau,s}^{2})^{-\frac{s}{2}}[1\otimes_{\nabla}D_{\mathcal{V},s},S_{i}]:\mathcal{E}^{k+1}_{-k-1}\otimes_{\Lip_{\tau}(\Omega_{A})}C_{c}(\mathcal{V}_{A},\C^{2})\rightarrow \mathcal{E}^{k}_{-k}\otimes_{\Lip_{\tau}(\Omega_{A})}C_{c}(\mathcal{V}_{A},\C^{2}).
\end{equation}
We denote the restricted operator of Equation \eqref{comrestitok} by $T_{i,k}$. By Equation \eqref{redcomm} and the orthogonality of the decomposition of Proposition \ref{rangevlambda}, it holds that
$$\left\|(1+D_{\lambda,\tau,s}^{2})^{-\frac{s}{2}}[1\otimes_{\nabla}D_{\mathcal{V},s},S_{i}]\right\|_{\Bo(\He(\tau))}\leq\sup_k\|T_{i,k}\|_{\Bo(\He(\tau))}$$
As such, it suffices to show that for any $k$ and any finite sum
\small
\begin{equation}
\label{xexpaenda}
x=\sum_{j} e_j\otimes \begin{pmatrix}\phi_+^j\\\phi^j_-\end{pmatrix}=\sum_{j} \chi^{k+1}_{-k-1,\circ}\otimes \pi_\tau\left(\rho\left((\chi^{k+1}_{-k-1,\circ})^*e_j\right)\right) \begin{pmatrix}\phi_+^j\\\phi^j_-\end{pmatrix}\in \mathcal{E}^{k+1}_{-k-1}\otimes_{\Lip_{\tau}(\Omega_{A})}C_{c}(\mathcal{V}_{A},\C^{2})
\end{equation}
\normalsize
it holds that
\begin{align}
\nonumber
\|&T_{i,k}x\|_{\He(\tau)}=\left\|(1+D_{\lambda,\tau,s}^{2})^{-\frac{s}{2}}[1\otimes_{\nabla}D_{\mathcal{V},s},S_{i}]x\right\|_{\He(\tau)}\\
\nonumber
&=\left\|\sum_{j}(1+D_{\lambda,\tau,s}^{2})^{-\frac{s}{2}}\chi^{k}_{-k,\cemptyset}\otimes [D_{\mathcal{V},s},\pi_{\tau}(\chi_{C_i}\circ\sigma^{k})]\pi_{\tau}(\rho(\chi^{(k+1)*}_{-k-1,\cemptyset}e_{j})\begin{pmatrix}\phi_{+}^{j} \\  \phi_{-}^{j}\end{pmatrix}\right\|_{\He(\tau)}\\
\label{unifestink}
&\qquad\quad\leq\left \|\sum_{j}\chi_{-k-1,\circ}^{k+1}\otimes \pi_{\tau}\left(\rho(\chi^{(k+1)*}_{-k-1,\circ}e_{j})\right)\begin{pmatrix}\phi_{+}^{j} \\  \phi_{-}^{j}\end{pmatrix}\right\|_{\He(\tau)}=\|x\|_{\He(\tau)}.
\end{align} 
It is of computational importance to note that when writing 
$$x= \chi^{k+1}_{-k-1,\circ}\otimes v, \quad\mbox{where}\quad v= \sum_j\pi_\tau\left(\rho\left((\chi^{k+1}_{-k-1,\circ})^*e_j\right)\right) \begin{pmatrix}\phi_+^j\\\phi^j_-\end{pmatrix},$$ 
as in Equation \eqref{xexpaenda}, we have that 
\begin{equation}
\label{normsinheta}
\|x\|_{\He(\tau)}^2=\left\langle v,\pi_\tau\left(\rho( ( \chi^{k+1}_{-k-1,\circ})^* \chi^{k+1}_{-k-1,\circ})\right)v\right\rangle_{\ell^2(\mathcal{V}_A,\C^2)}=\|v\|_{\ell^2(\mathcal{V}_A,\C^2)}^2,
\end{equation}
because $\rho(( \chi^{k+1}_{-k-1,\circ})^* \chi^{k+1}_{-k-1,\circ})=1$ by Lemma \ref{kisom}.
It follows from the construction of $D_\lambda$ that $D_{\lambda}\otimes 1$ acts as multiplication by $-2k$ on $\mathcal{E}^{k}_{-k}\otimes_{\Lip_{\tau}(\Omega_{A})}C_{c}(\mathcal{V}_{A},\C^{2})$. With this fact at hand, the verification of this estimate is a straightforward computation using  the inequality \eqref{ks}:
\begin{align}
\nonumber
\|&T_{i,k}x\|_{\He(\tau)}=\big\|(1+D_{\lambda,\tau,s}^{2})^{-\frac{s}{2}}[1\otimes_{\nabla}D_{\mathcal{V},s},S_{i}]x\big\|_{\He(\tau)}\\
\nonumber
&= \left\|\sum_{j}  \right.(1+4k^{2}+ \left.(1\otimes_{\nabla}D_{\mathcal{V},s})^{2})^{-\frac{s}{2}}\chi^{k}_{-k,\cemptyset}\otimes [D_{\mathcal{V},s},\pi_{\tau}(\chi_{C_i}\circ\sigma^{k})]\pi_{\tau}\left(\rho(\chi^{(k+1)*}_{-k-1,\cemptyset}e_{j})\right)\begin{pmatrix}\phi_{+}^{j} \\  \phi_{-}^{j}\end{pmatrix}\right\|_{\He(\tau)} \\
\nonumber
&= \left\|  \right.(1+4k^{2}+ \left.(1\otimes_{\nabla}D_{\mathcal{V},s})^{2})^{-\frac{s}{2}}\sum_{j}\chi^{k}_{-k,\cemptyset}\otimes [D_{\mathcal{V},s},\pi_{\tau}(\chi_{C_i}\circ\sigma^{k})]\pi_{\tau}\left(\rho(\chi^{(k+1)*}_{-k-1,\cemptyset}e_{j})\right)\begin{pmatrix}\phi_{+}^{j} \\  \phi_{-}^{j}\end{pmatrix}\right\|_{\He(\tau)} \\
\nonumber
&\qquad\qquad \leq (1+k^{2})^{-\frac{s}{2}}
\left\|\sum_{j}\chi^{k}_{-k,\circ}\otimes[D_{\mathcal{V},s},\pi_{\tau}(\chi_{C_i}\circ\sigma^{k})]\pi_{\tau}\left(\rho(\chi^{(k+1)*}_{-k-1}e_{j}\right)\begin{pmatrix}\phi_{+}^{j} \\  \phi_{-}^{j}\end{pmatrix}\right\|_{\He(\tau)}\\
\label{firstteptik}
&\qquad\qquad= (1+k^{2})^{-\frac{s}{2}}
\left\|\chi^{k}_{-k,\circ}\otimes[D_{\mathcal{V},s},\pi_{\tau}(\chi_{C_i}\circ\sigma^{k})]\sum_{j}\pi_{\tau}\left(\rho(\chi^{(k+1)*}_{-k-1}e_{j}\right)\begin{pmatrix}\phi_{+}^{j} \\  \phi_{-}^{j}\end{pmatrix}\right\|_{\He(\tau)}.
\end{align}
Using the identity \eqref{normsinheta}, we arrive at
\begin{align}
\nonumber
 \|T_{i,k}x\|_{\He(\tau)}&\leq (1+k^{2})^{-\frac{s}{2}}
\left\|[D_{\mathcal{V},s},\pi_{\tau}(\chi_{C_i}\circ\sigma^{k})]\sum_{j}\left(\rho(\chi^{(k+1)*}_{-k-1}e_{j}\right)\begin{pmatrix}\phi_{+}^{j} \\  \phi_{-}^{j}\end{pmatrix}\right\|_{\ell^2(\mathcal{V}_A,\C^2)}\\
\nonumber
&\leq (1+k^{2})^{-\frac{s}{2}}k^{s}
\left\|\sum_{j}\pi_{\tau}\left(\rho(\chi^{(k+1)*}_{-k-1}e_{j}\right)\begin{pmatrix}\phi_{+}^{j} \\  \phi_{-}^{j}\end{pmatrix}\right\|_{\ell^2(\mathcal{V}_A,\C^2)}\\
\nonumber
&\qquad\qquad\leq \left\|\sum_{j}\pi_{\tau}\left(\rho(\chi^{(k+1)*}_{-k-1}e_{j}\right)\begin{pmatrix}\phi_{+}^{j} \\  \phi_{-}^{j}\end{pmatrix}\right\|_{\ell^2(\mathcal{V}_A,\C^2)}\\
\label{secondteptik}
&\qquad\qquad= \left\|\sum_{j}\chi^{k+1}_{-k-1,\circ}\otimes\pi_{\tau}\left(\rho(\chi^{(k+1)*}_{-k-1}e_{j}\right)\begin{pmatrix}\phi_{+}^{j} \\  \phi_{-}^{j}\end{pmatrix}\right\|_{\He(\tau)}=\|x\|_{\He(\tau)}.
\end{align}
Hence \eqref{unifestink} holds proving that $\left(1+(D_{\lambda,\tau,s})^{2}\right)^{-\frac{s}{2}}[1\otimes_{\nabla}D_{\mathcal{V},s},S_{i}]$ is bounded. Boundedness of the reverse product follows from a similar computation, reversing the order in which the estimates \eqref{firstteptik} and \eqref{secondteptik} respectively, are applied. Now Lemma \ref{corecomm} implies that the commutators $[D,S_{i}]$ are $\epsilon$-bounded. Thus, by Proposition \ref{*-alg}, $D$ has $\epsilon$-bounded commutators with the $*$-subalgebra of $O_{A}$ generated by the operators $S_{i}$, which is dense in $O_{A}$. Thus we have an $\epsilon$-unbounded Fredholm module with $\epsilon=1-s$. To see that this $\epsilon$-unbounded Fredholm module represents the Kasparov product one uses Theorem \ref{epsilonprod} which applies because the connection condition 1.), the domain condition 2.) and the semiboundedness condition 3.) are satisfied by construction. 
\end{proof}

\begin{remark}
We remark once more that $s=1$ is excluded from Theorem \ref{comegaandbp} because the theory of $\epsilon$-unbounded Fredholm modules breaks down at $\epsilon=0$. As the proof of Theorem \ref{comegaandbp} shows, the operators $[1\otimes_{\nabla}D_{\mathcal{V},1},S_{i}](1+D_{\lambda,\tau,1}^{2})^{-\frac{1}{2}}$ and $(1+D_{\lambda,\tau,1}^{2})^{-\frac{1}{2}}[1\otimes_{\nabla}D_{\mathcal{V},1},S_{i}]$ are bounded, but it is unclear if the bounded transform is well defined and represents the Kasparov product $[\mathpzc{E},D_{\lambda}]\otimes_{C(\Omega_{A})}[BP_{1}(\tau)]$. 
\end{remark}

\subsection{The rational $K$-homology class of the product}
Lastly, we identify the rational $K$-homology class of the Kasparov products constructed in the previous subsection. The identification is done via an index theoretic argument, therefore it needs only to hold rationally.

\begin{thm}
\label{rationkhomcomp}
In $K^1(O_A)\otimes \Q$ we have
\begin{align*}
[\mathpzc{E},D_{\lambda}]&\otimes_{C(\Omega_{A})}[BP_{s}(\tau)]\otimes \Q\\
&=
\begin{cases}
[\beta_{j_+}]\otimes \Q-[\beta_{j_-}]\otimes \Q,\quad &\mbox{if}\; \lambda=\circ_A,\\
(A(\lambda_\ell,j_+)-A(\lambda_\ell,j_-))[\beta_{\lambda_1}]\otimes \Q,\quad &\mbox{if}\; \lambda=\lambda_1\cdots \lambda_\ell\in \mathcal{V}_A\setminus\{\circ_A\},
\end{cases}\end{align*}
where $j_\pm$ is the first letter of $\tau_\pm(\circ_A)$.
\end{thm}

\begin{proof}
The computation of the class $[\mathpzc{E}^\Omega_A,D_{\lambda}]\otimes_{C(\Omega_{A})}[BP_{s}(\tau)]$ in $K^1(O_A)\otimes \Q$ relies on the fact that $O_A$ is in the bootstrap class with finitely generated $K$-theory and $K$-homology, so $K^1(O_A)\otimes \Q\cong \Hom_{\!\Z}(K_1(O_A),\Q)$ and rational classes are determined by their index pairing. Furthermore, using Remark \ref{decomposingbptriples}, we write 
\[\He(\tau)=\bigoplus_{\mu\in \mathcal{V}_A} \mathpzc{E}^\Omega_A\otimes_{\omega_{\tau_+(\mu)}\oplus \omega_{\tau_-(\mu)}}\C^2\quad\mbox{and}\quad D_{\lambda,\tau,s}|D_{\lambda,\tau,s}|^{-1}=\bigoplus_{\mu\in \mathcal{V}_A}F_\mu.\]
Since in each fixed summand $\mathpzc{E}^\Omega_A\otimes_{\omega_{\tau_+(\mu)}\oplus \omega_{\tau_-(\mu)}}\C^2\subseteq \He(\tau)$, $D_{\lambda,\tau,s}$ is a bounded perturbation of the operator $D_\lambda\otimes_{\omega_{\tau_+(\mu)}\oplus \omega_{\tau_-(\mu)}} (1\oplus-1)$, it follows from \cite[Appendix A, Theorem $8$]{cp1}, and an argument similar to the proof of Proposition \ref{posspeccom}, that for any $\mu$ 
\begin{align}
\label{cpompprert}
F_\mu-&(2p_\lambda-1)\otimes_{\omega_{\tau_+(\mu)}\oplus \omega_{\tau_-(\mu)}} (1\oplus -1)\in \Ko\left( \mathpzc{E}^\Omega_A\otimes_{\omega_{\tau_+(\mu)}\oplus \omega_{\tau_-(\mu)}}\C^2\right)\\
\nonumber
&\mbox{and}\quad \|F_\mu-(2p_\lambda-1) \otimes_{\omega_{\tau_+(\mu)}\oplus \omega_{\tau_-(\mu)}} (1\oplus -1)\|_{\Ko( \mathpzc{E}^\Omega_A\otimes_{\omega_{\tau_+(\mu)}\oplus \omega_{\tau_-(\mu)}}\C^2)}\leq 2+|\mu|^s.
\end{align}
For any $x\in K_1(O_A)$, represented by a unitary $u$, there is a finite set $F_u\subseteq \mathcal{V}_A$ such that 
\begin{align*}
x\otimes_{O_A}[\mathpzc{E}^\Omega_A,D_{\lambda}]\otimes_{C(\Omega_{A})}[BP_{s}(\tau)]&=\sum_{\mu\in F_u}x\otimes_{O_A}[\mathpzc{E}^\Omega_A,D_{\lambda}]\otimes_{C(\Omega_{A})}[\mathfrak{S}_\mu]\\
&=\sum_{\mu\in F_u} x\otimes (\omega_{\tau_+(\mu)}-\omega_{\tau_-(\mu)})_*[\mathpzc{E}^\Omega_A,D_\lambda]
\end{align*}
If such a finite set $F_u$ does not exist, the index pairing $x\otimes_{O_A}[\mathpzc{E}^\Omega_A,D_{\lambda}]\otimes_{C(\Omega_{A})}[BP_{s}(\tau)]$ can not be well defined. The cylinder condition implies that for any nonempty word $\mu$ it holds that the first letter of $\tau_+(\mu)$ is the same as that of $\tau_-(\mu)$. Hence
$$x\otimes_{O_A}[\mathpzc{E}^\Omega_A,D_{\lambda}]\otimes_{C(\Omega_{A})}[BP_{s}(\tau)]=x\otimes (\omega_{\tau_+(\circ_A)}-\omega_{\tau_-(\circ_A)})_*[\mathpzc{E}^\Omega_A,D_\lambda].$$
The theorem now follows from Theorem \ref{computingkhomclasseslikeaboss}.
\end{proof}

\begin{remark}
\label{integralremark}
It would be interesting to compute the integral class $[\mathpzc{E}^\Omega_A,D_{\lambda}]\otimes_{C(\Omega_{A})}[BP_{s}(\tau)]\in K^1(O_A)$ explicitly. It is to the authors unclear if there is a deeper homological obstruction for Theorem \ref{rationkhomcomp} to hold over $\Z$. A direct $K$-homological proof, e.g. using partial isometries, would require a deeper understanding of the Hilbert space $\mathpzc{H}(\tau)$. One might speculate that the analytic difficulties arising in this problem are analogous to the limiting behaviour in the construction of the measure $\mu_{A}=\mathrm{w}^*\mbox{-}\lim_{s\downarrow \delta_{A}} \mu_{s}$ from Subsection \ref{realizingboundarysubsec}. 
\end{remark}

\appendix
\vspace{3mm}

\Large
\section{$\epsilon$-unbounded $KK$-cycles and the Kasparov product}
\normalsize
\renewcommand{\theequation}{\thesection.\arabic{equation}}
\label{theappendix}

We describe a weakening of the definition of an unbounded $KK$-cycle \cite{BJ}. This notion, and in particular Theorem \ref{epsilonFred} below, originated from discussions of the second author with A. Rennie. One of the key observations in the proof of this theorem appears in \cite[Lemma 51]{grensing}. Related notions are anticipated in the literature, (eg. \cite{cp1, Lesch}) but to the authors' knowledge, a concise exposition as in this appendix has not appeared before. The main idea here is to relax the requirement on the commutators $[D,a]$ to be bounded by only asking for $\epsilon$-\emph{boundedness} of these operators.

\begin{deefapp} 
Let $B$ be a $C^*$-algebra and $\mathpzc{E}$ be a $B$-Hilbert $C^*$-module. An operator $a\in \End^{*}_{B}(\mathpzc{E})$ has $\epsilon$-\emph{bounded commutators} with the selfadjoint regular operator $D$ if 
\begin{enumerate} 
\item $a\Dom D\subset \Dom D $;
\item $[D,a](1+D^{2})^{-\frac{1-\epsilon}{2}} \textnormal{and } (1+D^{2})^{-\frac{1-\epsilon}{2}}[D,a] \textnormal{ extend to } \End^{*}_{B}(\mathpzc{E})$.
\end{enumerate}
In short we say that $[D,a]$ is $\epsilon$-bounded. We write $\delta:=\frac{\epsilon}{2}$ throughout this section.
\end{deefapp}

\begin{remark}
Let us give a geometric example of $\epsilon$-bounded commutators to explain the appearance of the parameter $\epsilon>0$. Let $D$ be a self-adjoint elliptic pseudodifferential operator of order $m >0$ acting on a vector bundle $E\to M$ on a closed manifold $M$. The Hilbert space is $\He=L^2(M,E)$. The domain of $D$ is the Sobolev space $W^{m,2}(M,E)$. If $a\in C^\infty(M)$, then $[D,a]$ is a pseudodifferential operator of order $m-1$. Hence $(1+D^2)^{\frac{1-m}{2m}}[D,a]$ and $[D,a](1+D^2)^{\frac{1-m}{2m}}$ are pseudodifferential operators of order $0$, thus bounded on $L^2(M,E)$. We conclude that any $a\in C^\infty(M)$ has $1/m$-bounded commutators with $D$. As such, one can consider the reciprocal $\epsilon^{-1}$ as an ``order" of the operator $D$ appearing in an $\epsilon$-bounded commutator.
\end{remark}

\begin{deefapp} 
Let $A$ and $B$ be $C^{*}$-algebras and $\epsilon>0$. An odd $\epsilon$-$KK$-cycle is a pair $(\mathpzc{E},D)$ where $\mathpzc{E}$ is a $C^{*}-(A,B)-$bimodule and $D$ a selfadjoint regular operator such that
\begin{enumerate}
\item $a(1+D^{2})^{-\frac{1}{2}}\in \K(\mathpzc{E})$;
\item the space\[\Lip^{\epsilon}(\mathpzc{E},D):=\{a\in A: [D,a] \textnormal{ is $\epsilon$-bounded}\}\] is dense in $A$.
\end{enumerate}
If $\mathpzc{E}$ is a $\Z/2\Z$-graded $C^{*}-(A,B)-$bimodule\footnote{In this appendix, ungraded $C^{*}$-algebras $A$ and $B$ will be equipped with the trivial gradings whenever a grading on them is required.}, and $(\mathpzc{E},D)$ is as above with $D$ anticommuting with the grading operator on $\mathpzc{E}$, we say that $(\mathpzc{E},D)$ is an even $\epsilon$-$KK$-cycle. If $B=\C$, we call an odd/even $\epsilon$-$KK$-cycle an odd/even $\epsilon$-unbounded Fredholm module and if the $B$-action is faithful, we call it an $\epsilon$-spectral triple. 
\end{deefapp}

We remind the reader that in the context of graded $C^{*}$-algebras, it suffices to consider even unbounded $KK$-cycles, because the odd group $KK_{1}(A,B)$ can be naturally identified with the even group $KK_{0}(A,B\otimes \C_{1})$, where $\C_{1}$ denotes the first complex Clifford algebra. This is known as \emph{formal Bott periodicity}  (cf. \cite{AKH, Kas1, Kas2}). For this reason we will in this appendix formulate things mostly for even $KK$-cycles.

\begin{remarkapp}
Although the definition of $\epsilon$-boundedness allows for larger classes of unbounded Fredholm modules, obstructions to finite summability remains. The reader can check that the proof of \cite[Theorem 8]{Connestrace} implies the following statement:  if $A$ is a $C^*$-algebra and $(\pi,\He,D)$ is an $\epsilon$-unbounded Fredholm module with $(1+D^2)^{-1}\in \mathcal{L}^p(\He)$, for some $p\in [1,\infty)$, then there is a tracial state on $A$.
\end{remarkapp}

An $\epsilon$-cycle is an $\epsilon'$-cycle for any $\epsilon'\leq\epsilon$. All of the proofs below rely on the integral representation formula and the estimates in the following lemma.

\begin{lemapp}
\label{intes} 
Let $D$ be a regular self-adjoint operator on a $B$-Hilbert $C^*$-module $\mathpzc{E}$. For any $0<r<1$
\begin{equation}
\label{magicint} 
(1+D^{2})^{-r}=\frac{\sin(r\pi)}{\pi}\int_{0}^{\infty}\lambda^{-r}(1+D^{2}+\lambda)^{-1}\rd\lambda,
\end{equation}
is a norm convergent integral. Moreover we have the estimates
\begin{align*} 
\|(1+D^{2}+\lambda)^{-s}\|_{\End^*_B(\mathpzc{E})}&\,\leq (1+\lambda)^{-s};\\
\|D(1+D^{2}+\lambda)^{-\frac{1}{2}}\|_{\End^*_B(\mathpzc{E})}\leq 1\quad\mbox{and}&\quad \|D^{2}(1+D^{2}+\lambda)^{-1}\|_{\End^*_B(\mathpzc{E})}\leq 1.
\end{align*}
\end{lemapp}

The integral formula has been used in the Hilbert $C^*$-module context since the work of Baaj-Julg \cite{BJ}. A detailed treatment can be found in \cite[Appendix A, Remark $3$]{cp1}. The estimates can be found in \cite[Appendix A, Remark $5$]{cp1}.

\begin{propapp}
\label{*-alg} 
If $a,b\in\Lip^{\epsilon}(\mathpzc{E},D)$ then $a^{*}, ab\in\Lip^{\epsilon}(\mathpzc{E},D)$. In particular $\Lip^{\epsilon}(\mathpzc{E},D)$ is a $*$-algebra.
\end{propapp}

\begin{proof} 
The statement $a^*\in \Lip^{\epsilon}(\mathpzc{E},D)$ follows directly from the definition. For the product of $a$ and $b$, we write
\[[D,ab](1+D^{2})^{-\frac{1}{2}+\halfepsilon}=a[D,b](1+D^{2})^{-\frac{1}{2}+\halfepsilon}+[D,a]b(1+D^{2})^{-\frac{1}{2}+\halfepsilon},\]
and observe that the first summand admits a bounded extension. For the second summand we use the integral expression
\begin{align*}
[&b,(1+D^{2})^{-\frac{1}{2}+\halfepsilon}]\\
&=\frac{\sin{(\frac{1}{2}-\halfepsilon)\pi}}{\pi}\int_{0}^{\infty}\lambda^{-\frac{1}{2}+\halfepsilon}(1+D^{2}+\lambda)^{-1}([b,D]D+D[D,b])(1+D^{2}+\lambda)^{-1}\rd\lambda.
\end{align*}
Multiplying with $[D,a]$ and estimating the relevant parts of the integral gives
\[\begin{split} \|[D,a]\lambda^{-\frac{1}{2}+\halfepsilon}(1+D^{2}+&\lambda)^{-1}[b,D]D(1+D^{2}+\lambda)^{-1}\| \\
&\leq \frac{C_{a,\halfepsilon}C_{b,\epsilon}}{\lambda^{\frac{1}{2}-\halfepsilon}}\|(1+D^{2}+\lambda)^{-\epsilon}\|\|(1+D^{2}+\lambda)^{-\frac{1}{2}}\|\leq\frac{C_{a,\epsilon}C_{b,\epsilon}}{\lambda^{1+\halfepsilon}},
\end{split}\]
and similarly
\[\|[D,a]\lambda^{-\frac{1}{2}+\delta}(1+D^{2}+\lambda)^{-1}D[D,b](1+D^{2}+\lambda)^{-1}\|\leq\frac{C_{a,\epsilon}C_{b,\epsilon}}{\lambda^{1+\halfepsilon}}.\]
Therefore, the integral converges in norm and $[D,ab](1+D^2)^{-\frac{1}{2}+\delta}$ admits a bounded extension. The proof that $(1+D^2)^{-\frac{1}{2}+\delta}[D,ab]$ admits a bounded extension is carried out analogously.
\end{proof}

We now come to the main result about $\epsilon$-$KK$-cycles, concerning the bounded transform and the relation to $KK$-theory.

\begin{thmapp}[cf. \cite{BJ}]
\label{epsilonFred} 
The bounded transform $(\mathpzc{E}, D(1+D^{2})^{-\frac{1}{2}})$ of an $\epsilon$-$KK$-cycle is an $(A,B)$ Kasparov module and hence defines a class in $KK_{*}(A,B)$.
\end{thmapp}

We note that the proof of this Theorem is carried out analogously to the proof of \cite[Lemma $51$]{grensing}.

\begin{proof} 
The proof of the theorem relies on the integral formula \eqref{magicint} to show that the commutators $[F,a]$ are compact. The properties $a(F-F^{*}), a(1-F^{2})\in \K(\mathpzc{E})$ hold trivially. Recall that $\delta:=\epsilon/2$. 

We have
\[[D(1+D^{2})^{-\frac{1}{2}},a]=[D,a](1+D^{2})^{-\frac{1}{2}}+D[(1+D^{2})^{-\frac{1}{2}},a].\]
The first term is compact because $(1+D^{2})^{-\halfepsilon}$ is compact and $[D,a](1+D^{2})^{-\frac{1}{2}+\halfepsilon}$ is bounded. For the second term, we expand
\begin{equation}
\begin{split}
D[(1+D^{2})^{-\frac{1}{2}},a] =\frac{1}{\pi}&\int_{0}^{\infty}\lambda^{-\frac{1}{2}}D(1+D^{2}+\lambda)^{-1}[a,D]D(1+D^{2}+\lambda)^{-1}\rd \lambda\\ 
\label{magicinttwo}
&+\frac{1}{\pi}\int_{0}^{\infty}\lambda^{-\frac{1}{2}}D^{2}(1+D^{2}+\lambda)^{-1}[a,D](1+D^{2}+\lambda)^{-1}\rd\lambda.
\end{split}
\end{equation}
Using the estimates from Lemma \ref{intes} we find that
\[\|\lambda^{-\frac{1}{2}}D(1+D^{2}+\lambda)^{-1}[a,D]D(1+D^{2}+\lambda)^{-1}\|\leq\frac{C_{a,\epsilon}}{2\lambda^{1+\halfepsilon}},\]
and
\[\|\lambda^{-\frac{1}{2}}D^{2}(1+D^{2}+\lambda)^{-1}[a,D](1+D^{2}+\lambda)^{-1}\|\leq\frac{C_{a,\epsilon}}{\lambda^{1+\halfepsilon}},\]
where $C_{a,\epsilon}:=\|[D,a](1+D^{2})^{-\frac{1}{2}+\halfepsilon}\|$. We conclude that the integral formula \eqref{magicinttwo} converges in norm and the commutators are compact.
\end{proof}

Kucerovsky \cite{Kuc} gives sufficient conditions for a triple of even cycles to represent a Kasparov product. As in the bounded case, to formulate this result,  we need the mappings 
$$T_x\in \End^*_C(\mathpzc{F},\mathpzc{E}\otimes_{B}\mathpzc{F}),\quad T_x:f\mapsto x\otimes_B f,$$ 
defined for $x\in \mathpzc{E}$. The adjoint of the operator $T_{x}$ is given by
$$T_x^{*}\in \End^*_C(\mathpzc{E}\otimes_{B}\mathpzc{F},\mathpzc{F}),\quad T_x^{*}:e\otimes f\mapsto \langle x,e\rangle f,$$

\begin{thmapp}[cf. \cite{Kuc}]
\label{epsilonprod} 
Let $(\mathpzc{E},S)$ be an even $\epsilon$-unbounded $(A,B)-KK$-cycle, $(\mathpzc{F},T)$  an even $\epsilon$-unbounded $(B,C)-KK$-cycle and $(\mathpzc{E}\otimes_{B}\mathpzc{F},D)$ an even $\epsilon$-unbounded $(A,C)-KK$-cycle such that:
\begin{enumerate}\item for all $x$ in a dense subspace of $A\mathpzc{E}$, the operator
\[\left[\begin{pmatrix} D & 0 \\ 0 & T\end{pmatrix},\begin{pmatrix} 0 & T_{x} \\ T^{*}_{x} & 0\end{pmatrix}\right]\begin{pmatrix}(1+D^{2})^{-\frac{1}{2}+\halfepsilon}& 0\\ 0& (1+T^{2})^{-\frac{1}{2}+\halfepsilon}\end{pmatrix},\]
defined on $\Dom D\oplus \Dom T$, extends to an operator in $\End^{*}_{B}(\mathpzc{E}\hotimes_{B}\mathpzc{F})$;
\item $\Dom D\subset\Dom S\otimes 1$;
\item there is $\lambda\in\R$ such that $\langle Dx,S\otimes 1 x\rangle+\langle S\otimes 1 x,Dx\rangle\geq -\lambda \langle x,x\rangle$.
\end{enumerate}
Then $(\mathpzc{E}\otimes_{B}\mathpzc{F},D)$ represents the Kasparov product of $(\mathpzc{E},S)$ and $(\mathpzc{F},T)$.
\end{thmapp}

\begin{proof} 
As in \cite{Kuc}, conditions 2.) and 3.) imply the positivity condition for the bounded transforms. The proof that condition 1.) implies the bounded connection condition is the same as the proof that an $\epsilon$-unbounded $KK$-cycle gives a Fredholm module.
\end{proof}

Sufficient conditions for products in which one of the factors is an odd $\epsilon$-unbounded Fredholm module can be derived by formal Bott periodicity. We refer to the relevant discussions in \cite{BMS, AKH, KaLe}.

The following lemma describes a weakening of the domain preservation condition, and is useful in practice for proving $\epsilon$-boundedness.

\begin{lemapp}
\label{corecomm} 
Suppose $a$ maps a core for $D$ into $\Dom D$ and $[D,a](1+D^{2})^{-\frac{1}{2}+\halfepsilon}$ and  $(1+D^{2})^{-\frac{1}{2}+\halfepsilon}[D,a]$ extend to operators in $\End_{B}^{*}(\mathpzc{E})$. Then the commutator $[(1+D^{2})^{-\frac{1}{2}},a]$ maps $\mathpzc{E}$ into $\Dom D$. Consequently $a$ preserves $\Dom D$ and $[D,a](1+D^{2})^{-\frac{1}{2}+\halfepsilon}$and $(1+D^{2})^{-\frac{1}{2}+\halfepsilon}[D,a]$ are bounded on $\Dom D$.
\end{lemapp}

\begin{proof} 
Denote the core respected by $a$ by $X$. By \eqref{magicint} and \cite[Lemma 2.3]{cp1} and the discussion succeeding it, we can write
\begin{equation}
\begin{split}
[a,(1+D^{2})^{-\frac{1}{2}}]=\frac{1}{\pi}\int_{0}^{\infty}\lambda^{-\frac{1}{2}}&(1+D^{2}+\lambda)^{-1}D[D,a](1+D^{2}+\lambda)^{-1}\rd \lambda \\
\label{intexpansion}  
&+ \frac{1}{\pi}\int_{0}^{\infty}\lambda^{-\frac{1}{2}}(1+D^{2}+\lambda)^{-1}[D,a]D(1+D^{2})^{-1}\rd\lambda,
\end{split}
\end{equation}
as a norm convergent integral on $X$. The integral expression \eqref{magicinttwo} for $D[(1+D^{2})^{-\frac{1}{2}},a]$ converges in norm on $X$. Since $X$ is a core, it is of the form $(1+D^{2})^{-\frac{1}{2}}Y$ for some dense $Y\subset\mathpzc{E}$. For a Cauchy sequence $y_{n}\in Y$, with limit $e\in\mathpzc{E}$, the integrals \eqref{magicint} and \eqref{intexpansion} converge in norm at $y_{n}-y_{m}$. Thus $[(1+D^{2})^{-\frac{1}{2}},a]y_{n}\in \Dom D$ is Cauchy for the graph norm and therefore $[(1+D^{2})^{-\frac{1}{2}},a]e\in \Dom D$. 
From this it follows that for a sequence $y_{n}\rightarrow e$ we have 
\[a(1+D^{2})^{-\frac{1}{2}}y_{n}=[a,(1+D^{2})^{-\frac{1}{2}}]y_{n}+(1+D^{2})^{-\frac{1}{2}}ay_{n},\]
and thus $a(1+D^{2})^{-\frac{1}{2}}e\in\Dom D$.
\end{proof}

\subsubsection*{\bf Acknowledgements}
The authors do first and foremost wish to thank the Mathematisches Forschungsinstitut Oberwolfach (Germany) for their support through the Research in Pairs program in 2013, during which most of this work was carried out. The authors also thank the Leibniz Universit\"{a}t Hannover and GRK1463 for facilitating this collaboration. The second author was supported by the EPSRC grant EP/J006580/2. We are indebted to Adam Rennie for discussing the subtleties surrounding spectral triples with relatively bounded commutators, resulting in the exposition in the appendix. We are grateful to Robin Deeley and Bogdan Nica for carefully reading the manuscript and valuable comments. We would also like to thank Richard Sharp, Georges Skandalis, Aidan Sims, Mike Whittaker and Johan \"Oinert for inspiring discussions.

\end{document}